\documentclass[a4paper,11pt]{amsart}
\usepackage[english]{babel}
\usepackage[utf8]{inputenc}
\usepackage[T1]{fontenc}
\usepackage{csquotes} 
\usepackage{amssymb}
\usepackage{float} 
\usepackage{amsthm}
\usepackage[top=2cm, bottom=2cm, left=2cm, right=2cm,twoside=false]{geometry} 

\usepackage{mathtools} 
\usepackage[usenames,x11names]{xcolor}

\usepackage[shortlabels]{enumitem}
\setlist[1]{leftmargin=*}
\setlist[enumerate,1]{label=(\alph*)}
\setlist[enumerate,2]{label=(\roman*), ref=(\alph{enumi}.\roman*), leftmargin=.5em}
\newlist{parlist}{enumerate}{1}
\setlist[parlist]{leftmargin=0cm, itemindent=2\parindent, label=\textbf{\small{\arabic*.}}, ref=\textbf{\small{(\arabic*)}}, itemsep=0.3em, topsep=0.4em}
\newcounter{foo} 

\newlist{enumthm}{enumerate}{1} 
\setlist[enumthm]{align=left, listparindent=\parindent, parsep=\parskip, font=\textit, 
	leftmargin=0pt, labelwidth=0pt, itemindent=.4em,labelsep=.4em, topsep=.6em, itemsep=.4em
}
\setlist[enumthm,1]{label=(\alph*), ref=(\alph*)}
\newcommand\litem[1]{\item{\emph{#1}.\enspace}}

\usepackage{caption}
\usepackage{subcaption} 
\usepackage{tikz}
\usetikzlibrary{calc}
\usetikzlibrary{intersections}
\tikzset{
	add/.style args={#1 and #2}{
		to path={%
			($(\tikztostart)!-#1!(\tikztotarget)$)--($(\tikztotarget)!-#2!(\tikztostart)$)%
			\tikztonodes},add/.default={.2 and .2}}
}
\usepackage{tikz-cd}
\usepackage{pgfplots}

\usepackage{hyperref} 
\hypersetup{
	linkcolor=DodgerBlue3, 
	citecolor  = teal,
	urlcolor   = DodgerBlue4,
	colorlinks = true, 
	hyperfootnotes =false,
	bookmarksdepth=3,
	linktoc=all
}
\makeatletter
\renewcommand\subsection{\@startsection{subsection}{2}%
	\z@{-.5\linespacing\@plus-.7\linespacing}{.5\linespacing}%
	{\bfseries\itshape}}
\renewcommand\subsubsection{\@startsection{subsubsection}{3}%
	\z@{.5\linespacing\@plus.7\linespacing}{-.5\linespacing}%
	{\normalfont\bfseries}}
\renewcommand\paragraph{\@startsection{paragraph}{4}%
	\z@{.5\linespacing\@plus.7\linespacing}{-.5\linespacing}%
	{\normalfont\bfseries}}
\makeatother

\usepackage{indentfirst}

\makeatletter \renewenvironment{proof}[1][\proofname]{
	\par\pushQED{\qed}\normalfont
	\topsep6\p@\@plus6\p@\relax
	\trivlist\item[\hskip\labelsep\bfseries#1\@addpunct{.}]
	\ignorespaces}{
	\popQED\endtrivlist\@endpefalse} \makeatother

\theoremstyle{plain}
\newtheorem{theo}{Theorem}[section]
\newtheorem*{theo*}{Theorem}

\theoremstyle{definition}
\newtheorem{definition}[theo]{Definition}
\newtheorem{lema}[theo]{Lemma}
\newtheorem{prop}[theo]{Proposition}
\newtheorem{example}[theo]{Example}
\newtheorem{notation}[theo]{Notation}
\newtheorem{remark}[theo]{Remark}
\newtheorem{setting}[theo]{Setting}
\newtheorem{cor}[theo]{Corollary}

\newcommand{\C}{\mathbb{C}} 
\newcommand{\D}{\mathbb{D}}

\newcommand{\Q}{\mathbb{Q}} 
\newcommand{\R}{\mathbb{R}} \newcommand{\RR}{\mathbb{R}}

\newcommand{\Z}{\mathbb{Z}}

\newcommand{\cA}{\mathcal{A}}

\newcommand{\cC}{\mathcal{C}}
\newcommand{\cE}{\mathcal{E}}

\newcommand{\cI}{\mathcal{I}}
\newcommand{\cJ}{\mathcal{J}}
\newcommand{\cL}{\mathcal{L}}
\newcommand{\cM}{\mathcal{M}}

\newcommand{\cO}{\mathcal{O}}
\newcommand{\cP}{\mathcal{P}}
\newcommand{\cQ}{\mathcal{Q}}
\newcommand{\cR}{\mathcal{R}}
\newcommand{\cS}{\mathcal{S}}
\newcommand{\cT}{\mathcal{T}}

\newcommand{\cV}{\mathcal{V}}
\newcommand{\cW}{\mathcal{W}}

\newcommand{\cZ}{\mathcal{Z}}

\renewcommand{\epsilon}{\varepsilon}

\DeclareFontFamily{U}{mathx}{}
\DeclareFontShape{U}{mathx}{m}{n}{<-> mathx10}{}
\DeclareSymbolFont{mathx}{U}{mathx}{m}{n}
\DeclareMathAccent{\widehat}{0}{mathx}{"70}
\DeclareMathAccent{\widecheck}{0}{mathx}{"71}

\renewcommand{\tilde}{\widetilde}
\renewcommand{\check}{\widecheck}
\renewcommand{\hat}{\widehat}
\renewcommand{\bar}{\overline}

\newcommand{\Sk}{\operatorname{Sk}}

\newcommand{\sdot}{\, \cdot \, } 

\renewcommand{\leq}{\leqslant}
\renewcommand{\geq}{\geqslant}

\renewcommand{\to}{\longrightarrow}
\newcommand{\map}{\dashrightarrow}

\newcommand{\codim}{\operatorname{codim}}

\newcommand{\Exc}{\operatorname{Exc}}

\newcommand{\Int}{\operatorname{Int}} 
\newcommand{\redd}{_{\mathrm{red}}} 
\newcommand{\sspan}{\operatorname{span}}

\renewcommand{\Im}{\operatorname{Im}}
\renewcommand{\Re}{\operatorname{Re}}

\newcommand{\Ram}{\operatorname{Ram}}

\newcommand{\Crit}{\operatorname{Crit}}

\newcommand{\std}{\mathrm{std}}
\newcommand{\new}{\mathrm{new}}

\newcommand{\rest}{\vartheta}

\renewcommand{\d}{\partial}
\newcommand{\dout}{\partial_{\textnormal{out}}}

\newcommand{\de}{\coloneqq} 

\newcommand{\vol}{\textnormal{vol}}
\newcommand{\gen}{\textnormal{gen}}
\newcommand{\sm}{\textnormal{sm}}

\usepackage[scr=boondoxo]{mathalfa}
\renewcommand{\ll}{\mathscr{l}}
\newcommand{\qq}{\mathscr{q}}

\newcommand{\CY}{\textnormal{CY}}

	\usepgfplotslibrary{fillbetween}
	\usetikzlibrary{snakes,patterns,calligraphy,arrows}

\begin{document}
\title[Fibrations by Lagrangian tori]{Fibrations by Lagrangian tori \\ for maximal Calabi--Yau degenerations and beyond}
\author{Javier Fern\'andez de Bobadilla}
\address{Javier Fern\'andez de Bobadilla: 
	(1) IKERBASQUE, Basque Foundation for Science, Euskadi Plaza, 5, 48009 Bilbao, Basque Country, Spain; 
	(2) BCAM,  Basque Center for Applied Mathematics, Mazarredo 14, 48009 Bilbao, Basque Country, Spain;  
	(3) Academic Colaborator at UPV/EHU.} 
\email{jbobadilla@bcamath.org}

\thanks{J.F.B.\ was supported by the Basque Government through the BERC 2022-2025 program, by the Ministry of Science and Innovation: BCAM Severo Ochoa accreditation CEX2021-001142-S/MICIN/AEI/10.13039/501100011033, and by the Spanish Ministry of Science, Innovation and Universities, project reference PID2020-114750GB-C33. J.F.B.\ and T.P.\ were supported by the  
semester \enquote{Singularity Theory and Low Dimensional Topology} at R\'enyi Institute, Budapest.
}

\author{Tomasz Pe{\l}ka}	
\address{Tomasz Pe{\l}ka: 
	(1) University of Warsaw, Faculty of Mathematics, Informatics and Mechanics, Banacha 2, 02-097 Warsaw, Poland; 
	(2) Jagiellonian University, Faculty of Mathematics and Computer Science, Łojasiewicza 6, 30-348 Kraków, Poland
}
\email{tomasz.pelka@uj.edu.pl}

\begin{abstract}
	We introduce a general technique to construct Lagrangian torus fibrations in degenerations of K\"ahler manifolds. We show that such torus fibrations naturally occur at the boundary of the \emph{A'Campo space}. 
	This space extends a degeneration over a punctured disc to a hybrid space involving its real oriented blowup and the dual complex of the central fiber; equipped with an exact modification of a given K\"ahler form which becomes fiberwise symplectic at the boundary. Therefore, it allows to move Lagrangian tori from the \enquote{radius zero} fibers at the boundary to the nearby fibers of the original degeneration using the symplectic connection. As a consequence, for \emph{any} maximal Calabi--Yau degeneration we prove that a generic region of each nearby fiber admits a Lagrangian torus fibration with asymptotic properties expected from the (conjectural) Strominger--Yau--Zaslow fibration.
\end{abstract}

\maketitle
\setcounter{tocdepth}{1}
\tableofcontents

\section{Introduction}\label{sec:intro}

The Strominger--Yau--Zaslow conjecture \cite{SYZ} interprets mirror symmetry for Calabi--Yau manifolds in terms of conjectural fibrations by special Lagrangian tori: for a mirror pair, these fibrations are expected to become dual when one takes large volume limit on one side, and large complex structure  limit on the other, see \cite[Conjecture 1.5]{BigBook}. The latter limit gives rise to a maximal Calabi--Yau degeneration. In this setting, explicit conjectures concerning existence and geometric properties of special Lagrangian torus fibrations were formulated by Kontsevich and Soibelman in \cite[Conjectures 1, 2]{KS_Conjecture}. In this article, we prove their asymptotic versions. In order to state our main results, we briefly recall the relevant definitions; for more details see Section \ref{sec:prelim} or e.g.,  \cite{BigBook,Gross_survey,Li_survey}.
\smallskip

A \emph{maximal Calabi--Yau degeneration} is a complex projective morphism $f^\circ:X^\circ\to\D^*$, where $\D^{*}$ is a punctured disc, such that the fibers $X_z\de (f^\circ)^{-1}(z)$ are Calabi--Yau $n$-folds, and the monodromy has Jordan blocks of maximal size $n+1$. Applying the semistable reduction theorem \cite{KKMSD_semistable-reduction} after a finite base change we can extend $f^{\circ}$ to a model $f\colon X\to\D$ such that $X$ is smooth, the central fiber $D=\sum_{i} D_i$ is a reduced snc divisor, and $f|_{X\setminus D}=f^{\circ}$. Since each fiber of $f^{\circ}$ is Calabi--Yau, there exists a holomorphic section $\Omega\in\Gamma(X,K_{X/\D})$ with zeros only at the central fiber. Multiplying $\Omega$ by a holomorphic function on the base $\D^{*}$, we can assume that the divisor of $\Omega$ is $\sum_{i}a_iD_i$, with $\mathrm{min}_{i} a_{i}=0$. We say that a component $D_i$ is {\em essential} if $a_i=0$. The \emph{essential skeleton} $\Delta_{\cS}$ of $f^{\circ}$ is a subcomplex of the dual complex of $D$ spanned by vertices corresponding to essential divisors.

Let $t\de -1/\log|f|$. It is known, cf.\ \cite[\textsection 3]{KS_Conjecture} or Section \ref{sec:weak_MA}, that
\begin{equation}\label{eq:omega_new}
	 \mathrm{vol}^\Omega_{\new}\de t^n(\tfrac{\imath}{2})^n(-1)^{\frac{n(n-1)}{2}}\Omega\wedge\overline{\Omega}
\end{equation} 
 restricts to a volume form on each $X_z$ such that the limit of the volume of $X_z$ is finite as $z\rightarrow 0$. Let $\omega_X$ be any K\"ahler form in $X$. Yau's proof of the Calabi conjecture \cite{Yau_Calabi-conjecture} implies that each $X_z$~admits a Ricci-flat K\"ahler metric $\omega^{\CY}_{X_{z}}$ in the same cohomology class as $\omega_{X}|_{X_{z}}$, so
there exists a family of smooth functions $\phi_z^{\CY}\colon X_z\to\RR$ such that $\omega_{X_z}^{\CY}= \omega_X|_{X_z}+dd^c\phi_z^{\CY}$. We refer to $\omega^{\CY}_{X_{z}}$ as \emph{the Calabi--Yau metric}, and call $\phi_{z}^{\CY}$ the \emph{Calabi--Yau potentials.} Ricci flatness of $\omega^{\CY}_{X_z}$ is equivalent to the existence of a smooth function $c\colon \D^*\to\RR_{>0}$ such that 
\begin{equation}
	\label{eq:ricciflatintro}
	\tfrac{1}{n!}(\omega_{X_z}^{\CY})^n=c(z)\cdot \vol^\Omega_{\new}|_{X_z}.
\end{equation}
The rescaling in formula \eqref{eq:omega_new} ensures that $\lim_{z\rightarrow 0} c(z)$ exists and is  nonzero.
\smallskip

Motivated by Mirror Symmetry, Strominger, Yau and Zaslow \cite{SYZ} predicted that Calabi--Yau $3$-folds should admit $\omega_{X_z}^{\CY}$-Lagrangian fibrations by tori which are special with respect to $\Omega$, cf.\ \cite{Gross_survey}. A Lagrangian submanifold $L$ is \emph{special of phase $\varpi\in\mathbb{S}^1$} if the imaginary part $\Im(\varpi \Omega)|_{L}$ vanishes. It was realized by several authors, see e.g., \cite{Gross-Wilson_K3,KS_Conjecture} that such fibrations are an emergent feature of a maximal Calabi--Yau degeneration, i.e., they should appear on $X_z$ for $|z|$ sufficiently small. More precise versions of the problem (taking into account inevitable singularities of some tori) were formulated e.g., in \cite{Gross_Siebert-program,Joyce_examples,Li_NA-MA}, with strong evidence provided by various constructions in particular contexts, such as \cite{Ruan_quintic-1,Ruan,CM,GTZ,RZ,EM,MazzonSchneider-toric}. We do not attempt to give a complete list here, but rather refer the reader to surveys \cite{Auroux_survey, Chan_survey,Li_survey} or \cite[\textsection 7]{BigBook}. 

This article is mostly inspired by the formulation due to Kontsevich and Soibelman \cite[\textsection 3]{KS_Conjecture}. In short, the main aspects of their prediction are as follows; for details see Conjectures 1 and 2 loc.\ cit.

\begin{parlist}
	\litem{Gromov--Hausdorff limit}\label{item:intro-GH} On each fiber $X_z$, consider the metric induced by the Calabi--Yau form $\omega_{X_z}^{\CY}$, rescaled so that the diameter of $X_{z}$ is independent of $z$. Then the Gromov--Hausdorff limit of $X_{z}$  as $z\rightarrow 0$ is a metric space $\Delta$ containing a smooth affine $n$-manifold $\Delta^{\circ}$ whose complement in $\Delta$ has Hausdorff codimension at most $2$. Moreover, the Riemannian metric on $\Delta^{\circ}$ has a potential satisfying the real Monge-Amp\`ere equation, i.e., it is locally given by a matrix $\operatorname{Hess}(K)$ for some function $K$ of the affine coordinates; such that $\det\operatorname{Hess}(K)$ is constant.
	\litem{Special Lagrangian fibration}\label{item:intro-special} There is $\delta>0$ such that for any $z\in\D_{\delta}^*$, there is a subset  $X^{\sm}_z\subseteq X_z$ and a fibration $X^{\sm}_z\to \Delta^\circ$ by special Lagrangian tori. Moreover, the pairs  $(X_z,X_z\setminus X^{\sm}_z)$ equipped with the rescaled Calabi--Yau metric as in \ref{item:intro-GH} converge in the Gromov--Hausdorff sense to $(\Delta,\Delta\setminus\Delta^\circ)$.
\end{parlist}

It is expected that as $z\rightarrow 0$, the volume of the region $X_z^{\sm}$  asymptotically fills the whole volume of $X_z$, see \cite[\textsection 2.5]{Li_survey}. Here, the volume is computed with respect to $\mathrm{vol}^\Omega_{\new}|_{X_z}$ or equivalently, by Ricci-flatness condition \eqref{eq:ricciflatintro}, with respect to $\frac{1}{n!}(\omega^{\CY}_{X_z})^n$.

Moreover the Gromov--Hausdorff limit $(\Delta,\Delta^{\circ})$ should be tightly related with $(\Delta_{\cS},\Delta_{\cS}^{\circ})$, where $\Delta_{\cS}$ is the essential skeleton of $f^{\circ}$, and $\Delta_{\cS}^{\circ}$ is the complement of its codimension $2$ faces. When $f^{\circ}$ is the Fermat family, in \cite{Li_Fermat} Li constructed special Lagrangian fibrations over codimension $0$ faces of $\Delta_{\cS}$; for a wider class of hypersurfaces this was done in \cite{HJMMC_NAMA}. In general,  Nicaise, Xu and Yu \cite{NXY} constructed a non-archimedean analogue of a torus fibration over $\Delta_{\cS}^{\circ}$. An approach to compare the archimedean and non-archimedean fibrations was proposed by Li in \cite{Li_NA-MA}, see \cite[\textsection 5.6]{Li_survey}.  
\smallskip

In this article we introduce a general technique to produce fibrations by Lagrangian tori. We summarize it in Theorem \ref{theo:general}. In the setting of maximal Calabi--Yau degenerations, this result specializes to Theorem \ref{theo:CY} below, and  with little additional arguments given in Section~\ref{sec:CY}, it implies the subsequent Theorem \ref{theo:CY_af}.  The conjunction of Theorems~\ref{theo:CY} and~\ref{theo:CY_af} provides an asymptotic version of the Kontsevich and Soibelman prediction in the {\em general} case of maximal Calabi-Yau degenerations. 
\smallskip

The base of our Lagrangian torus fibration is an \emph{expanded} version $E_{\cS}$ of $\Delta_{\cS}^{\circ}$, constructed in Section \ref{sec:expanded_construction} as follows.  By \cite{NX_skeleton}, the essential skeleton $\Delta_{\cS}$ is a closed $n$-dimensional  pseudomanifold, see \textsection 4.1.2 loc.\ cit.\ for definition. As a consequence,  $\Delta_\cS^\circ$ is a topological manifold, and any $(n-1)$-dimensional face $F^{n-1}$ of $\Delta_{\cS}$ is the intersection of exactly two $n$-dimensional faces $F^n_0$, $F^n_1$ of $\Delta_{\cS}$. Now $E_{\cS}$ is obtained by replacing any $(n-1)$-dimensional face $F^{n-1}$ by $F^{n-1}\times [0,1]$ and identifying $F^{n-1}\times\{i\}$ with the subset $F^{n-1}\subseteq F^n_i$ for each $i\in \{0,1\}$; see Figures \ref{fig:Fermat_ES} and \ref{fig:K3-E} for an example. Such $E_{\cS}$ is a smooth manifold with boundary and corners, whose interior is homeomorphic to $\Delta_{\cS}^{\circ}$. 

We denote by $\C_{\log}\de [0,\infty)\times \mathbb{S}^{1}$ the real oriented blowup of $\C$ at the origin, and by $\D_{\delta,\log}\subseteq \C_{\log}$ the preimage of $\D_{\delta}$ in $\C_{\log}$. 

\begin{theo}\label{theo:CY}
	Let $f^{\circ}\colon X^{\circ}\to \C^{*}$ be a maximal Calabi--Yau degeneration admitting a semi-stable model. There is a model $f\colon X\to \D_{\delta}$ of $f^{\circ}$, 
	such that the following holds. Let $\omega_X$ be any K\"ahler form on $X$. There exists a family of smooth functions $\phi_q\colon X^{\circ}\to\RR$, smoothly depending on a parameter $q\in [0,1]$, such that the family of $2$-forms $\omega_{q}\de \omega_X+dd^c\phi_q$ satisfies the following properties.
	\begin{enumthm}
		\litem{K\" ahler potential}\label{item:fiberwise_Kahler} For every $z\in \D_{\delta}^{*}$ and every $q\in [0,1]$, the restriction $\omega_{q}|_{X_{z}}$ is K\"ahler. 
		\litem{Lagrangian fibration}\label{item:CY_Lagrangian} There is a smooth family of maps $\ll_{z,q}\colon X_{z,q}^{\sm}\to E_{\cS}$, parametrized by $q\in (0,1]$, where each $X_{z,q}^{\sm}$ is a submanifold of $X_{z}$ of codimension zero, with boundary and corners, each $\ll_{z,q}$ is a Lagrangian torus fibration with respect to the form $\omega_{q}$, and $E_{\cS}$ is the expanded version of $\Delta_{\cS}^{\circ}$, constructed in Section \ref{sec:expanded_construction}. 
		\litem{Gromov--Hausdorff limit}\label{item:CY_GH} Choose a smooth path $(z,q)\colon [0,\delta)\to \D_{\delta,\log} \times [0,1]$ such that $|z(0)|=q(0)=0$ and $|z(h)|,q(h)> 0$ for $h>0$, see Figure \ref{fig:any_path}. For each $h>0$ let $X_{h}$ be the fiber $X_{z(h)}$ equipped with a K\"ahler metric given by $\omega_{q(h)}$, rescaled so that its diameter is independent of $h$. Let $X_{h}^{\sm}\de X^{\sm}_{z(h),q(h)}$ be the region introduced in \ref{item:CY_Lagrangian}. Then the Gromov-Hausdorff limit $\lim_{h\rightarrow 0}(X_{h},X_{h}\setminus X^{\sm}_{h})$ is the essential skeleton $(\Delta_{\cS},\Delta_\cS\setminus\Delta_\cS^\circ)$, equipped with a positive multiple of the euclidean metric.
	\end{enumthm}
\end{theo}
The euclidean metric on $\Delta_{\cS}$ is restricted from the dual complex of $D$, viewed as a subset of some euclidean space $\R^{N}$. Clearly, it has a potential satisfying the real Monge-Amp\`ere equation for the natural affine structure of the maximal faces, as required by \cite[Conjecture 1(c)]{KS_Conjecture}, see \ref{item:intro-GH} above.
\smallskip

The base $E_{\cS}$ in Theorem \ref{theo:CY}\ref{item:CY_Lagrangian} 
contains, as an open subset, the union of relative interiors of $n$-dimensional faces $\Delta_{\cS}$. We denote this open subset by $\Delta_{\cS}^{\gen}$. Let $X^{\gen}_{z,q}\subseteq X_{z,q}^{\sm}$ be the preimage of $\Delta_{\cS}^{\gen}$ through the Lagrangian torus fibration $\ll_{z,q}$ from Theorem \ref{theo:CY}\ref{item:CY_Lagrangian}. We refer to $X^{\gen}_{z,q}$ as the \emph{generic region} of the fiber $X_z$. The next result asserts that the generic region is large, and the family $\omega_{q}$ approximates there the conjectural behavior of the Ricci-flat forms.

\begin{theo}\label{theo:CY_af}
	We keep the notation and assumptions from Theorem \ref{theo:CY}. Let $\Omega\in \Gamma(X,K_{X/\D})$ be a fiberwise holomorphic volume form, chosen so that it does not vanish along $X_0$, and let $\vol^{\Omega}_{\new}$ be the induced fiberwise volume form \eqref{eq:omega_new}. Fix a smooth path as in Theorem \ref{theo:CY}\ref{item:CY_GH}. For $h>0$ let $\ll_{h}\colon X_{h}^{\sm}\to E_{\cS}$ be the Lagrangian fibration introduced in Theorem \ref{theo:CY}\ref{item:CY_Lagrangian}, and let $X_{h}^{\gen}\de \ll_{h}^{-1}(\Delta_{\cS}^{\gen})$ be the preimage of the relative interiors of the maximal faces. 
	Then as $h\rightarrow 0$, we have the following.
		\begin{enumthm}
			\litem{Volume filling}\label{item:CY_Lagrangian_volume} The volume of $X_{h}^{\gen}$ with respect to $\mathrm{vol}^\Omega_{\new}$ converges to the total volume of $X_h$.
			\litem{Asymptotic Ricci-flatness} \label{item:assflat} Let $c_{h}\colon X_h^{\gen}\to \R_{>0}$ be a smooth family of functions satisfying $\frac{1}{n!} (\omega_{q(h)})^{n}=c_{h}\cdot \vol^{\Omega}_{\new}$. Then $c_{h}$ converges to a positive constant.
			\litem{Lagrangian tori are asymptotically special}\label{item:CY_asymptotic specialty} Assume that the path $\gamma$ is  tame, see Definition \ref{def:special_path}: this condition holds e.g., if $|z(h)|=h$ for all $h\geq 0$, see Figure \ref{fig:special_path}. Then there is a phase $\varpi\in \mathbb{S}^1$ such that for every $b\in \Delta_{\cS}^{\gen}$, the fibers $L_{h}\de \ll_{h}^{-1}(b)\subseteq X_{h}^{\gen}$ of the Lagrangian torus fibration $\ll_{h}$ from Theorem \ref{theo:CY}\ref{item:CY_Lagrangian} 
			satisfy
			\begin{equation*}
				\lim_{h\rightarrow 0}\  \frac{\Im (\varpi \Omega)|_{L_{h}}}{\Re (\varpi \Omega)|_{L_{h}}}=0.
			\end{equation*}
			In particular, the Lagrangian tori $L_h\subseteq X_{h}^{\gen}$ asymptotically minimize volume in their homology classes, see Corollary \ref{cor:volume_minimizing} for a precise statement.
			\litem{Maslov class} \label{item:CY_Maslov} All the fibers of the Lagrangian torus fibration $\ll_{h}$ have trivial Maslov class.
		\end{enumthm}	
\end{theo}
	The \emph{Maslov class} of a Lagrangian torus $L$ is the element of $H^{1}(L,\Z)$ induced by the \emph{phase map} $\varpi\colon L\to \mathbb{S}^1$, defined so that $\varpi\cdot \Omega|_{L}$ is a volume form on $L$, see \cite[\textsection 3.6.1.2]{BigBook}. Triviality of the Maslov class is important to define the grading in Lagrangian Floer homology, see \cite[\textsection 8.3.3]{BigBook}.
	
	The tameness assumption in Theorem \ref{theo:CY_af}\ref{item:CY_asymptotic specialty} is a mild technical condition implying that the path is not very tangent to the $q$-axis, see Figure \ref{fig:special_path}.
\begin{figure}[ht]
	\begin{subfigure}{0.47\textwidth}
		\centering
		\begin{tikzpicture}[scale=0.5]
			\path[use as bounding box] (-1.6,0) rectangle (8,3);
			\draw[->] (0,0) -- (0,3);
			\node[below left] at (0,3) {\small{$q(h)$}};
			\draw[->] (0,0) -- (7,0);
			\node[above] at (7,0) {\small{$|z(h)|$}};
			\draw[thick] (0,0) to[out=90,in=180] (2,1.5) to[out=0,in=180] (3.5,1.2) to[out=0,in=-120] (5,2.5);
		\end{tikzpicture}
		\caption{Path in Theorems \ref{theo:CY}\ref{item:CY_GH} and \ref{theo:CY_af}\ref{item:CY_Lagrangian_volume},\ref{item:assflat}.}
		\label{fig:any_path}
	\end{subfigure}
	\begin{subfigure}{0.47\textwidth}
		\centering
		\begin{tikzpicture}[scale=0.5]
			\path[use as bounding box] (-1.6,0) rectangle (8,3);
			\draw[->] (0,0) -- (0,3);
			\node[below left] at (0,3) {\small{$q(h)$}};
			\draw[->] (0,0) -- (7,0);
			\node[above] at (7,0) {\small{$|z(h)|$}};
			\draw[thick] (0,0) to[out=60,in=180] (2,1) to[out=0,in=-120] (5,2.5);;
		\end{tikzpicture}
		\caption{Tame path in Theorem \ref{theo:CY_af}\ref{item:CY_asymptotic specialty}.}
		\label{fig:special_path}
	\end{subfigure}
	\caption{Paths in Theorems \ref{theo:CY}\ref{item:CY_GH} and \ref{theo:CY_af}.}
	\label{fig:paths}
\end{figure}

We now have three distinguished cohomologous K\"ahler forms on the fiber $X_h$: the Ricci-flat form $\omega^{\CY}_{X_{h}}$, our form $\omega_{q(h)}$, and the form $\omega_{X}$ restricted from a chosen model, e.g., the Fubini--Study form restricted from some ambient projective space. Our form $\omega_{q(h)}$ can be regarded as an intermediate one between $\omega_{X}$ and $\omega^{\CY}_{X_{h}}$. In fact, formula \eqref{eq:omega-s} defines $\omega_{q(h)}$ as an arbitrarily small deformation of $\omega_{X}$, so for any $\epsilon\in (0,1]$, the form $(1-\epsilon)\cdot \omega_{X}+\epsilon\cdot \omega_{q(h)}$ satisfies the assertions of Theorems \ref{theo:CY} and \ref{theo:CY_af}. In particular, the Gromov--Hausdorff limit of the corresponding rescaled metrics has the properties conjectured for $\omega^{\CY}_{X_{h}}$, sharply different from the ones of $\omega_{X}$ (for which the Gromov--Hausdorff limit is simply the central fiber $X_0$). Moreover, as $\epsilon$ increases, the deformation $(1-\epsilon)\cdot \omega_{X}+\epsilon\cdot \omega_{q(h)}$ redistributes the volume from the non-generic region, where $\omega_{X}$ dominates, to the generic region, where $\omega_{X_h}^{\CY}$ concentrates the whole volume. Nonetheless, the volume of the non-generic region with respect to $\omega_{q(h)}$ does not converge to zero, see  Remark \ref{rem:MA_fails}. Hence from the viewpoint of the Gromov--Hausdorff limit our forms $\omega_{q(h)}$ resemble $\omega_{X_h}^{\CY}$, but the associated volume forms $\omega_{q(h)}^{n}$ do not approximate $(\omega_{X_h}^{\CY})^{n}$.

\smallskip

The main novelty of this article is a general technique that allows to construct fibrations by Lagrangian tori in K\"ahler degenerations, where fibers do not need to be Calabi--Yau nor compact. Theorems \ref{theo:CY} and \ref{theo:CY_af} above are obtained by combining this technique with special geometric properties of maximal Calabi--Yau degenerations. Now, we briefly describe the construction.
\smallskip

Let $f\colon X\to \C$ be a holomorphic map whose fibers have dimension $n$, and the unique singular fiber $D\de f^{-1}(0)$ is simple normal crossings, with irreducible components $D_1,\dots, D_{N}$. Let $\omega_{X}$ be a Kähler form on $X$. Choose a set of indices $\cS\in \{1,...,N\}$ corresponding to components of $D$ which we call \emph{essential}. In the maximal Calabi--Yau degeneration case we take $\cS$ to be the set of indices of the components of minimal discrepancy, but the construction works for any choice of $\cS$. Let $\Delta_D$ be the dual complex of $D$ and let $\Delta_{\cS}$ be the essential skeleton, defined as the subcomplex of $\Delta_{D}$ spanned by vertices corresponding to essential components.
\smallskip

The key tool in our construction is the A'Campo space $A$ associated to $f$, which fits into a diagram 
\begin{equation}\label{eq:AX-diagram-intro}
	\begin{tikzcd}
		A
		\ar[r]
		\ar[d, "f_{A}"]
		\ar[rd, "\pi"]
		& X 
		\ar[d, "f"]  
		\\
		\C_{\log}
		\ar[r] & \C
	\end{tikzcd}
\end{equation}
where $\C_{\log}=[0,\infty)\times \mathbb{S}^1$ is the real oriented blowup of $\C$ at the origin. The space $A$ is a \enquote{hybrid} construction: it compactifies $f^{\circ}\colon X^\circ\to \C^*$ with the \enquote{radius-zero} fibration combining the Kato--Nakayama space \cite{Kato_Nakayama} of the log structure $(X,D)$ with the hybrid space \cite{BJ_measures} for the model $f$, or, in other words, combining polar coordinates of the real oriented blowup of $X$ at $D$ with the tropical coordinates of the dual complex $\Delta_{D}$ of the snc divisor $D$. By construction, any \enquote{radius-zero} fiber $A_{\theta}$ for $\theta\in\partial \C_{\log}$ splits naturally as $A_{\theta}=\bigsqcup_{I}A_{I,\theta}^\circ$, where the disjoint union is indexed by all faces $\Delta_{I}$ of $\Delta_{D}$. Each of the pieces $A_{I,\theta}^\circ$ admits a concrete geometric description~(see Lemma~\ref{lem:product}), in particular:
\begin{enumerate}[(i)]
 \item\label{item:intromaximal} if $\dim \Delta_I=n$ (so  $\Delta_{I}$ is a \emph{maximal face}) then we have  $A_{I,\theta}^\circ=\Delta_I\times (\mathbb{S}^1)^n$,
 \item\label{item:introsubmaximal} if $\dim \Delta_I=n-1$ then we have $A_{I,\theta}^\circ=\Delta_I\times (X_{I}^{\circ})_{\log,\theta}$, where $(X_{I}^{\circ})_{\log,\theta}\to X_{I}^{\circ}$ is an $(\mathbb{S}^1)^{n-1}$-bundle over a compact Riemann surface with as many punctures as maximal faces contain $\Delta_I$. 
\end{enumerate}
As a topological fibration, $f_{A}\colon A\to \C_{\log}$ was constructed by A'Campo \cite{A'Campo}. In \cite{FdBP_Zariski} we have upgraded this construction so that the maps in diagram \eqref{eq:AX-diagram-intro} are smooth, and $A$ is endowed with an exact modification $\omega_A$ of $\pi^{*}\omega_{X}$, symplectic on each fiber of $f_A$. This was used to compare the usual symplectic monodromy with the convenient radius-zero model via the symplectic connection. In Section \ref{sec:A} we recall the construction of $A$ from \cite{FdBP_Zariski}, and slightly modify it so that it admits, for any $\epsilon>0$ small enough, a family of fiberwise symplectic forms
\begin{equation}\label{eq:omega-s-intro}
	\omega_{q}^{\epsilon}=\pi^{*}\omega_{X}+\epsilon \cdot ( q\cdot \omega^{\sharp}+(1-q)\cdot \omega^{\flat}),\quad q\in [0,1],
\end{equation}
which at positive radius are fiberwise equal to $\omega_{q}^{\epsilon}=\pi^{*}\omega_{X}+dd^{c}\phi_{q}^{\epsilon}$ for some potential $\phi_{q}^{\epsilon}$, so they are fiberwise K\"ahler, as claimed in Theorem \ref{theo:CY}\ref{item:fiberwise_Kahler}. This additional property (compatibility with the complex structure) was not proved for $\omega_{A}$ in \cite{FdBP_Zariski}.

The proof of the remaining parts of Theorems \ref{theo:CY}, \ref{theo:CY_af} follows the general philosophy of~\cite{FdBP_Zariski}: we perform the main constructions at radius zero, where the structure of $A$ and the symplectic form is very well known and special, 
and then pull them to positive radius by the symplectic connection. We now highlight the main ideas and point to the sections of the paper where they are implemented.
\smallskip

For any maximal face $\Delta_I$ of $\Delta_{\cS}$, close to the interior of $A_{I,\theta}^\circ\subseteq \d A$ the form $\omega_{0}^{\epsilon}$ is non-degenerate, see Proposition \ref{prop:omega}\ref{item:symplectic_zero}, and as can be seen from the definition \eqref{eq:omega_flat} of $\omega^{\flat}$, the induced metric  $\omega_{0}^{\epsilon}$ resembles asymptotically the semi-flat one described in \cite{SYZ}, cf.\ \cite[p.\ 4]{Li_survey}. This is the reason for asymptotic Ricci-flatness claimed in Theorem~\ref{theo:CY_af}\ref{item:assflat}. We prove it in Proposition \ref{prop:MA_weak}. 

Moreover, tori arising in the product structure~\ref{item:intromaximal} become Lagrangian with respect to  $\omega_q^\epsilon$ for any $q\in [0,1]$, see Proposition \ref{prop:Lagranian-at-radius-zero}. Pulling them to positive radius by the symplectic connection of $\omega_q^\epsilon$ yields a Lagrangian torus fibration as in Theorem \ref{theo:CY}\ref{item:CY_Lagrangian}, see Proposition~\ref{prop:Lagranian-at-positive-radius}. In the maximal Calabi--Yau degeneration case, as $q$ approaches $0$ faster than the radius (see Figure \ref{fig:special_path}) we prove in Proposition~\ref{prop:special} that the Lagrangian tori become asymptotically special, as claimed in Theorem~\ref{theo:CY_af}\ref{item:CY_asymptotic specialty}. As a consequence, they have trivial Maslov class, as claimed in Theorem \ref{theo:CY_af}\ref{item:CY_Maslov}.

For any face $\Delta_I$ of dimension $n-1$, see case \ref{item:introsubmaximal} above, we take a proper Morse function $h_I\colon X_{I}^{\circ}\to (0,\infty)$ and prove that for each $c\in (0,\infty)$, the preimage of $h_{I}^{-1}(c)$ in $(X_{I}^{\circ})_{\log}$ is isotropic for $\omega_{q}^{\epsilon}$ for any $q\in [0,1]$. Thus we get a Lagrangian torus fibration with singularities, whose discriminant locus is the union of simplices $\Delta_{I}\times \Crit(h_{I})$, where $\Crit(h_{I})$ is the set of critical points of $h_{I}$.

Define $A_\theta^{\sm}=\bigsqcup_{I}A_{I,\theta}^\circ$, where the union is indexed by all faces $\Delta_{I}$ of $\Delta_{\cS}$ of dimension $n$ and $n-1$. An appropriate choice of Morse functions $h_I$ guarantees that the Lagrangian fibrations defined above in each of the pieces
$A_{I,\theta}^\circ$ glue to a Lagrangian fibration 
$\ll\colon A_\theta^{\sm}\to E_{\cS}$ 
with singularities. Here $E_{\cS}$ is the expanded skeleton, endowed with a structure of an \emph{ivy-like manifold}, see Section \ref{sec:ivy}: it is a smooth manifold with some singularities corresponding to critical points of $h_{I}$. 

In the maximal Calabi--Yau degeneration case we can take a model with  $X_{I}^{\circ}\cong \C^{*}$, see Proposition \ref{prop:model}\ref{item:sk_P1}, so we have a canonical choice of Morse functions $h_{I}$ with no critical points; namely the absolute value of a complex coordinate. With this choice, the map $\ll$ becomes a Lagrangian fibration without singularities, over a smooth manifold $E_{\cS}$, which agrees with the one described at the beginning of the introduction. Pulling it back by the symplectic connection associated with $\omega_{q}^{\epsilon}$ for every $q\in (0,1]$, one gets Theorem~\ref{theo:CY}\ref{item:CY_Lagrangian}, see Proposition \ref{prop:Lagranian-at-positive-radius}. Moreover, just like in the semi-flat model, after a suitable rescaling we see that the above Lagrangian tori collapse as $z\rightarrow 0$, both with respect to the metric, and to the Calabi--Yau volume. This implies the Gromov--Hausdorff convergence and volume filling, claimed in Theorem \ref{theo:CY}\ref{item:CY_GH} and \ref{theo:CY_af}\ref{item:CY_Lagrangian_volume}. We prove these results in Propositions \ref{prop:GH} and \ref{prop:volume}, respectively.

The main results of the article (Theorems \ref{theo:CY}, \ref{theo:CY_af} and \ref{theo:general}) are formulated in an \enquote{existential} way, so their conclusions hold regardless of the choices made within our constructions. The relevance of these choices is emphasized throughout the article, see in particular Remarks \ref{rem:new_set}, \ref{rem:canonical_choice}, and \ref{rem:canonical_choice-positive-radius}. In the short Section~\ref{sec:choices} at the end of the paper a synoptic explanation of all choices is provided.

\paragraph{Acknowledgments} This work grew during visits of the first author to IMPAN, Warsaw, the second author to BCAM, Bilbao, and both authors to R\'enyi Institute, Budapest and CIRM, Marseille. We are grateful to each of these institutions for hospitality and excellent working conditions. We would also like to thank Duco van Straten and Valentino Tosatti  for inspiring conversations, and the referee for valuable comments which improved the quality of the article.

\section{Preliminaries on the essential skeleton}\label{sec:prelim}

In this section, we set up the notation for the remaining part of the article, and, in case of maximal Calabi--Yau degenerations, choose an appropriate model for which our Lagrangian tori have no singularities. The latter is achieved in Proposition \ref{prop:model} by applying the Minimal Model Program.
\smallskip

We work in the following setting. 
Let $(X,\omega_{X})$ be a K\"ahler manifold of dimension $n+1$, and let 
\begin{equation*}
	f\colon X\to \C	
\end{equation*}
be a holomorphic function with unique critical value $0\in \C$, such that the fiber $D\de f^{-1}(0)$ is snc. Let $D_{1},\dots, D_{N}$ be the irreducible components of $D$, and write
\begin{equation}\label{eq:multiplicites}
	D=\sum_{i=1}^{N}m_{i}D_{i}.
\end{equation}
Moreover, we choose a subset $\cS\subseteq \{1,\dots, N\}$. In practice, we either take $\cS$ to be $\{1,\dots, N\}$, or, in the setting of Theorem \ref{theo:CY}, the set of indices of essential components of $D$, see formula  \eqref{eq:essential}.
\smallskip

For $z\in \C^{*}$ we write  $X_z=f^{-1}(z)$. Since we are interested only in fibers $X_{z}$ for $|z|$ small enough, we often replace $X$ by the preimage $f^{-1}(\D_{\delta})$ for some $\delta>0$ such that $f|_{X\setminus D}\colon X\setminus D\to \D_{\delta}^{*}$ is a submersion; and keep shrinking $\delta>0$ whenever needed.

Throughout the article, the adjective \emph{fiberwise} always refers to the fibers of $f|_{X\setminus D}$, or its extension $f_{A}$ to the A'Campo space, introduced in Section \ref{sec:A}.  We write $\imath=\sqrt{-1}$.

\subsection{Dual complex of an snc divisor}\label{sec:dual_complex}

Let $D=\sum_{i=1}^{N}m_i D_i$ be an snc divisor as above. For $I\subseteq \{1,\dots, N\}$ we put $X_{I}=\bigcap_{i\in I}D_{i}$, $X_{I}^{\circ}=X_{I}\setminus \bigcup_{j\not\in I}D_{j}$. This way, $X_{I}=\bar{X}_{I}^{\circ}$ and $X=\bigsqcup_{I} X_{I}^{\circ}$ is a stratification of $X$. 

Let $\Delta=\{(w_{1},\dots,w_{N})\in \R^{N}: \sum_{i}w_{i}=1\}$ be the standard $(N-1)$-dimensional simplex. For a subset $I\subseteq \{1,\dots,N\}$ we put $\Delta_{I}=\{(w_1,\dots,w_{N})\in \Delta: w_{i}=0 \mbox{ for } i\neq I\}$. The \emph{dual complex} of $D$ is $\Delta_{D}\de\bigcup_{\{I: X_{I}\neq \emptyset\}} \Delta_{I}\subseteq \Delta$. We equip it with a metric induced by the standard Riemannian metric  $\sum_{i}(dw_{i})^{2}$ on $\R^{N}$. For a given subset $\cS\subseteq \{1,\dots, N\}$, we put $\Delta_{\cS}=\Delta_{D}\cap (\bigcup_{I\subseteq \cS}\Delta_{I})$.

Since the divisor $D$ is snc, for every face $\Delta_{I}$ of $\Delta_{D}$ we have $\#I=\codim_{X} X_{I}=n+1-\dim X_{I}\leq n+1$. We say that the face $\Delta_{I}$ is \emph{maximal} if $\#I=n+1$ (so $X_{I}$ is a point), and \emph{submaximal} if $\#I=n$ (so $X_{I}$ is a curve). We are mostly interested in the case when $\Delta_{D}$ has at least one maximal face.
\smallskip

Each face $\Delta_{I}$, viewed as a subset of the affine space $\{\sum_{i}w_{i}=1, w_{i}=0 \mbox{ for } i\not\in I\}\subseteq \R^{N}$ is a manifold with boundary and corners, see \cite[Appendice]{corners} for precise definitions. By Proposition 3.1 loc.\ cit, we can and do think of a manifold with boundary and corners as a subset of a usual smooth manifold, mapped by each local chart to an open subset of $\R^{p}\times \R^{m-p}_{\geq 0}$ for some $0\leq p \leq m$. We use the notions of smooth functions, tangent space etc.\ inherited from the ambient smooth manifold.

\subsection{Maximal Calabi--Yau degenerations and their models}\label{sec:CY-intro}

A smooth projective complex manifold $Y$ is \emph{Calabi--Yau} if its canonical divisor $K_{Y}$ is trivial, i.e., there is a nowhere vanishing holomorphic $n$-form $\Omega\in \Omega^{n,0}(Y)$, for $n=\dim Y$.  Yau's solution to the Calabi conjecture \cite{Yau_Calabi-conjecture} asserts that any K\"ahler class in $H^{1,1}(Y)\cap H^{2}(Y,\R)$ admits a unique representative whose associated metric is Ricci flat, see \cite[\textsection 5]{Joyce_yellow}. We refer to this representative as \emph{the Calabi--Yau form}, and denote it by $\omega_{\CY}$. Ricci-flatness of the metric $g_{\CY}\de \omega_{\CY}(\sdot,J\sdot)$ is equivalent to the fact that there exists a holomorphic volume form $\Omega$  satisfying the \emph{complex Monge-Amp\`ere equation}
\begin{equation}\label{eq:complex_MA}
		\tfrac{1}{n!}\cdot \omega_{\CY}^{n}=\vol^{\Omega},\quad\mbox{where}\quad \vol^{\Omega}\de \left(\tfrac{\imath}{2}\right)^{n}\cdot (-1)^{\frac{1}{2}n(n-1)}\cdot  \Omega \wedge \bar{\Omega},
\end{equation} 
cf.\ \cite[formula 11]{Joyce_yellow} or \cite[Proposition 6.25]{BigBook}.
\smallskip

Let $f^{\circ}\colon X^{\circ}\to \C^{*}$ be a projective morphism whose fibers $X_{z}\de (f^{\circ})^{-1}(z)$ are Calabi--Yau manifolds of dimension $n$. We say that $f$ is a \emph{Calabi--Yau degeneration} if $f^{\circ}$ admits an extension to a projective  \emph{model} $f\colon X\to \C$. It is \emph{maximal} if the monodromy action on $H^n(X_z,\Q)$ has a Jordan block of maximal size $n+1$, 
cf.\ \cite[\textsection 3.1]{KS_Conjecture}. 
\smallskip

Resolving singularities, we see that every $f^{\circ}$ admits a model $f\colon X\to \C$ which is \emph{snc}, i.e., such that $X$ is a smooth K\"ahler manifold, and the central fiber $f^{-1}(0)$ is an snc divisor. Moreover, every $f^{\circ}$ admits, after a finite base change, an snc model which is \emph{semistable}, i.e., its central fiber is reduced \cite[Chapter II]{KKMSD_semistable-reduction}. We note that after a finite base change, a maximal Calabi--Yau degeneration is still maximal. Indeed, the general fiber remains the same, while the monodromy transformation is replaced by its iterate, so its Jordan blocks remain of the same size.
\smallskip

We now introduce the \emph{essential skeleton} $\Delta_{\cS}$ of $f^{\circ}$, which agrees with the one used in \cite{MN_weight-functions,NX_skeleton}, see also \cite[\textsection 2.1]{NXY}. Let $f\colon X\to \C$ be an snc model of $f^{\circ}$. We write its central fiber as $D=\sum_{i=1}^{N} m_{i}D_{i}$, see formula \eqref{eq:multiplicites}, and use notation from Section \ref{sec:dual_complex}. By adjunction, for every $z\neq 0$ the smooth fiber $X_{z}$ satisfies $K_{X}|_{X_{z}}=K_{X_{z}}=0$, so the canonical bundle of $X$ admits a meromorphic section $\Theta$ with no zeros or poles off $D$. We choose one such $\Theta$, and let $\nu_{i}\in \Z$ be an integer such that $\Theta$ has a zero of order $\nu_{i}-1$ along $D_{i}$. Thus we have a linear equivalence
\begin{equation}\label{eq:discrepancies}
	K_{X}+D\redd\sim \sum_{i=1}^{N}\nu_{i}D_{i}.
\end{equation}
Note that the coefficients $\nu_{i}$ depend on the choice of the section $\Theta$. If we choose another section, say $\Theta'$, then $\Theta'=uf^{k}\cdot \Theta$ for some holomorphic unit $u\in \cO_{X}^{*}$ and integer $k\in \Z$, so the corresponding coefficients are $\nu_{i}'=\nu_{i}+km_{i}$. It follows that the set
\begin{equation}\label{eq:essential}
	\cS=\left\{ i:\frac{\nu_i}{m_i}=\min_{j}\frac{\nu_j}{m_j} \right\}
\end{equation}
does not depend on $\Theta$. We say that a component $D_{i}$ of $D$ is \emph{essential} if $i\in \cS$; a face $\Delta_{I}$ of the dual complex $\Delta_{D}$ is \emph{essential} if $I\subseteq \cS$. The \emph{essential skeleton} $\Delta_{\cS}\subseteq \Delta_{D}$ is the union of all essential faces.

In \cite[\textsection 4.7]{MN_weight-functions}, the essential skeleton $\Delta_{\cS}$ (or rather its homeomorphic image in the Berkovich space) is called the \emph{Kontsevich--Soibelman skeleton} $\Sk(X^{\circ},\Omega)$, where 
\begin{equation}\label{eq:Omega}
	\Omega\de \frac{\Theta}{d \log f}=\frac{f\cdot\Theta}{df}
\end{equation}
is a chosen section of the vertical canonical bundle. As explained in \textsection 4.7.3 loc.\ cit, the quotient $\frac{\nu_i}{m_i}$ equals $\textnormal{wt}_{\Omega}(x)+1$, where $x$ is the divisorial valuation of $\C(X^{\circ})$ corresponding to $D_i$, and $\textnormal{wt}_{\Omega}$ is the \emph{weight function} introduced in \textsection 4.3.5 loc.\ cit. The essential skeleton is defined in 4.10 loc.\ cit.\ as $\Sk(X^{\circ})\de \bigcup_{\Omega} \Sk(X^{\circ},\Omega)$, so since $K_{X}$ is Calabi--Yau, this definition agrees with $\Sk(X^{\circ},\Omega)$, see \textsection 4.10.2 loc.\ cit. Hence up to a homeomorphism, our essential skeleton $\Delta_{\cS}$ agrees with the one in loc.\ cit. 

Theorem 4.7.5 and Corollary 3.2.4 loc.\ cit.\ show that the homeomorphism type of $\Delta_{\cS}$, and its piecewise affine structure inherited from $\Delta_{D}$ (see \textsection 3.2 loc.\ cit), do not depend on the snc model $X$.  By \cite[Theorem 4.1.10]{NX_skeleton}, the Calabi--Yau degeneration $f^{\circ}$ is maximal if and only if $\dim \Delta_{\cS}=n$. 

We remark that by \cite[Proposition 4.3.4]{MN_weight-functions} the property of being essential depends only on the corresponding valuation of $\C(X^{\circ})$, and not on a particular snc model. More precisely, if a valuation of $\C(X^{\circ})$ corresponds to a component $E$ (resp.\ $E'$) of the central fiber of some snc model $f$ (resp.\ $f'$), then $E$ is essential if and only if $E'$ is.

\smallskip

The following proposition is known to experts, cf.\ e.g., \cite[p.\ 969]{NXY}. For the reader's convenience, we include its proof.
\begin{prop}\label{prop:model}
	Let $f^{\circ}\colon X^{\circ}\to \C^{*}$ be a Calabi--Yau degeneration admitting a semistable snc model. Then $f^{\circ}$ admits a (not necessarily semistable) snc model $f\colon X\to \C$ with the following properties.
	\begin{enumerate}
		\item \label{item:sk_reduced} Every essential component $D_{i}$ of the central fiber $D=f^{-1}(0)$ has multiplicity one in $D$.
		\item\label{item:sk_nu-i} The canonical bundle $K_{X}$ admits a section $\Theta$ such that, denoting as above by $\nu_{i}-1$ the order of $\Theta$ along a component $D_i$ of $D$, we have $\nu_{i}\geq 0$, with equality if and only if $D_{i}$ is essential.
	\item \label{item:sk_P1_or_ell} Let $\Delta_{I}$ be a submaximal, essential face of the dual complex $\Delta_{D}$. Then every maximal face of $\Delta_{D}$ containing $\Delta_{I}$ is essential, too, and the following hold.
	\begin{enumerate}
			\item\label{item:sk_P1} If $f^{\circ}$ is a maximal degeneration then $(X_{I},X_{I}^{\circ})\cong (\mathbb{P}^1,\C^{*})$.
			\item\label{item:sk_ell} If $f^{\circ}$ is not a maximal degeneration then $X_{I}=X_{I}^{\circ}$ is an elliptic curve.
	\end{enumerate} 
	\end{enumerate}
\end{prop}
\begin{proof}
	Let $\tilde{f}\colon \tilde{X}\to \C$ be a semi-stable snc model of $f^{\circ}$, with central fiber $\tilde{D}$. By \cite[Theorem 1.1]{Fujino_semi-stable}, there is a birational map of normal, $\Q$-factorial varieties $\tilde{X}\map \bar{X}$ over $\C$, such that $K_{\bar{X}}\sim 0$ and the central fiber $\bar{D}$ of the induced morphism $\bar{X}\to \C$ is reduced. Moreover, by construction, $(\bar{X},\bar{D})$ is dlt, see the proof of loc.\ cit. Note that $(\bar{X},\bar{D})$ is a good dlt model of $f^{\circ}$, see \cite[Theorem 2.3.6(1)]{NX_skeleton}.	

Since $K_{\bar{X}}$ and $\bar{D}$ are principal Cartier divisors, we get that  
$K_{\bar{X}}+\bar{D}$ is Cartier. Since 
$\bar{X}$ is $\Q$-factorial, 
all components of $\bar{D}$ are $\Q$-Cartier. Hence \cite[Theorem 4.5]{NXY} implies that 
\begin{equation}\label{eq:NXY}
	\mbox{each one-dimensional stratum of }\bar{D}\mbox{ is contained in the snc locus of }\bar{X},
\end{equation}
where we consider the standard stratification of $\bar{D}$ by intersections of components, see \textsection 1.10 loc.\ cit.

Let $\pi\colon (X,D)\to (\bar{X},\bar{D})$ be a log resolution, i.e., a proper morphism which is an isomorphism over the locus where $\bar{X}$ is smooth and $\bar{D}$ is snc; such that $X$ is smooth and $D\de \pi^{*}\bar{D}$ is snc. Such a resolution exists, see \cite[Theorem 10.45(2)]{Kollar_singularities_of_MMP} and references there. Write $D=\sum_{i=1}^{N}m_{i}D_{i}$, and split $\{1,\dots,N\}=\cE\sqcup \cP$, where $\sum_{i\in \cE} D_{i}$ is the exceptional divisor of $\pi$, and $\sum_{i\in \cP} D_{i}$ is the strict transform of $\bar{D}$. Note that $m_i=1$ for all $i\in \cP$.

Recall that $K_{\bar{X}}+\bar{D}$ is a Cartier divisor. Hence $K_{X}+D\redd-\pi^{*}(K_{\bar{X}}+\bar{D})$ is a Cartier divisor, too. Its push-forward on $\bar{X}$ is trivial since $\pi_{*}D\redd=\bar{D}$. Thus for some integers $\nu_{i}$ we have a linear equivalence 
	\begin{equation*}
		K_{X}+D\redd\sim \pi^{*}(K_{\bar{X}}+\bar{D})+\sum_{i\in \cE} \nu_i D_{i}\sim \sum_{i=1}^{N} \nu_{i} D_{i},
	\end{equation*}
	where we put $\nu_{i}=0$ for $i\in \cP$, and use the linear equivalences $K_{\bar{X}}\sim \bar{D}\sim 0$. We now choose a section $\Theta$ of $K_{X}$ with zeros of order $\nu_{i}-1$ along each $D_{i}$, see formula \eqref{eq:discrepancies}.
	
	Since the pair $(\bar{X},\bar{D})$ is dlt and the resolution $\pi$ is an isomorphism over its snc locus, we have $\nu_{i}>0$ for all $i\in \cE$, see \cite[\textsection 2.3.1]{NX_skeleton}. Thus  $\cP=\{i:\nu_{i}=0\}=\{i:\frac{\nu_i}{m_i}=\min_{j} \frac{\nu_{j}}{m_j}\}$. We conclude that $\cP$ is the set \eqref{eq:essential}, i.e., $i\in \cP$ if and only if a component $D_{i}$ of $D$ is essential, cf.\ \cite[Theorem 3.3.3]{NX_skeleton}. Now conditions \ref{item:sk_reduced}, \ref{item:sk_nu-i} follow from the fact that $(m_i,\nu_i)=(1,0)$ for all $i\in \cP$.
	
	Let now $\Delta_{I}$ be a submaximal face of the dual complex of $D$, so $\dim X_{I}=1$. Assume that $\Delta_{I}$ is essential, so $I\subseteq \cP$. By definition of $\cP$, the morphism $\pi$ is an isomorphism at a general point of $X_{I}$, so $\dim\pi(X_{I})=\dim X_{I}=1$. Now condition \eqref{eq:NXY} implies that $\pi(X_{I})$ is contained in the snc locus of $\bar{D}$, so $\pi$ is an isomorphism near $X_{I}$. It follows that for every component $D_{j}$ of $D$ meeting $X_{I}$, we have $j\in \cP$, hence $D_j$ is essential. This proves the first part of \ref{item:sk_P1_or_ell} and shows that $(K_{X}+D)|_{X_{I}}\sim 0$.
	
	Write $B_{I}=(\sum_{i\not \in I} m_{i}D_{i})|_{X_{I}}$, so $B_{I}$ is a divisor on $X_{I}$ supported on $X_{I}\setminus X_{I}^{\circ}$. Since each component $D_{i}$ meeting $X_{I}$ has $i\in \cP$, so $m_i=1$, we see that $B_{I}$ is reduced. Applying Lemma \ref{lem:adj_ind} below to $J=\emptyset$, we get $K_{X_{I}}+B_{I}\sim (K_{X}+D)|_{X_{I}}\sim 0$. It follows that either $B_{I}=0$ and $X_{I}=X_{I}^{\circ}$ is an elliptic curve; or $B_{I}\neq 0$ and $(X_{I},X_{I}^{\circ})\cong (\mathbb{P}^1,\C^{*})$. Since by \cite[Theorem 4.2.4(3)]{NX_skeleton}, the essential skeleton $\Delta_{\cS}$ is a pseudomanifold (with boundary) we have $\dim\Delta_{\cS}=n-1$ in the first case and $\dim\Delta_{\cS}=n$ in the second case, see condition (1) of the definition of pseudomanifold given in \textsection 4.1.2 loc.\ cit. By Theorem 4.1.10 loc.\ cit., the second case happens if and only if $f^{\circ}$ is a maximal degeneration. 
\end{proof}

In the above proof, we used the following elementary lemma.

\begin{lema}\label{lem:adj_ind}
	Let $D$ be an snc divisor on a smooth projective variety $X$. Write $D=\sum_{i=1}^{N} m_{i}D_{i}$, and for $I\subseteq \{1,\dots, N\}$ put $X_{I}=\bigcap_{i\in I} D_{i}$, $X_{\emptyset}=X$,  $B_{I}=\sum_{i\not\in I}m_{i}D_{i}|_{X_{I}}$. Fix two subsets $J\subseteq I\subseteq \{1,\dots, N\}$, and assume that $m_{i}=1$ for all $i\in I$. Then $K_{X_I}+B_I\sim (K_{X_{J}}+B_{J})|_{X_{I}}$.
\end{lema}
\begin{proof}
	By induction on $\#I\setminus J$, we can assume that $I=J\sqcup\{i\}$. By adjunction, we have $K_{X_{I}}=(K_{X_{J}}+D_{i}|_{X_{J}})|_{X_{I}}=K_{X_{J}}|_{X_{I}}+D_{i}|_{X_{I}}$. Moreover, $B_{I}=\sum_{j\not\in I} m_jD_j|_{X_{I}}=\sum_{j\not\in J} m_{j}D_{j}|_{X_{I}}-m_{i}D_{i}|_{X_{I}}=B_{J}|_{X_{I}}-D_{i}|_{X_{I}}$ since $m_{i}=1$ by assumption. Adding the two equalities we get $K_{X_{I}}+B_{I}=(K_{X_{J}}+B_{J})|_{X_{I}}$.
\end{proof}

\begin{example}[Hesse pencil]\label{ex:Fermat}
	Consider a family of elliptic curves $X=\{(x_1^3+x_2^3+x_3^3)\cdot z=x_1x_2x_3\}\subseteq \mathbb{P}^{2}\times \C$, and let $f\colon X\to \C$ be the projection to the $z$-coordinate. The restriction $f^{\circ}$ of $f$ over $\C^{*}$ is a maximal Calabi--Yau degeneration, and $f$ is its model satisfying all conditions of Proposition \ref{prop:model}. Indeed, the central fiber $D=f^{-1}(0)$ is just a union of three non-concurrent lines $D_1,D_2,D_3\subseteq \mathbb{P}^2$; and $K_{X}=0$, so each $D_i$ is essential. We conclude that the essential skeleton is a triangle.
	
	Fix $p\in D$ and let $\pi\colon X'\to X$ be a blowup at $p$, so $f'\de f\circ \pi$ is another model of $f^{\circ}$. Put $D'=(f')^{-1}(0)$, $D_i'=\pi^{-1}_{*}D_i$ and $D_{0}'=\Exc\pi$, so $K_{X'}=D_{0}'$. Let $m_{i}'$, $\nu_{i}'$ be as in formulas \eqref{eq:multiplicites}, \eqref{eq:discrepancies}, so $\nu_{i}'/m_{i}'=1$ for $i\in \{1,2,3\}$ and $\nu_{0}'=2$. 
	
	Assume that $p$ is a smooth point of $D$. Then $m_{0}'=1$, so $\nu_{0}'/m_{0}'=2>1$, i.e., $D_{0}'$ is not essential. Hence the model $f'$ satisfies conditions \ref{prop:model}\ref{item:sk_reduced} and \ref{item:sk_nu-i}, but does not satisfy \ref{prop:model}\ref{item:sk_P1_or_ell}. Note that the component $D_{i}'$ meeting $D_0'$ satisfies  $(D_{i}',(D_{i}'\setminus (D'-D_{i}'))\cong (\mathbb{P}^{1},\C^{**})$, so \ref{prop:model}\ref{item:sk_P1} fails, too. The essential skeleton $\Delta_{\cS}'$ computed using this model is still a triangle, in agreement with \cite[Theorem 4.7.5]{MN_weight-functions}.
	
	Assume now that $p$ is a singular point of $D$. Then $m_{0}'=2$, so $\nu_{0}'/m_{0}'=1$, i.e., all components of $D'$ are essential. This model satisfies  conditions \ref{prop:model}\ref{item:sk_nu-i} and \ref{item:sk_P1_or_ell}, but does not satisfy \ref{prop:model}\ref{item:sk_reduced}. Now $\Delta_{\cS}'$ is a square obtained by subdividing the triangle $\Delta_{\cS}$, again in agreement with \cite[Theorem 4.7.5]{MN_weight-functions}.
\end{example}

\begin{example}[Families of Calabi--Yau hypersurfaces in $\mathbb{P}^{n+1}$]
	\label{ex:general-model}
	Fix $n\geq 2$. Consider a family of hypersurfaces $\bar{X}=\{P(x_0,\dots,x_{n+1})\cdot
	z=x_0\cdot \ldots \cdot x_{n+1}\}\subseteq \mathbb{P}^{n+1}\times \C$, where $P$ is a general
	homogeneous polynomial of degree $n+2$, and let $\bar{f}\colon X\to \C$ be the
	projection to the $z$-coordinate. We have $K_{\bar{X}}\sim 0$ by adjunction. The restriction $f^{\circ}$ of
	$\bar{f}$ over $\C^{*}$ is a maximal Calabi--Yau degeneration. For $n=2$ it is a classical K3 degeneration studied e.g.\ in \cite{Gross-Wilson_K3,CG-K3}, see Figure \ref{fig:K3-simplex} below. Case $n=3$ is thoroughly described in \cite[\textsection 3]{MazzonSchneider-toric}. For general $n$, such families were studied e.g.\ in \cite{MazzonSchneider-toric,Pille-Schneider,HJMMC_NAMA}, see Example \ref{ex:general-fibration} for a comparison of  some results of loc.\ cit.\ with our constructions.
	
	The central fiber $\bar{D}=\bar{f}^{-1}(0)$ is the sum of coordinate
	hyperplanes $\bar{D}_0,...,\bar{D}_{n+1}$. Nonetheless, $\bar{f}$ is not an snc model of $f^{\circ}$, because the total space $\bar{X}$ is singular at
	the intersection of the zero locus of $P$ and any two coordinate hyperplanes (it is not even a good dlt model, as $\bar{D}_{i}$ are not $\Q$-Cartier). For instance, in case $n=2$, $\bar{X}$ has 24 singular points, $6$ in each submaximal stratum of $\bar{D}$, see Figure \ref{fig:K3-simplex}.
	
	To obtain an snc model, fix a permutation $\sigma$ of the set $\{0,...,n+1\}$, and blow up $\bar{X}$ successively at the strict transforms of the Weil divisors $\bar{D}_{\sigma(0)},...,\bar{D}_{\sigma(n)}$. The resulting morphism $\pi\colon \bar{X}\to X$ is crepant, so $K_{X}\sim 0$. The composition $f=\bar{f}\circ\pi \colon X\to \C$ is an snc model of $f^{\circ}$, whose central fiber $D$ is the strict transform of $\bar{D}$, in particular, it is reduced. Since $K_{X}\sim 0$, all components of $D$ are essential. We conclude that $f$ is a model as in Proposition \ref{prop:model}, and the essential skeleton $\Delta_{\cS}$ equals the dual complex of $D$, which is an $n$-dimensional simplex. For each $i\in \cS=\{0,\dots,n+1\}$, the maximal face $\Delta_{\cS\setminus \{i\}}$ of $\Delta_{\cS}$ corresponds to the unique common point of all components of $D-D_i$. Each submaximal face $\Delta_{\cS\setminus \{i,j\}}$ corresponds to the intersection of all components of $D-D_i-D_j$, which is a projective line meeting $D_i$ and $D_j$ in two different points, as required by Proposition \ref{prop:model}\ref{item:sk_P1}.
	
	For different permutations $\sigma$ we get different models, connected by flops.
\begin{figure}[ht]
	\begin{tikzpicture}[scale=0.9]
		\begin{scope}
			\coordinate (A) at (0,0);
			\coordinate (B) at (3,0);
			\coordinate (C) at (2,2.5);
			\coordinate (D) at (4,1);
			\draw[add = 0.1 and 0.1] (A) to (B);
			\draw[add = 0.1 and 0.1] (A) to (C);
			\draw[add = 0.1 and 0.1, densely dashed] (A) to (D);
			\draw[add = 0.1 and 0.1] (B) to (C);
			\draw[add = 0.1 and 0.1] (B) to (D);
			\draw[add = 0.1 and 0.1] (C) to (D);
			\filldraw ($(A)!0.2!(B)$) circle (0.06);
			\filldraw ($(A)!0.4!(B)$) circle (0.06);
			\filldraw ($(A)!0.6!(B)$) circle (0.06);
			\filldraw ($(A)!0.8!(B)$) circle (0.06);
			\filldraw ($(A)!0.2!(C)$) circle (0.06);
			\filldraw ($(A)!0.4!(C)$) circle (0.06);
			\filldraw ($(A)!0.6!(C)$) circle (0.06);
			\filldraw ($(A)!0.8!(C)$) circle (0.06);
			\filldraw ($(B)!0.2!(C)$) circle (0.06);
			\filldraw ($(B)!0.4!(C)$) circle (0.06);
			\filldraw ($(B)!0.6!(C)$) circle (0.06);
			\filldraw ($(B)!0.8!(C)$) circle (0.06);
			\filldraw ($(B)!0.2!(D)$) circle (0.06);
			\filldraw ($(B)!0.4!(D)$) circle (0.06);
			\filldraw ($(B)!0.6!(D)$) circle (0.06);
			\filldraw ($(B)!0.8!(D)$) circle (0.06);
			\filldraw ($(C)!0.2!(D)$) circle (0.06);
			\filldraw ($(C)!0.4!(D)$) circle (0.06);
			\filldraw ($(C)!0.6!(D)$) circle (0.06);
			\filldraw ($(C)!0.8!(D)$) circle (0.06);
			\filldraw ($(A)!0.2!(D)$) circle (0.06);
			\filldraw ($(A)!0.4!(D)$) circle (0.06);
			\filldraw ($(A)!0.6!(D)$) circle (0.06);
			\filldraw ($(A)!0.8!(D)$) circle (0.06);
			\draw[<-] (4.6,1.5) -- (7,1.5);
			\node at (5.8,1.7) {\small{$\pi$}};
			\node at (5.8,1.3) {\scriptsize{small resolution}};
			\node at (3.8,0.2) {\small{$\bar{D}$}};
		\end{scope}
		\begin{scope}[shift={(7,0)}]
			\coordinate (A) at (0,0);
			\coordinate (B) at (3,0);
			\coordinate (C) at (2,2.5);
			\coordinate (D) at (4,1);
			\draw[add = 0.1 and 0.1] (A) to (B);
			\draw[add = 0.1 and 0.1] (A) to (C);
			\draw[add = 0.1 and 0.1, densely dashed] (A) to (D);
			\draw[add = 0.1 and 0.1] (B) to (C);
			\draw[add = 0.1 and 0.1] (B) to (D);
			\draw[add = 0.1 and 0.1] (C) to (D);
			\draw[thick, add=0 and -0.9] ($(A)!0.2!(B)$) to (C); 
			\draw[thick, add=0 and -0.9] ($(A)!0.4!(B)$) to (C);
			\draw[thick, add=0 and -0.9] ($(A)!0.6!(B)$) to (C);
			\draw[thick, add=0 and -0.9] ($(A)!0.8!(B)$) to (C);
			\draw[thick, add=0 and -0.9] ($(A)!0.2!(C)$) to (B); 
			\draw[thick, add=0 and -0.9] ($(A)!0.4!(C)$) to (B);
			\draw[thick, add=0 and -0.9] ($(A)!0.6!(C)$) to (B);
			\draw[thick, add=0 and -0.9] ($(A)!0.8!(C)$) to (B);
			\draw[thick, add=0 and -0.9] ($(B)!0.2!(C)$) to (A); 
			\draw[thick, add=0 and -0.9] ($(B)!0.4!(C)$) to (A);
			\draw[thick, add=0 and -0.9] ($(B)!0.6!(C)$) to (A);
			\draw[thick, add=0 and -0.9] ($(B)!0.8!(C)$) to (A);
			\draw[thick, add=0 and -0.9] ($(B)!0.2!(D)$) to (C); 
			\draw[thick, add=0 and -0.9] ($(B)!0.4!(D)$) to (C);
			\draw[thick, add=0 and -0.9] ($(B)!0.6!(D)$) to (C);
			\draw[thick, add=0 and -0.9] ($(B)!0.8!(D)$) to (C);
			\draw[thick, add=0 and -0.9] ($(C)!0.2!(D)$) to (B); 
			\draw[thick, add=0 and -0.9] ($(C)!0.4!(D)$) to (B);
			\draw[thick, add=0 and -0.9] ($(C)!0.6!(D)$) to (B);
			\draw[thick, add=0 and -0.9] ($(C)!0.8!(D)$) to (B);
			\draw[thick, add=0 and -0.85,densely dashed] ($(A)!0.2!(D)$) to (C); 
			\draw[thick, add=0 and -0.85,densely dashed] ($(A)!0.4!(D)$) to (C);
			\draw[thick, add=0 and -0.85,densely dashed] ($(A)!0.55!(D)$) to (C);
			\draw[thick, add=0 and -0.85,densely dashed] ($(A)!0.8!(D)$) to (C);
			\node at (3.6,0.2) {\small{$D$}};
		\end{scope}
	\begin{scope}[shift={(13,-0.2)}]
		\coordinate (A) at (1.5,0.5);
		\coordinate (B) at (0,2);
		\coordinate (C) at (3,2);
		\coordinate (D) at (1.5,3);
		\draw (A) to (B);
		\draw (A) to (C);
		\draw[densely dashed] (A) to (D);
		\draw (B) to (C);
		\draw (B) to (D);
		\draw (C) to (D);
		\node at (1.5,0.2) {\small{$\Delta_{\bar{D}}=\Delta_{D}=\Delta_{\cS}$}};
	\end{scope}	
	\end{tikzpicture}
	\caption{Example \ref{ex:general-model}, case $n=2$: a maximal K3 degeneration.}
	\label{fig:K3-simplex}
\end{figure}
\end{example}

\section{The A'Campo space and its \texorpdfstring{$\cC^{\infty}$}{smooth} structures}\label{sec:A}

Let $f\colon X\to \C$ be as in the beginning of Section \ref{sec:prelim}, i.e.,  $f$ is a holomorphic function whose unique singular fiber $D\de f^{-1}(0)$ is snc; and  $f|_{X\setminus D}\colon X\setminus D\to \D_{\delta}^{*}$ is a submersion. In \cite{A'Campo}, A'Campo introduced a topological space $A$ which fits into diagram~(\ref{eq:AX-diagram-intro}), hence allows to extend the monodromy of $f$ to \enquote{radius zero}, i.e., to the monodromy of the restriction $f|_{\d A}\colon \d A \to \d \C_{\log}$. This radius-zero monodromy has simpler dynamical properties than the natural one at positive radius. In \cite{FdBP_Zariski}, we upgraded the construction of $A$ to the $\cC^{\infty}$ category, and endowed it with a fiberwise symplectic form $\omega_{A}^{\epsilon}$ such that the radius-zero monodromy is induced by its associated symplectic connection. Our aim is to further improve this construction so that $\omega_{A}^{\epsilon}$ becomes fiberwise \emph{K\"ahler} at positive radius. 

\subsection{Definition of the A'Campo space}

In this section, we recall the definition of the topological A'Campo space $A$ from \cite{A'Campo}. In order to set up the notation needed to introduce the fiberwise symplectic structure, it is convenient for us to follow the construction of $A$ given in \cite{FdBP_Zariski}. 
 
Recall that $\sum_{i=1}^{N}m_iD_i$ is the irreducible decomposition of the central fiber $D=f^{-1}(0)$. 

We begin with the notion of an \emph{adapted chart} \cite[Definition 3.1]{FdBP_Zariski}. 
For a holomorphic chart $(z_{i_1},\dots,z_{i_n})\colon U_{X}\to \C^{n+1}$ of $X$, we define its \emph{associated index set} as $S=\{i\in \{1,\dots, N\}:D_i\cap U_{X}\neq \emptyset\}$. We say that this chart is \emph{adapted to $f$} if for every $i\in \{1,\dots, N\}$ we have $U_X\cap D_i=\{z_i=0\}$ and $f|_{U_X}=\prod_{i\in S}z_{i}^{m_i}$ in case $S\neq \emptyset$. We also require $U_{X}$ to be small enough so that $\log |z_{i}|<1$ for all $i\in S$, and each coordinate $z_i$ extends to a continuous function on the compact closure $\bar{U}_X$.

For any $i\in S$ we have the following smooth functions on $U_{X}\setminus D$, see \cite[\textsection 3.1.2]{FdBP_Zariski}
\begin{equation}\label{eq:basic functions_r}
		r_{i}=|z_i|,\quad \theta_i=\frac{z_i}{r_i},\quad s_i=\log r_i,\quad t_i=\frac{-1}{m_i s_i}.
\end{equation}
We call $r_i$ and $\theta_i$ the \emph{radial} and \emph{angular} coordinates of $U_{X}$. We also define global functions on $X\setminus D$
\begin{equation}\label{eq:basic functions_t}
	t=\frac{-1}{\log|f|},\quad g=\eta(t),\quad \theta=\frac{f}{|f|},
\end{equation}
where $\eta\colon [0,1]\to[0,1]$ is given by $\eta(\tau)=(1-\log \tau)^{-1}$. The key property of this auxiliary function is that its inverse is a smooth function whose all derivatives vanish at the origin. Note that the functions $t$ and $g$ are just re-scaled versions of $|f|$, in particular we have $t,g\rightarrow 0$ as $f\rightarrow 0$. We remark that $t$ is a \enquote{natural} rescaling of $|f|$ in the setting of maximal Calabi-Yau degenerations, cf.\ formula \eqref{eq:omega_new}. 
The further rescaling $g$ is needed to define a smooth atlas on $A$, but is not as crucial.

We view $\theta_{i}$ and $\theta$ as maps to the circle $\mathbb{S}^{1}$, identified with the additive group $\R/(2\pi\Z)$. With this identification, we get $\theta_{i}=\Im\log z_i$, $\theta=\Im\log f$. 
Eventually, for each $i\in S$ we put
\begin{equation}\label{eq:basic_functions_v}
		w_{i}=\frac{t}{t_i},\quad u_{i}=\eta(w_i),\quad v_{i}=t_i-u_i \quad\mbox{and}\quad \sigma_{i}= t_{i}^{2}+t_{i}u_{i}^{2}.
\end{equation}
The functions $r_i$, $w_i$ and $v_i$ should be thought of as \enquote{natural}, \enquote{tropical} and \enquote{hybrid} coordinates, respectively. Note that $w_{i}\in [0,1]$, $w_{i}$ converges to $0$ or $1$ as we approach $D\setminus D_i$ and $D_i\setminus (D-D_i)$, respectively. As we will see below, the A'Campo space is designed to resolve the indeterminacy of $w_i$ as we approach $D_{i}\cap (D-D_i)$. The same holds for $u_i$, which is a rescaling of $w_i$. For more intuition, we refer the reader to 
\cite[Figures 2 and 3]{FdBP_Zariski}. The last function $\sigma_{i}$ plays but an auxiliary role.

In order to define the A'Campo space as a $\cC^{1}$-manifold, it is convenient to choose an atlas $\{U_{X}^{p}\}_p$ consisting of adapted charts, a subordinate partition of unity $\{\tau^{p}\}_p$, and define
\begin{equation*}
	\bar{u}_{i}=\sum_{p}\tau^{p}u_{i}^{p},\quad \bar{v}_i=\sum_{p}\tau^{p}v_{i}^{p},
\end{equation*}
where for every $p$ such that $U_{X}^{p}\cap D_{i}\neq \emptyset$, the functions $u_{i}^{p},v_{i}^{p}$ are the ones defined for $U_{X}^{p}$ by formulas \eqref{eq:basic_functions_v}; and for $p$ such that $U_{X}^{p}\cap D_{i}=\emptyset$, we put $u_{i}^{p}=v_{i}^{p}=0$.

We define a topological space $\Gamma$ as the closure of the graph of $(\bar{u}_{1},\dots, \bar{u}_{N})\colon X\setminus D\to \R^{N}$, and the A'Campo space $A$ as the fiber product over $X$ of $\Gamma$ and the Kato--Nakayama space $X_{\log}$ of $(X,D)$, see diagram \eqref{eq:AX-diagram}. This way, $A$ is a topological space equipped with a continuous map $\pi\colon A\to X$ which is a homeomorphism over $X\setminus D$. Moreover, for any adapted chart $U_{X}\subseteq X$, all basic functions \eqref{eq:basic functions_r}--\eqref{eq:basic_functions_v} extend to continuous functions on the preimage $\pi^{-1}(U_X)$ \cite[Lemma 3.6(a)]{FdBP_Zariski}. Crucially for us, the value of each $w_{i}$ on $\pi^{-1}(U_X)\cap \d A$, where $\d A=\pi^{-1}(D)$ does not depend on the choice of the adapted chart $U_X$, see Lemma 3.6(d),(e) loc.\ cit.
	\begin{equation}\label{eq:AX-diagram}
		\begin{tikzcd}
			A
			\ar[r] 
			\ar[d, 
			] 
			\ar[dd, bend right=90, looseness=1, "{(g,\theta)\sim f_A}"'] 
			\ar[dr, phantom, "\ulcorner", very near start]
			\ar[dr, "\pi", start anchor = south east, shorten={5mm}]
			& \Gamma\de \overline{\mbox{\small{graph}}(\mu)} 
			\ar[d, 
			] 
			\ar[r, hook]
			& X\times  \R^{N} 
			\ar[d] 
			\\
			X_{\log}
			\ar[r]
			\ar[d, "f_{\log}"] 
			& X 
			\ar[r, dashed, "\mu", "{(\bar{u}_1,\ldots,\bar{u}_{N})}"'] 
			\ar[d, "f"] 
			& \R^{N} 
			\\
			\C_{\log}
			\ar[r] & \C &
		\end{tikzcd}
	\end{equation}	
The stratification $X=\bigsqcup_{I}X_{I}^{\circ}$ lifts to a decomposition $A=\bigsqcup_{I} A_{I}^{\circ}$, where $A_{I}^{\circ}\de \pi^{-1}(X_{I}^{\circ})$. We will see in Lemma \ref{lem:product} that each piece $A_{I}^{\circ}$ for $I\neq \emptyset$ decomposes as a product of $(X_{I}^{\circ})_{\log}$, i.e., a torus bundle over $X_{I}^{\circ}$, and a simplex $\Delta_{I}$, smoothly \enquote{rounded} by the function $\eta$. An example of the A'Campo space, together with the above decomposition, is shown in Figure \ref{fig:ACampo} below.

\begin{figure}[htbp]
	\begin{tikzcd}[column sep=4em]
		\begin{tikzpicture}
			\draw[thick, fill=black!8] (-1.4,0) -- (0,-1.4) -- (0.6,0.6) -- (-1.4,0);
			\draw [->,gray] (0,0)-- (0,-2); \node at (-0.3,-1.9) {\small{$w_2$}};
			\draw [->,gray] (0,0) -- (-2,0); \node at (-1.8,-0.3) {\small{$w_1$}};
			\draw [->,gray] (0,0) -- (1.2,1.2); \node at (0.8,1.2) {\small{$w_3$}};
			\filldraw (-1.4,0) circle (0.08);
			\node at (-1.4,0.2) {\small{$\Delta_{1}$}};	
			\filldraw (0,-1.4) circle (0.08);
			\node at (0.3,-1.5) {\small{$\Delta_{2}$}};	
			\filldraw (0.6,0.6) circle (0.08);
			\node at (0.4,0.8) {\small{$\Delta_{3}$}};
			\node at (0.8,-0.4) {\small{$\Delta_{2,3}$}};
			\node at (-1,-0.9) {\small{$\Delta_{1,2}$}};
			\node at (-0.4,0.5) {\small{$\Delta_{1,3}$}};
			\node at (-0.2,-0.4) {\small{$\Delta_{1,2,3}$}};
			\node at (0.8,-2) {$\R^3$};
		\end{tikzpicture}		
		&		
		\begin{tikzpicture}[scale=0.85]
			\path[use as bounding box] (-2.5,-2.5) rectangle (3.6,3.6);
			\path [fill=black!8] (0.4,0.4) to [out=180,in=45] (-1,0) to [out=270,in=180] (0,-1) to [out=45,in=270] (0.4,0.4);
			\node at (-0.25,-0.25) {\small{$A_{\! 1,2,3}^{\circ}$}};
			\path [fill=black!4] (0.4,0.4) to [out=180,in=45] (-1,0) -- (-1,2.4) to [out=45,in=180] (.4,2.8) -- (0.4,0.4);
			\node at (-.3,1.8) {\small{$A_{1,3}^{\circ}$}};
			\path [fill=black!4] (-1,0) to [out=270,in=180] (0,-1) -- (-1.5,-2.5) to [out=180,in=270] (-2.5,-1.5) -- (-1,0);
			\node at (-1.3,-1.4) {\small{$A_{1,2}^{\circ}$}};
			\path [fill=black!4] (0.4,0.4) to [out=270,in=45] (0,-1) -- (2.4,-1) to [out=45,in=270] (2.8,0.4) -- (-0.4,0.4);
			\node at (1.9,-0.5) {\small{$A_{2,3}^{\circ}$}};
			\path [fill=black!1] (0.4,0.4) --(0.4,2.8) -- (2.8,2.8) -- (2.8,0.4) -- (0.4,0.4);
			\node at (1.9,1.8) {\small{$A_{3}^{\circ}$}};
			\path [fill=black!1] (-1,0) -- (-1,2.4) -- (-2.5,0.9) -- (-2.5,-1.5) -- (-1,0);
			\node at (-1.9,0.5) {\small{$A_{1}^{\circ}$}};
			\path [fill=black!1] (0,-1) -- (2.4,-1) -- (0.9,-2.5) -- (-1.5,-2.5) -- (0,-1);
			\node at (0.7,-1.9) {\small{$A_{2}^{\circ}$}};	
			\draw[->,gray] (-2,0) -- (3.6,0);
			\node [below] at (3.4,0) {\small{$v_1$}};
			\draw[->,gray] (0,-2) -- (0,3.6);
			\node [left] at (0,3.4) {\small{$v_2$}};
			\draw[->,gray] (1,1) -- (-2.8,-2.8);
			\node [above] at (-2.7,-2.7) {\small{$v_3$}};

			\draw [thick] (-1,2.4) -- (-1,0) to [out=270,in=180]	(0,-1) -- (2.4,-1);
			\draw [black!20] (-1.6,1.8) -- (-1.6,-0.6) to [out=270,in=180] (-0.6,-1.6) -- (1.8,-1.6);
			\draw [black!20] (-2.2,1.2) -- (-2.2,-1.2) to [out=270,in=180] (-1.2,-2.2) -- (1.2,-2.2);
			\draw [thick] (2.8,0.4) -- (.4,.4) to [out=180,in=45] (-1,0) -- (-2.5,-1.5);
			\draw [black!20] (2.8,1.4) -- (.5,1.4) to [out=180,in=45] (-1,1) -- (-2.5,-0.5);
			\draw [black!20] (2.8,2.4) -- (.5,2.4) to [out=180,in=45] (-1,2) -- (-2.5,0.5);
			\draw [thick] (0.4,2.8) -- (.4,.4) to [out=270,in=45] (0,-1) -- (-1.5,-2.5);
			\draw [black!20] (1.4,2.8) -- (1.4,.5) to [out=270,in=45] (1,-1) -- (-0.5,-2.5);
			\draw [black!20] (2.4,2.8) -- (2.4,.5) to [out=270,in=45] (2,-1) -- (0.5,-2.5);
			\node[right] at (1.1,-2.2) {\small{$\times\ (\mathbb{S}^{1})^{3}$}};
			\node[right] at (1.1,-2.6) {\small{(angular directions $\theta_i$)}};	
		\end{tikzpicture}
		\ar[r, "\pi"] 
		\ar[d, "f_{A}"'] 
		\ar[l, "{(w_1,w_2,w_3)}"']
		& 
		\begin{tikzpicture}[scale=0.9]
			\path[fill=black!1] (-1,2.4) -- (-1,0) -- (1.4,0) -- (1.4,2.4) -- (-1,2.4);
			\path[fill=black!1] (-1,2.4) -- (-1,0) -- (-2.4,-1.4) -- (-2.4,1) -- (-1,2.4);
			\path[fill=black!1](-1,0) -- (-2.4,-1.4) -- (0,-1.4) -- (1.4,0) -- (-1,0);
			\draw [thick] (-1,2.4) -- (-1,0) -- (1.4,0);
			\draw [black!20] (-1.6,1.8) -- (-1.6,-0.6) -- (0.8,-0.6);
			\draw [black!20] (-2.2,1.2) -- (-2.2,-1.2) -- (0.2,-1.2);
			\draw [thick] (-1,0) -- (-2.4,-1.4);
			\draw [black!20] (1.4,1) -- (-1,1) -- (-2.4,-0.5);
			\draw [black!20] (1.4,2) -- (-1,2) -- (-2.4,0.5);
			\draw [black!20] (0,2.4) -- (0,0) -- (-1.5,-1.4);
			\draw [black!20] (1,2.4) -- (1,0) -- (-0.5,-1.4);
			\filldraw (-1,0) circle (0.08);
			\node at (-1.8,0.6) {\small{$X_{1}^{\circ}$}};
			\node at (-0.4,-1) {\small{$X_{2}^{\circ}$}};
			\node at (0.5,1.4) {\small{$X_{3}^{\circ}$}};
			\node at (1,0.2) {\small{$X_{2,3}^{\circ}$}};
			\node at (-1.8,-1.4) {\small{$X_{1,2}^{\circ}$}};
			\node at (-0.6,2) {\small{$X_{1,3}^{\circ}$}};
			\node at (-0.5,0.2) {\small{$X_{1,2,3}^{\circ}$}};
			\node at (1.2,-1.6) {\small{$\C^{3}$}};
		\end{tikzpicture}
		\ar[d,"f"',"z_{1}\cdot z_{2}\cdot z_{3}"]
		\\
		&
		\begin{tikzpicture}[scale=0.8]
			\draw[fill=black!10] (0,0) circle (0.6);
			\draw[thick, fill=white] (0,0) circle (0.1);
			\node at (1,-0.4) {\small{$\D_{\log }$}};	
			\node at (-1,-0.4) {\small{$\phantom{\D_{\log }}$}};
		\end{tikzpicture}	
		\ar[r,"{(r,\theta)\ \mapsto\  r\cdot e^{2\pi i\, \theta}}"] 
		&
		\begin{tikzpicture}[scale=0.8]
			\draw[fill=black!10] (0,0) circle (0.6);
			\draw[thick] (0.05,0.05) -- (-0.05,-0.05);
			\draw[thick] (-0.05,0.05) -- (0.05,-0.05);
			\node at (0.9,-0.4) {\small{$\D$}};	
			\node at (-0.9,-0.4) {\small{$\phantom{\D}$}};
		\end{tikzpicture}
	\end{tikzcd}
\caption{The A'Campo space for $f=z_{1}\cdot z_{2}\cdot z_{3}$, cf.\ \cite[Figures 2 and 3]{FdBP_Zariski}.}
\label{fig:ACampo}
\end{figure}

We note that, since the values of all $u_i$ for different overlapping charts agree on $\d A\subseteq X_{\log}\times \R^{N}$, the latter subset, together with the above decomposition, does not depend on the choice of $(\{U_{X}^{p}\}_{p},\{\tau^{p}\}_{p})$. In turn, the subset $A\setminus \d A$ can be naturally identified with $X\setminus D$. This way, if $A$ and $A'$ denote the A'Campo spaces obtained via different choices of $(\{U_{X}^{p}\}_{p},\{\tau^{p}\}_{p})$, we get a natural homeomorphism $\Phi\colon A\to A'$ over $X$ which is the identity on $\d A$, see \cite[Proposition 3.7(e)]{FdBP_Zariski}. For a topological construction of $A$ which is independent of any choice we refer to the original work of A'Campo \cite{A'Campo}.

The $\cC^{1}$-atlas on $A$ is introduced as follows. For an adapted chart $U_X$ with associated index set $S$, we cover its preimage by open sets $U_{i}=\{w_{i}>\frac{1}{n+2}\}$, $i\in S$, which will be used as domains of smooth charts on $A$. To make the notation more compact, we put $\rest\de ((\theta_{i})_{i\in S},(z_{j})_{j\in \{i_1,\dots,i_n\}\setminus S})$. For $k\in \{1,\dots, n+1\}$ we write $Q_{k,n+1}=[0,\infty)\times \R^{k-1}\times (\mathbb{S}^1)^{k}\times \C^{n+1-k}$. Now, we define the  $\cC^{1}$-charts as 
\begin{equation}\label{eq:AC1-chart}
	\psi_{i}=(g,({v}_{j})_{j\in S\setminus \{i\}};\rest)\colon 
	U_{i}\to Q_{k,n+1}.
\end{equation}
Results of \cite[\textsection 3.3]{FdBP_Zariski} show that these charts induce a $\cC^{1}$-structure on the A'Campo space, which is independent of the choice of $\{(U_{X}^{p},\tau^{p})\}_p$, i.e., the homeomorphism $\Phi$ defined above is $\cC^1$. The $\cC^{\infty}$-charts are defined in \textsection 3.4 loc.\ cit.\ by a similar formula
\begin{equation}\label{eq:AC-chart}
	\bar{\psi}_{i}=(g,({\bar{v}}_{j})_{j\in S\setminus \{i\}};\rest)\colon 
	U_{i}\to Q_{k,n+1}.
\end{equation}
These charts upgrade the above $\cC^{1}$-structure to a $\cC^{\infty}$ one, which depends on the choice of the adapted atlas $(\{U_X^{p}\}_p,\{\tau^{p}\}_p)$, i.e., the $\cC^1$-diffeomorphism $\Phi$ is not $\cC^2$ in general, see \cite[Remark 3.22]{FdBP_Zariski}.

\subsection{Distance functions}\label{sec:distance_functions}

Our next goal is to endow the A'Campo space $A$ with another $\cC^{\infty}$-structure, which is $\cC^1$-compatible with the above one, but is more technically convenient to introduce a fiberwise symplectic form \emph{which is fiberwise K\"ahler at positive radius.} To this end, we  slightly modify the above construction, replacing the global functions $\bar{u}_i$, $\bar{v}_i$ by $\hat{u}_i$, $\hat{v}_i$, defined as follows. 
First, we fix global \emph{distance functions} $\hat{r}_i$, i.e., a collection of smooth functions satisfying the conclusion of Lemma \ref{lem:rihat}. Next, we build $\hat{u}_i$, $\hat{v}_i$ out of $\hat{r}_i$ in the same way as the local functions  $u_i$, $v_i$ are defined from the radial coordinates $r_i$.

We work in a fixed open set $W_{X}\subseteq X$ such that the restriction $f|_{\bar{W}_{X}}$ is proper, has connected fibers; and $f|_{\bar{W}_X\setminus D}\colon \bar{W}_{X}\setminus D\to \D_{\delta}^{*}$ is a submersion. If $f$ is an snc model of a Calabi--Yau degeneration, which is the case of most interest for us, we simply take $W_{X}=X$. For another example, let $X$ be an embedded resolution of an isolated hypersurface singularity, and take for $W_{X}$ the preimage of some bounded neighborhood of the origin. We can and do assume that $\delta<\exp(-\max_{i}m_{i})$. 

\begin{lema}\label{lem:rihat}
	After shrinking $\delta>0$ if necessary, there exist, for $i=1,\dots,N$, an open neighborhood $R_{i}$ of $D_i$ in $X$ and a smooth function $\hat{r}_i\colon X\to [0,\frac{1}{e}]$ such that the following hold.
	\begin{enumerate} 
		\item \label{item:r_covering} Every fiber $f^{-1}(z)\cap \bar{W}_{X}$ is contained in $\bigcup_{i=1}^{N}R_{i}$ and meets $R_{i}^{\circ}\de R_{i}\setminus \bigcup_{j\neq i} \bar{R}_{j}$ for $i\in \{1,\dots, N\}$.		
		\item\label{item:r_comparison} For every point $x\in D_{i}$ and every adapted chart $U_{X}$ containing $x$, there is a neighborhood $V_X$ of $x$ in $U_X$ and a  smooth function $\lambda\colon \bar{V}_X\to (0,\infty)$ such that the radial coordinate $r_{i}$ of $U_X$ satisfies $\hat{r}_{i}|_{V_X}=\lambda r_{i}|_{V_X}$.
		\item \label{item:r_maximal} Let $\Delta_{I}$ be a maximal face of $\Delta_{D}$. There is an adapted chart $U_X$ around the point $X_{I}$ with radial coordinates $r_i$ such that on $U_X$ we have $r_i=\hat{r}_i$ if $i\in I$ and $\hat{r}_i=\frac{1}{e}$ otherwise.
		\item\label{item:r_generic} For every adapted chart $U_X\subseteq R_{i}^{\circ}$ meeting $D_i$, its radial coordinate $r_{i}$ is equal to $\hat{r}_i$.	
		\item\label{item:r_Ri} We have $\hat{r}_{i}<\frac{1}{e}$ on $R_{i}$ and $\hat{r}_{i}=\frac{1}{e}$ on $X\setminus R_{i}$. 
		\item\label{item:r_small} We have $\hat{r}_i^{m_i}\geq |f|$ and $\hat{r}_i^{m_i}=|f|$ on $R_{i}^{\circ}\cap f^{-1}(\D_{\delta})$. 	
	\end{enumerate}
\end{lema}
\begin{proof}
	Fix an atlas of adapted charts $\{U_{X}^{p}\}_{p}$, and denote by $S^{p}=\{i:U_X^p\cap D_i\neq \emptyset\}$ the associated index set of $U_{X}^{p}$. Put $R_{i}=\bigcup_{\{p\ :\ i\in S^{p}\}} U_{X}^{p}$. 
Refining the chosen atlas 
if needed, we can assume that 
\begin{enumerate}[(i)]
	\item\label{item:atlas_ri} For every $p$ and every $i\in S^{p}$ the radial coordinate $r_{i}^{p}$ of $U_{X}^{p}$ satisfies $r_{i}^{p}<\frac{1}{e}$.
	\item\label{item:atlas_maximal} For every maximal face $\Delta_{I}$ of $\Delta_{D}$, the point $X_{I}$ lies in exactly one chart $U_{X}^{p}$.
	\item\label{item:atlas_generic_1} For every $i\in \{1,\dots, N\}$ the open set $R_{i}^{\circ}\de R_{i}\setminus \bigcup_{j\neq i} \bar{R}_{j}$ is non-empty. 
\end{enumerate}
Since the set $\bar{W}_{X}\cap D$ is compact, shrinking $\delta>0$ if needed we can further assume that property \ref{item:r_covering} holds.  
Eventually, we choose a subordinate partition of unity $\{\tau^{p}\}_{p}$, put $r_{i}^{p}=\frac{1}{e}$ for $i\not\in S^{p}$, and define
\begin{equation*}
	\hat{r}_{i}=\sum_{p}\tau^{p}r_{i}^{p}\colon X\to [0,\tfrac{1}{e}].
\end{equation*}
We now prove that the functions $\hat{r}_i$ satisfy the required properties \ref{item:r_comparison}--\ref{item:r_small} above.

	\ref{item:r_comparison} Put $P_{x}=\{p:U_{X}^{p}\cap \bar{U}_{X}\neq \emptyset\}$, so $\hat{r}_{i}|_{\bar{U}_{X}}=\sum_{p\in P_{x}}\tau^{p}r_{i}^{p}$. Since $x\in D_i$, shrinking the chart $U_{X}$ around $x$ if necessary, we can assume that every chart $U_{X}^{p}$ meeting $\bar{U}_X$ meets $D_i$, too, i.e., $P_{x}\subseteq \{p:i\in S^{p}\}$. Fix $p\in P_{x}$ and put $V^{p}=U_{X}^{p}\cap U_{X}$. Recall that, by definition of the adapted charts, both $r_i$ and $r_i^{p}$ extend continuously to the compact closure $\bar{V}^{p}$. Put $\lambda^{p}\de \frac{r_{i}^{p}}{r_i}\colon \bar{V}^{p} \to (0,\infty)$. The function $\lambda\de \sum_{p\in P_x}\tau^{p}\lambda^{p}=\frac{\hat{r}_i}{r_i}$ is defined on the whole compact closure $\bar{U}_{X}$ and satisfies the required equality $\hat{r}_i=\lambda r_i$. Moreover,since $\tau^{p}\geq 0$ and $\lambda^p>0$, each summand $\tau^{p}\lambda^{p}$ is non-negative, and since $\sum_{p}\tau^{p}=1$, at each point of $\bar{U}_X$ at least one of them does not vanish. Hence $\lambda>0$ on $\bar{U}_X$, as needed.

	\ref{item:r_maximal} By property \ref{item:atlas_maximal} the point $X_{I}$ lies in $U_{X}^{p}$ for a unique $p$. Now, take $U_X=U_{X}^{p}\setminus \bigcup_{q\neq p}\bar{U}_{X}^{q}$.
	
	\ref{item:r_generic} Because $U_X$ is an adapted chart with associated index $\{i\}$, we have $r_{i}=|f|^{1/m_i}$. Since $U_{X}\subseteq R_{i}^{\circ}$, the same holds for all charts $U_{X}^{p}$ meeting $U_{X}$, so $\hat{r}_{i}=|f|^{1/m_i}$, too. Thus $\hat{r}_i=r_i$, as claimed.

	\ref{item:r_Ri} By property \ref{item:atlas_ri}, on $R_i$ we have $\tau^{p}r_{i}^{p}<\tau^{p}e^{-1}\leq e^{-1}$ if $U_{X}^{p}$ meets $D_i$ and $\tau^{p}r_{i}^{p}=\tau^{p}e^{-1}\leq e^{-1}$ otherwise. Since there is at least one chart $U_{X}^{p}$ meeting $D_i$, we get $\hat{r}_i|_{R_i}<e^{-1}$. On the other hand, on $X\setminus R_i$ we have $\tau^{p}r_{i}^{p}=\tau^{p}e^{-1}$ for all $p$, so $\hat{r}_i|_{X\setminus R_i}=e^{-1}$, as claimed.

	\ref{item:r_small} On each adapted chart $U_{X}^{p}$ we have $|f|=\prod_{j\in S^{p}} (r_{j}^{p})^{m_i}\leq (r_{i}^{p})^{m_i}$ for any $i\in S^{p}$, since  $r_i^{p}<e^{-1}<1$ by property \ref{item:atlas_ri}. In turn, if $i\not\in S^{p}$ then on $U_{X}^{p}$ we have $(r_{i}^{p})^{m_i}=e^{-m_i}>\delta>|f|$ by our assumption about $\delta$. Thus for every $i$, we have $\tau^{p}r_{i}^{p}\geq \tau^{p}|f|^{1/m_{i}}$ on $U_{X}^{p}$. Away from $U_{X}^{p}$ both sides of this inequality are zero, so it holds on the entire $X$. Now $\hat{r}_{i}=\sum_{p}\tau^{p}r_{i}^{p}\geq \sum_{p}\tau^{p}|f|^{1/m_i}=|f|^{1/m_i}$, as claimed. The last assertion follows from part \ref{item:r_generic}.
\end{proof}

\begin{remark}
The distance functions $\hat{r}_i$ can also be constructed as follows. Fix a hermitian metric on the normal bundle $N_{D_i}$ to $D_{i}$ in $X$; and identify a tubular neighborhood $R_{i}$ of $D_{i}$ with the total space of the disc bundle in $N_{D_{i}}$. Next, define $\hat{r}_{i}|_{R_i}$ as the distance function in each fiber, composed with a smooth bijection $[0,1]\to [0,\frac{1}{e}]$ whose all derivatives at $1$ vanish. Then extend $\hat{r}_{i}$ to $X$ by $\hat{r}_i|_{X\setminus R_i}=\frac{1}{e}$.

Nonetheless, this alternative construction should not be regarded as more canonical than the one in Lemma \ref{lem:rihat}. Indeed, choosing a hermitian metric on $N_{D_i}$ boils down to choosing a partition of unity.
\end{remark}

\subsection{A new smooth structure}\label{sec:A_smooth}

Fix any collection of functions $\hat{r}_1,\dots,\hat{r}_N\colon X\to [0,\frac{1}{e}]$ satisfying the conclusion of Lemma \ref{lem:rihat}. We repeat the definitions  of basic functions \eqref{eq:basic functions_r} and \eqref{eq:basic_functions_v} using 
$\hat{r}_i$ instead of the local 
$r_{i}$. That is, we put 
\begin{equation}\label{eq:basic_functions_hat} 
		\hat{s}_i=\log \hat{r}_i,\quad \hat{t}_i=\frac{-1}{m_i \hat{s}_i},\quad	\hat{w}_{i}=\frac{t}{\hat{t}_i},\quad \hat{u}_{i}=\eta(\hat{w}_i),\quad \hat{v}_{i}=\hat{t}_i-\hat{u}_i,\quad  \hat{\sigma}_i=\hat{t}_{i}^2+\hat{t}_i\hat{u}_i^2, 
		\hat{\alpha}_i=-d^{c}\hat{s}_i.
\end{equation}
 Note that Lemma \ref{lem:rihat}\ref{item:r_small} implies that $\hat{w}_{i}\in [0,1]$.  We use the convention $d^{c}=\imath (\d - \bar{\d})$, so $d\theta_{i}=-d^{c}s_i$. Therefore, the $1$-form $\hat{\alpha}_i$ is an analogue of the angular form $\alpha_i$ from  \cite[\textsection 4.1]{FdBP_Zariski}.

To write the new smooth charts, we replace $\bar{v}_i$ in formula \eqref{eq:AC-chart} by $\hat{v}_i$. That is, we define
\begin{equation}\label{eq:AC-chart_hat}
	\hat{\psi}_{i}=(g,(\hat{v}_{j})_{j\in S\setminus \{i\}};\rest)\colon 
	U_{i}\to Q_{k,n+1}.
\end{equation}

With these definitions, we have an analogue of \cite[Proposition 3.33]{FdBP_Zariski}.

\begin{prop}\label{prop:AX_hat_smooth}
	The collection $(\hat{\psi}_{i})$ for all adapted charts $U_{X}$ and all $i$ in the associated index set of $U_X$, is a $\cC^{\infty}$ atlas on the A'Campo space $A$. This atlas is compatible with the $\cC^{1}$-atlas given by charts \eqref{eq:AC1-chart}. The space $A$ endowed with this $\cC^{\infty}$ structure has the following properties.
	\begin{enumerate}
	\item\label{item:AX-pismooth} The map $\pi\colon A\to X$ is smooth. Its restriction $\pi|_{A\setminus \d A}\colon A\setminus \d A\to X\setminus D$ is a diffeomorphism.
	\item\label{item:AX-gsmooth} The map $(g,\theta)\colon A\to \C_{\log}$ is a smooth submersion. In particular, $(t,\theta)$ and $f_A$ are smooth. 
	\item \label{item:AX-vbar-smooth}For every $i\in \{1,\dots, N\}$ the functions $\hat{v}_i,\hat{w}_{i}$ and the $1$-form $\hat{\alpha}_i$ are smooth on $A$. 
\end{enumerate}	
\end{prop}

The proof of Proposition \ref{prop:AX_hat_smooth}, which we outline below, follows the same steps as the proof of \cite[Proposition 3.33]{FdBP_Zariski}. The key point is that we can compare the functions $\hat{t}_i,\hat{w}_i,\hat{v}_i$ with local functions $t_i,w_i,v_i$ in exactly the same way as we do in loc.\ cit. More precisely, we have the following formulas.

\begin{lema}\label{lem:comparison}
	Let $U_{X}$ be an adapted chart meeting the divisor $D_i$. For a subset $I$ of the associated index set of $U_X$ put $U_{X,I}^{\circ}=U_{X}\cap X_{I}^{\circ}$. Moreover, put $a\de -m_i \log\lambda\in \cC^{\infty}(U_X)$, where $\lambda$ is the function from Lemma \ref{lem:rihat}\ref{item:r_comparison}. Then the functions introduced in \eqref{eq:basic functions_r}, \eqref{eq:basic_functions_v} and \eqref{eq:basic_functions_hat} satisfy the following identities.
	\begin{enumerate}
		\item \label{item:t-comparison} On $U_{X}$, we have $t_i-\hat{t}_{i}=\frac{at_i^2}{1-at_{i}}$, see \cite[Lemma 3.9(c) and formula (35)]{FdBP_Zariski}
		\item \label{item:w-comparison} On $U_{X}\setminus D$, we have $\hat{w}_{i}-w_{i}=at$, see \cite[Lemma 3.9(c)]{FdBP_Zariski}.
		\item \label{item:u-comparison} On $U_{X}\setminus D$, we have $\hat{u}_{i}-u_i=\frac{u_{i}^{2}\log(1+a t_{i})}{1-u_{i}\log(1+a t_i)}$, see \cite[formula (34)]{FdBP_Zariski}
		\item \label{item:v-comparison} On $U_{X}\setminus D$, we have $v_{i}-\hat{v}_{i}=\frac{a t_{i}^{2}}{1+a t_{i}}+\frac{u_{i}^{2}\log(1+a t_{i})}{1-u_{i}\log(1+a t_i)}$, see \cite[Lemma 3.26(c)]{FdBP_Zariski}.
		\item \label{item:v-sigma-comparison} There is a bounded function $b\in \cC^{\infty}(U_{X,I}^{\circ})$ such that on $U_{X,I}^{\circ}$ we have $v_i-\hat{v}_i=\sigma_{i}b$, where $\sigma_{i}=t_{i}^{2}+t_iu_i^2$, see \cite[Lemmas 3.26(d) and 3.35(a)]{FdBP_Zariski}. 
		\item \label{item:dv-comparison} Fix $\epsilon\in (0,1)$. There are bounded functions $c,q\in \cC^{\infty}(U_{X,I}^{\circ})$, and a bounded 
		$1$-form $\gamma\in \Omega^{1}(U_{X,I}^{\circ})$ such that   
			$d(v_i-\hat{v}_{i})\
			= t_{i}\cdot c\, dv_{i}+t_i^{\epsilon}\cdot q\, dg+\sigma_{i}\cdot \gamma$ on $U_{X,I}^{\circ}$, 
		see \cite[Lemmas 3.26(e), 3.35(b)]{FdBP_Zariski}.
		\item \label{item:dvhat-fiberwise} On $U_{X,I}^{\circ}$ we have a fiberwise equality 
			$d\hat{v}_i=\sigma_{i}m_{i}(1+ct_i)\, ds_i+\sigma_i\gamma$ for some bounded $c\in \cC^{\infty}(U_{X,I}^{\circ})$, $\gamma\in \Omega^{1}(U_{X,I}^{\circ})$.
		\item \label{item:alpha-comparison} There is a smooth $1$-form $\beta\in \Omega^{1}(U_{X})$ such that on $U_X\setminus D$ we have $\hat{\alpha}_i=d\theta_{i}+\beta$. In particular, the $2$-form $d\hat{\alpha}_i$ extends to a smooth form on $X$, see \cite[Lemma 4.5(a),(d)]{FdBP_Zariski}.
	\end{enumerate}
\end{lema}
\begin{proof}
	To get \ref{item:t-comparison}--\ref{item:dv-comparison} and \ref{item:alpha-comparison}, we repeat the computation in the quoted part of \cite{FdBP_Zariski}, with the following modification. In loc.\ cit, we took another adapted chart $U_{X}'$, and on the intersection $U_{X}\cap U_{X}'$, we compared $t_i$, $w_i$ etc.\ with the corresponding functions $t_{i}'$, $w_i'$ etc.\ defined for $U_{X}'$. Each computation was based on the existence of smooth function $\lambda\colon \bar{U}_X\cap \bar{U}_X'\to (0,\infty)$ satisfying an equality $r_i'=\lambda r_i$. Now, we work on $U_{X}$ instead of $U_{X}\cap U_{X}'$, and use the analogous equality $\hat{r}_i=\lambda r_i$ given by Lemma \ref{lem:rihat}\ref{item:r_comparison}. Note that \ref{item:v-sigma-comparison} and \ref{item:dv-comparison} are also analogues of \cite[Lemma 3.35]{FdBP_Zariski}, with $\bar{v}_i$ replaced by $\hat{v}_i$. 
	
	For \ref{item:dvhat-fiberwise}, we use a fiberwise inequality $dg=0$. As in \cite[p.\ 218]{FdBP_Zariski}, from definitions \eqref{eq:basic_functions_v} we obtain a fiberwise equality $dv_i=\sigma_{i}m_i\, ds_i$. Substituting those two equalities to part \ref{item:dv-comparison} proves \ref{item:dvhat-fiberwise}.
\end{proof}

\begin{proof}[Proof of Proposition \ref{prop:AX_hat_smooth}]
	We adapt \cite[Proposition 3.33]{FdBP_Zariski} to the current setting, replacing the functions $\bar{v}_i$ with $\hat{v}_i$. First, we claim that the coordinates~\eqref{eq:AC-chart_hat} are $\cC^{1}$-compatible with coordinates~\eqref{eq:AC1-chart}, on the same domain. This follows from Lemma \ref{lem:comparison}\ref{item:v-sigma-comparison} comparing $\hat{v}_i$ and $v_i$, by the same argument as in the proof of \cite[Lemma 3.30]{FdBP_Zariski}. Since by \cite[Lemma 3.12]{FdBP_Zariski} the transition maps between charts \eqref{eq:AC1-chart} are $\cC^{1}$-diffeomorphisms, the same holds for the charts~\eqref{eq:AC-chart_hat}. To get a $\cC^{\infty}$ atlas, we need to show that these transition maps are smooth. Arguing as in \cite[p.\ 211]{FdBP_Zariski}, we see that it is enough to prove that all functions $\hat{v}_1,\dots,\hat{v}_N$, and the pullbacks of all smooth functions on $U_X$, are smooth in a chart \eqref{eq:AC-chart_hat}. This is the most technical part, where we compare 
		charts \eqref{eq:AC-chart_hat} and \eqref{eq:AC1-chart}, and the corresponding partial derivatives, using Lemma \ref{lem:comparison}. To do it in an efficient way, we introduce certain algebras of continuous functions on $\pi^{-1}(U_X)$ and use their formal properties, as follows.
	
	Fix an adapted chart $U_{X}$ with associate index set $S$, say, $S=\{1,\dots,k\}$. Let $U_{1}\subseteq \pi^{-1}(U_X)$ be a domain of the charts \eqref{eq:AC1-chart} and \eqref{eq:AC-chart_hat}. Fix $i\in S\setminus \{1\}$. Let $a=-m_{i}\log\frac{\hat{r}_i}{r_i}\in \cC^{\infty}(U_X)$ be the \enquote{transition} function from Lemma \ref{lem:comparison}. We define an algebra $\cA_{i}$ as in formula (40) loc.\ cit, with the set $\cT_{i}$ replaced by singleton $\{a\}$. This algebra keeps the properties listed in Lemma 3.43 loc.\ cit. In particular, it is closed under the operators $\d_{1},\dots,\d_k$, i.e., the first $k$ partial derivatives with respect to the coordinates \eqref{eq:AC1-chart}. Indeed, the only property of the set $\cT_{i}$ used in the proof is the fact that it consists of functions which are smooth with respect to the natural smooth structure on $U_X$, and small enough so that the denominators in Lemma \ref{lem:comparison}\ref{item:v-comparison} are bounded away from zero and infinity, which clearly holds for $\{a\}$. Moreover, we have $v_i-\hat{v}_i\in \cA_i$. Indeed, Lemma  \ref{lem:comparison}\ref{item:v-comparison} shows that it is enough to prove that $\cA_i$ contains the functions $a$, $(1+at_{i})^{-1}$, $\log(1+at_{i})$, $(1-u_{i}\log(1+at_i))^{-1}$, $t_i$ and $u_i$. The first four belong to $\cA_i$ by definition. We have $t_{i}\in \cA_i$ since $t_i$ belongs to both algebras $\cW_i$ and $\cR_i$ defined in \cite[formula (38)]{FdBP_Zariski}. Similarly, $u_i\in \cW_i$ by definition, and $u_i\in \cR_i$ by \cite[Lemma 3.41(a)]{FdBP_Zariski}.
	
	The modified definition of $\cA_{i}$ leads to the new definition of the ideal $\cI_{i}\subseteq \cA_{i}$ such that $\d_{j}\cA_i\subseteq \cI_j\cA_i$ for $j\neq i$ (by the proof of \cite[Lemma 3.43(f)]{FdBP_Zariski}), of the algebras $\cA$, $\cJ$, $\cP$, and of the matrix ring $\cM$. We introduce a matrix $\Psi$ in the same way as in formula (41) loc.\ cit, but with $\bar{v}_i$ replaced by $\hat{v}_i$. That is, we define $\Psi$ as the identity minus a $k\times k$ matrix whose first row is $[\d_{j}g]_{1\leq j\leq k}=[1,0,\dots,0]$, and the $i$-th row for $i=2,\dots,k$ is $[\d_{j}\hat{v}_i]_{1\leq j \leq k}$. Hence the first row of $\Psi$ is zero, and the $i$-th row is $[\d_{j}(v_{i}-\hat{v}_i)]_{1\leq j\leq k}$, where $v_{i}-\hat{v}_i\in \cA_i$. Since $\d_{i}\cA_i\subseteq \cA_i$,  $\d_{j}\cA_i\subseteq \cI_j\cA_i$ for $i\neq j$, and $1\in \cP$ by definition of $\cP$, we infer that all entries of $\Psi$ belong to $\cA_j\cdot \cP$, and the $(i,j)$-th entry for $i\neq j$, $2\leq i,j\leq k$ belongs to $\cI_j\cdot \cA_i\cdot \cP$. Hence $\Psi\in \cM$, i.e., the analogue of \cite[Lemma 3.47]{FdBP_Zariski} holds.
	
	As in \cite[p.\ 209]{FdBP_Zariski}, for a matrix $M\in \cM$ we define differential operators $\d_{j}^{M}$ by $[\d_{1}^{M},\dots,\d^{M}_k]^{\top}=M^{\top}\cdot [\d_{1},\dots,\d_{k}]^{\top}$. The analogue of \cite[Lemma 3.48(a)]{FdBP_Zariski} shows that the algebra $\cA_{i}\cdot \cW_{i}'\cdot \cP$ is closed under $\d_{j}^{M}$ for any $M\in \cM$. Here $\cW_{i}'$ is the smallest algebra containing $w_i$ and closed under $\d_{1},\dots,\d_{k}$, see \cite[p.\ 205]{FdBP_Zariski}. Since $1\in \cA_{i}\cap \cP$, we have $w_{i}\in \cA_{i}\cdot \cW_{i}'\cdot \cP$, so all higher order derivatives $\d^{M_1}_{j_1}\dots\d^{M_r}_{j_r} w_{i}$ exist and lie in $\cA_{i}\cdot \cW_{i}'\cdot \cP$, too, in particular, they are continuous.
	
	Let $\hat{\d}_{1},\dots,\hat{\d}_k$ be the first $k$ partial derivatives with respect to the chart \eqref{eq:AC-chart_hat}, i.e, $(g,\hat{v}_2,\dots,\hat{v}_k,\vartheta)$. Denote by $\cC^{\infty}(U_1)$ the algebra of functions which are smooth with respect to this chart. The beginning of the proof of \cite[Lemma 3.49]{FdBP_Zariski} shows that $\hat{\d}_1,\dots,\hat{\d}_{k}$ are uniform limits of the operators of the form $\d_{j}^{M}$ for $M\in \cM$. Thus all partial derivatives $\hat{\d}_{j_1}\dots\hat{\d}_{j_r} w_{i}$ are continuous, i.e., $w_{i}\in \cC^{\infty}(U_1)$.

	It follows that $w_{1}=1-\sum_{i=2}^{k}w_{i}\in \cC^{\infty}(U_1)$, too. Similarly, using \cite[Lemma 3.48(b)]{FdBP_Zariski} we conclude that the pullbacks of all functions which are smooth in the natural coordinates of $U_{X}$ are in $\cC^{\infty}(U_1)$. Moreover, $t=\exp(1-g^{-1})\in \cC^{\infty}(U_1)$, because $g$ is a coordinate of \eqref{eq:AC-chart_hat}. By Lemma \ref{lem:comparison}\ref{item:w-comparison}, for $j\in \{1,\dots,k\}$ we have $\hat{w}_{j}=w_{j}+at$ with $a\in \cC^{\infty}(U_X)$, so $\hat{w}_j\in \cC^{\infty}(U_1)$. Since on the domain $U_{1}$ of \eqref{eq:AC-chart_hat} we have $\hat{w}_1>0$, formula \eqref{eq:basic_functions_hat} gives $\hat{v}_{1}=\hat{t}_{1}-\hat{u}_1=t\hat{w}_{1}^{-1}-\eta(\hat{w}_1)\in \cC^{\infty}(U_1)$. Since $\hat{v}_2,\dots,\hat{v}_k$ are coordinates of \eqref{eq:AC-chart_hat}, we get $\hat{v}_j\in \cC^{\infty}(U_1)$ for all $j\in \{1,\dots,k\}$. For $j>k$ the function $\hat{r}_j$ is positive and smooth on $U_X$, so formulas \eqref{eq:basic_functions_hat} imply that $\hat{v}_j,\hat{w}_j\in \cC^{\infty}(U_1)$.
	
	We conclude that the charts \eqref{eq:AC-chart_hat} endow $A$ with a smooth structure such that the  functions $g$, $t$ and $\hat{v}_j$, $\hat{w}_j$ for all $j\in \{1,\dots, N\}$ are smooth, and so is the map $\pi\colon A\to X$. The angular function $\theta$ is locally a linear combination $\sum_{i} m_{i}\theta_{i}$ of the angular coordinates of \eqref{eq:AC-chart_hat}, so it is smooth on $A$. Eventually, by Lemma \ref{lem:comparison}\ref{item:alpha-comparison} the form $\hat{\alpha}_{j}$ is locally a sum of a coordinate form $d\theta_{j}$ and a pullback of a smooth form by a smooth map $\pi$, hence it is smooth, too, as needed.
\end{proof}

\subsection{The rounded dual complex}\label{sec:roundeddual}

The {\em rounded dual complex} $\tilde{\Delta}_{D}$ is the homeomorphic image of the dual complex $\Delta_{D}$ by the map $(-\eta,\dots,-\eta)$, see \cite[\textsection 3.2.2 and Figure 1]{FdBP_Zariski}. For a face $\Delta_{I}$ of the dual complex $\Delta_{D}$, we define its \emph{rounded} version $\tilde{\Delta}_{I}$ as the image of $\Delta_{I}$ by the same map. The set $\tilde{\Delta}_{I}$ is a manifold with boundary and corners, diffeomorphic to a standard simplex $\Delta_{I}$. Since the spaces $\tilde{\Delta}_{D}$, $\tilde{\Delta}_{I}$ are diffeomorphic to $\Delta_{D}$, $\Delta_{I}$ one can interchangeably think about each of them. Due to the nature of our smooth coordinate systems it is more convenient for us to formulate statements using the rounded versions. 

Recall that the stratification $X=\bigsqcup_{I}X_{I}^{\circ}$ lifts to a decomposition $A=\bigsqcup_{I} A_{I}^{\circ}$ given by $A_{I}^{\circ}=\pi^{-1}(X_{I}^{\circ})$. Lemma \ref{lem:product} below gives a handy description of each piece $A_{I}^{\circ}$ as a product of a real-oriented blowup $(X_{I}^{\circ})_{\log}$ of $X_{I}^{\circ}$, see below, and of the rounded face $\tilde{\Delta}_{I}$, with coordinates $\hat{v}_i=-\hat{u}_i=-\eta(\hat{w}_i)$, which are \enquote{rounded} barycentric coordinates $\hat{w}_i$; cf.\ \cite[Lemma 3.10(c) or Proposition 3.33(d)]{FdBP_Zariski}. This decomposition is depicted in the central part of Figure \ref{fig:ACampo}. In this picture, the angular coordinates of $(X_{I}^{\circ})_{\log}$ are not shown, so, for instance, the region labeled by $A_{1,2,3}^{\circ}$ is in fact the rounded face $\tilde{\Delta}_{1,2,3}$.
	
Let $\tau\colon X_{\log}\to X$ be the natural map from the Kato--Nakayama space, see diagram \eqref{eq:AX-diagram}. For $I\subseteq \{1,\dots, N\}$ put $(X_{I}^{\circ})_{\log}=\tau^{-1}(X_{I}^{\circ})$. The restriction of $\tau$ to $(X_{\log})_{I}^{\circ}$ is a smooth  $(\mathbb{S}^{1})^{\#I}$-bundle over $X_{I}^{\circ}$. Denoting by $\pi_{\log}$ the projection $A\to X_{\log}$, we have $\pi_{\log}^{-1}(X_{I}^{\circ})_{\log}=A_{I}^{\circ}$. 

\begin{lema}
	\label{lem:product}
	Let $I$ be a nonempty subset of $\{1,\dots,N\}$. Then on $A_{I}^{\circ}$ we have $\hat{u}_i=-\hat{v}_i\geq 0$ if $i\in I$ and $\hat{u}_i=0$ if $i\not \in I$. The map
	\begin{equation}\label{eq:product}
		(\pi_{\log},(\hat{v}_{i})_{i\in I})\colon A_{I}^{\circ}\to (X_{I}^{\circ})_{\log}\times \tilde{\Delta}_{I}
	\end{equation}
	is a diffeomorphism  of manifolds with boundary and corners. In particular, $\Int_{\d A} A_{I}^{\circ}$ is the subset of $\d A$ given by inequalities $\hat{u}_i> 0$ for $i\in I$ and $\hat{u}_i=0$ for $i\not\in I$.
\end{lema}
\begin{proof}
	On $A_{I}^{\circ}$ we have $\hat{t}_{i}=0$ for $i\in I$, so $\hat{v}_{i}=-\hat{u}_{i}=-u_{i}=-\eta(w_i)$ for $i\in I$, where $u_i$, $w_{i}$ are functions introduced in \eqref{eq:basic_functions_v}, for any adapted chart with $i$ in its associated index set. Hence the map \eqref{eq:product} indeed takes values in $(X_{I}^{\circ})_{\log}\times \tilde{\Delta}_{I}$. It is surjective by  \cite[Lemma 3.10(c)]{FdBP_Zariski}. It is injective because a point in the A'Campo space is uniquely determined by its image in $X_{\log}$ and the values of all functions $u_{i}$: since on $A_{I}^{\circ}$ we have $u_{i}=0$ for $i\not\in I$, these data are determined by the image of the map \eqref{eq:product}, as needed. To see that \eqref{eq:product} is a local diffeomorphism, fix a point $x\in A_{I}^{\circ}$ and a chart \eqref{eq:AC-chart_hat} around $x$, with coordinates $(\hat{v}_{j})_{j\in I\setminus \{i\}}$. At the image of $x$ in $\tilde{\Delta}_{I}$, we can choose a smooth chart with coordinates $-\eta(w_j)$ for $j\in I\setminus \{i\}$. In these coordinates, the map \eqref{eq:product} is the identity, as needed.
\end{proof}

\begin{remark}[Choices within the construction of the A'Campo space $A$]\label{rem:new_set}
	The A'Campo space $A$, viewed as a $\cC^{\infty}$ manifold together with smooth maps in diagram~\eqref{eq:AX-diagram-intro} and the collection of smooth functions $\hat{v}_i$, $\hat{w}_i$ and $1$-forms $\hat{\alpha}_i$, depends only on the initial function $f\colon X\to \C$ and on the choice of the distance functions $\hat{r}_i$, i.e., functions satisfying the conclusion of Lemma \ref{lem:rihat}. Moreover, the $\cC^{1}$-compatibility in Proposition \ref{prop:AX_hat_smooth} guarantees that different choices of distance functions yield A'Campo spaces which are equivalent by a $\cC^1$-diffeomorphism which is the identity at the boundary.

	We note that by the Ehresmann theorem, the space $A$ is $\cC^{\infty}$-diffeomorphic to $f^{-1}(\d \D_{\delta})\times [0,\delta)$, which depends only on $f$. However, \cite[Example 3.15]{FdBP_Zariski} shows that once we endow $A$ with a $\cC^{\infty}$ atlas \eqref{eq:AC-chart_hat} defined using another collection of distance functions $\hat{r}_i$, then the functions $\hat{v}_i$ are $\cC^1$, but no longer $\cC^2$ in general.  Since we use those functions to define the form $\omega_{q}^{\epsilon}$ and the subsequent Lagrangian fibration, see formulas  \eqref{eq:omega-s-intro}, \eqref{eq:omega_sharp} and \eqref{eq:submaximalquotient}, we fix one collection of distance functions $\hat{r}_i$, and use the  functions $\hat{v}_i$ which are smooth in the $\cC^{\infty}$ structure on $A$ defined by $\{\hat{r}_i\}_{i}$. 
\end{remark}

\section{Fiberwise K\"ahler forms}\label{sec:form}

Let $f\colon X\to \C$ be as in Section \ref{sec:prelim}, and let $A$ be the A'Campo space of $f$, equipped with the smooth structure given by charts \eqref{eq:AC-chart_hat}. In this section, we introduce exact forms $\omega^{\sharp}$ and $\omega^{\flat}$ appearing in formula \eqref{eq:omega-s-intro}, and prove that the resulting forms $\omega_{q}^{\epsilon}$ are indeed fiberwise symplectic, see Proposition \ref{prop:omega}.

The form $\omega^{\sharp}$ is an analogue of the form $\omega_{E}$ introduced in formula (44) of \cite{FdBP_Zariski}. It is designed so that for any sufficiently small $\epsilon>0$, the form  $\pi^{*}\omega_{X}+\epsilon \omega^{\sharp}$ extends to a fiberwise symplectic form on $A$, and in particular yields a well-behaved monodromy at radius zero. The additional feature of $\omega^{\sharp}$ with respect to $\omega_{E}$ from loc.\ cit.\ is its fiberwise $J$-compatibility, which we prove in Lemma \ref{lem:form-ACampo}. In turn, the form $\omega^{\flat}$ is designed so that the induced metric degenerates at radius zero to the euclidean one on $\Delta_{\cS}$,  as required by Theorem \ref{theo:CY}\ref{item:CY_GH}. It is inspired by the semi-flat models surveyed in \cite[\textsection 2.1]{Li_survey}. 

We identify $A\setminus \d A$ with $X\setminus D$ via the diffeomorphism $\pi|_{A\setminus \d A}\colon A\setminus \d A\to X\setminus D$, and denote by $J$ the induced complex structure. We fix an open set $W_{X}$ as in Section \ref{sec:distance_functions}, i.e., such that $f|_{\bar{W}_{X}}\colon \bar{W}_X\to \D_{\delta}$ is proper, has connected fibers, and restricts to a submersion over $\D_{\delta}^{*}$. We put $W=\pi^{-1}(W_X)$ and write  $W_{z}$ for the fiber $f_{A}^{-1}(z)\cap W$. We also identify $\D_{\delta}^{*}$ with its preimage in $\D_{\delta,\log}$.

\subsection{The \enquote{A'Campo form} $\omega^{\sharp}$}

Recall that by Proposition \ref{prop:AX_hat_smooth}, for each $i\in \{1,\dots, N\}$ we have a smooth function $\hat{v}_i\in \cC^{\infty}(A)$ and a smooth angular $1$-form $\hat{\alpha}_i\in \Omega^{1}(A)$, introduced in formula \eqref{eq:basic_functions_hat}. We now define smooth forms on $A$: 
\begin{equation}\label{eq:omega_sharp}
	\lambda^{\sharp}=\sum_{i=1}^{N}\hat{v}_i\cdot \hat{\alpha}_i,\quad
	\omega^{\sharp}=d\lambda^{\sharp}.
\end{equation}

\begin{lema}\label{lem:form-ACampo}
	There is a smooth function $\phi\colon W\setminus \d A\to \R$ satisfying a fiberwise equality $\lambda^{\sharp}=d^{c}\phi$. In particular, $\omega^{\sharp}$ is fiberwise $J$-compatible, i.e., we have an equality $\omega^{\sharp}(J\sdot,J\sdot)=\omega^{\sharp}(\sdot,\sdot)$ on the tangent space to the fiber $W_{z}$ for all $z\in \D_{\delta}^{*}$.
\end{lema}
\begin{proof}
	Clearly, it is enough to prove the first assertion. If $D$ has only one component $D_1$ then $\hat{s}_1=-m_{1}t^{-1}$ is fiberwise constant, hence $\lambda^{\sharp}=-\hat{v}_1\cdot d^{c}\hat{s}_1$ is fiberwise zero. Therefore, we can assume that $D$ has at least two components. Fix $i\in \{1,\dots, N\}$.
	
	By definition \eqref{eq:basic_functions_hat}, the function $\hat{v}_i\colon W_{X}\setminus D\to \R$ is equal to the composition $\check{v}_i\circ (\hat{s}_i,t)$, where 
	\begin{equation*}
		\check{v}_i\colon T\to \R\quad \mbox{is defined as}\quad 
		\check{v}_i(\check{s},\check{t})=-(m_i\check{s})^{-1}-\eta(-m_{i}\cdot \check{s}\check{t}),
	\end{equation*}
	and $T\subseteq \R^2$ is the image of the map $(\hat{s}_i,t)\colon W_{X}\setminus D\to \R^2$. 
	
	We claim that $T$ is a union of some horizontal segments (or half-lines), ending on the vertical line $\check{s}=-1$. By Lemma \ref{lem:rihat}\ref{item:r_Ri} we have $\hat{s}_{i}\leq -1$, and the equality $\hat{s}_i=-1$ holds on $X\setminus R_{i}$, where $R_i$ is the open set chosen in Lemma \ref{lem:rihat}. Since $D$ has at least two components, the set $X\setminus D_i$ contains a nonempty set $R_{j}^{\circ}$ for some $j\neq i$, which meets each fiber by Lemma \ref{lem:rihat}\ref{item:r_covering}. Thus the equality $\hat{s}_i=-1$ is attained somewhere on each fiber, and the claim follows since each fiber is connected.
	
	Therefore, we can define a smooth  function $\check{\phi}_i\colon T\to \R$ by the formula $\check{\phi}_{i}(\check{s},\check{t})=\int_{\check{s}}^{-1} \check{v}_i(s,\check{t})\, ds$. Fundamental theorem of calculus gives $\frac{\d}{\d \check{s}}\check{\phi}_{i}=-\check{v}_i$. We define a smooth function  $\phi_{i}\colon W_X\setminus D\to \R$ as $\phi_{i}\de \check{\phi}_{i}\circ (\hat{s}_i,t)$. Since $\frac{\d}{\d \check{s}}\check{\phi}_{i}=-\check{v}_i$, we have a fiberwise equality $d\phi_{i}=-\check{v}_i\circ (\hat{s}_i,t)\, d\hat{s}_i=-\hat{v}_i\, d\hat{s}_i$, so $d^{c}\phi_{i}=-\hat{v}_i\, d^{c}\hat{s}_i=\hat{v}_i\hat{\alpha}_i$ by definition \eqref{eq:basic_functions_hat} of $\hat{\alpha}_i$. Eventually, we put $\phi=\sum_{i=1}^{N}\phi_{i}$, so we have a fiberwise equality $d^{c}\phi=\sum_{i=1}^{N}\hat{v}_i\hat{\alpha}_i=\lambda^{\sharp}$, as needed.
\end{proof}

\subsection{The \enquote{semi-flat form} $\omega^{\flat}$}

As in Section \ref{sec:dual_complex}, we fix any subset $\cS\subseteq \{1,\dots, N\}$: in the setting of Calabi--Yau degenerations, it will correspond to essential divisors. Now, we define smooth forms on $A$:
\begin{equation}\label{eq:omega_flat}
	\lambda^{\flat}=-\sum_{i\in \cS} m_{i}\, \hat{w}_i\cdot \hat{\alpha}_i,
	\quad
	\omega^{\flat}=d\lambda^{\flat}.
\end{equation}

\begin{lema}\label{lem:form-flat}
	Put $\phi^{\flat} = -\frac{1}{2t}\sum_{i\in \cS} \hat{w}_i^2$. Then the form $d^{c}\phi^{\flat}$ extends to a smooth form on $A$, and we have an equality $\lambda^{\flat}=d^{c}\phi^{\flat}+\frac{1}{2}\sum_{i\in \cS} \hat{w}_{i}^{2}\cdot d\theta$. In particular, we have a fiberwise equality $\lambda^{\flat}=d^{c}\phi^{\flat}$, so the form $\omega^{\flat}$ is fiberwise $J$-compatible.
\end{lema}
\begin{proof}
	Using identities  $\hat{w}_i=-m_{i}\hat{s}_{i}t$ and $-d^{c}\hat{s}_i=\hat{\alpha}_i$ from \eqref{eq:basic_functions_hat}, we get
	\begin{equation*}
		d^{c}\hat{w}_i=-m_{i}\, d^{c}(\hat{s}_it)=m_{i}t\, \hat{\alpha}_i - m_{i}\hat{s}_i d^{c}t=
		m_{i}t\, \hat{\alpha}_i+\tfrac{1}{t}\hat{w}_i\, d^{c}t. 
	\end{equation*}
	Put $\phi_{i}^{\flat}=-\tfrac{1}{2t}\hat{w}_i^2$, so $\phi^{\flat}=\sum_{i\in \cS} \phi_{i}^{\flat}$. The above equality gives 
	\begin{equation*}
		d^{c} \phi_{i}^{\flat}=
		-\tfrac{1}{t}\hat{w}_i\, d^{c}\hat{w}_i
		+\tfrac{1}{2t^2} \hat{w}_i^2\, d^{c}t=
		-m_{i} \hat{w}_i\hat{\alpha}_i
		-\tfrac{1}{2t^2}\hat{w}_i^2\, d^{c}t.
	\end{equation*}
	Moreover, we have an equality
	\begin{equation*}
		d^{c}t=-t^{2}d\theta.
	\end{equation*} Indeed, putting $s=\log |f|$ we see that $d^{c}s=-d\theta$ and $t=-s^{-1}$, so $d^{c}t=-d^{c}s^{-1}=s^{-2}d^{c}s=-t^{2}d\theta$, as claimed. Substituting this to the previous equality, we conclude that
	\begin{equation*}
		d^{c} \phi_{i}^{\flat}=-m_{i}\, \hat{w}_i\alpha_i+\tfrac{1}{2} \hat{w}_i^2\, d\theta.
	\end{equation*}
	By Proposition \ref{prop:AX_hat_smooth}\ref{item:AX-gsmooth},\ref{item:AX-vbar-smooth}, the functions $\theta$, $\hat{w}_i$ and the form $\hat{\alpha}_i$ are smooth in $A$, hence so is $d^{c}\phi_{i}^{\flat}$. 
	The result follows after taking a sum over all $i\in \cS$.
\end{proof}	

\subsection{The form \texorpdfstring{$\omega_{q}^{\epsilon}$}{omega-q} is fiberwise K\"ahler}

Recall that we have fixed a K\"ahler form $\omega_{X}$ on $X$, and introduced $2$-forms $\omega^{\sharp}$ and $\omega^{\flat}$ by formulas \eqref{eq:omega_sharp} and \eqref{eq:omega_flat}. We now  introduce a family of smooth $2$-forms on $A$ by formula \eqref{eq:omega-s-intro}, which reads as
\begin{equation}\label{eq:omega-s}
	\omega_{q}^{\epsilon}=\pi^{*}\omega_{X}+\epsilon \cdot ( q\cdot \omega^{\sharp}+(1-q) \cdot \omega^{\flat}), \quad q\in [0,1].
\end{equation}
The next proposition summarizes key properties of this family. 

\begin{prop}\label{prop:omega}
	There is an $\epsilon_0>0$ and $\delta'>0$ such that the following hold.
	\begin{enumerate}
		\item\label{item:Kahler} For every $q\in [0,1]$ and $\epsilon\in [0,\epsilon_0]$, the form $\omega_{q}^{\epsilon}$ is K\"ahler on the fiber $W_{z}$ for every $z\in \D_{\delta'}^{*}$
		\item\label{item:symplectic} For every $q\in (0,1]$ and $\epsilon\in (0,\epsilon_0]$, the form $\omega_{q}^{\epsilon}$ is symplectic on the fiber $W_{z}$ for every $z\in \D_{\delta',\log}$.
		\item\label{item:symplectic_zero} Let $A^{\textnormal{int}}_{\cS}$ be the union of $\Int_{\d A} A_{I}^{\circ}$ for all faces $\Delta_{I}$ of $\Delta_{\cS}$. Then for every $\epsilon\in (0,\epsilon_0]$ there is a neighborhood $V_{\epsilon}$ of $W\cap A^{\textnormal{int}}_{\cS}$ in $A$ such that the restriction $\omega_{0}^{\epsilon}|_{V_{\epsilon}}$ is fiberwise symplectic.
	\end{enumerate}
\end{prop}
The proof of Proposition \ref{prop:omega} follows the steps of the proof  \cite[Proposition 4.2]{FdBP_Zariski}, which can be outlined as follows. We argue separately at positive radius (Proposition \ref{prop:omega}\ref{item:Kahler}) and at radius zero. The heart of the matter lies at positive radius. There, we write $\omega^{\epsilon}_{q}$ as a sum of $q_{j}\pi^{*}\omega_{X}+\epsilon \omega_{j}$ for $j=1,2,3$, see formula \eqref{eq:omega-parts}, so that: $\sum_{j}q_{j}=1$; $\omega_1$ and $\omega_2$ contain the contributions from $\omega^{\sharp}$ and  $\omega^{\flat}$, respectively; and $\omega_3$ gathers terms which are bounded in the natural coordinates of $X$. We prove that each form tames $J$, see assertions \ref{item:claim_sharp}--\ref{item:claim_rest}. To treat $\omega_1$, we write  $\omega_1=q_1\pi^{*}\omega_{X}+\sum_{i} dv_i\wedge d\theta_i+\dots$, where $\dots$ stand for extra terms coming from comparison formulas in Lemma \ref{lem:comparison}. We prove that these terms can be compensated: in the $D_i$-directions, by $dv_i\wedge d\theta_i$, and in the other directions, by the K\"ahler form $\omega_X$. A formal way to see this compensation is to perform Gaussian diagonalization. For $\omega_2$ we argue in the same way, using $-dw_i\wedge d\theta_i$ instead of $dv_i\wedge d\theta_i$. For $\omega_3$, we use an elementary compactness argument in the natural coordinates of $X$. Once Proposition \ref{prop:omega}\ref{item:Kahler} is proved, we argue at radius zero as in \cite[Proposition 4.12]{FdBP_Zariski}, using local formulas summarized in Lemma \ref{lem:TdA}.
	
We now follow the above outline. The starting point is an analogue of \cite[Formula (53)]{FdBP_Zariski}.

\begin{lema}\label{lem:53}
	Let $U_{X}$ be an adapted chart with associated index set $I$, and let $U=\pi^{-1}(U_X)$. Then on every fiber $f_{A}^{-1}(r,\theta)\cap U$ for $r>0$ we have equalities
	\begin{enumerate}
		\item\label{item:omega_sharp_local} $\omega^{\sharp}=\sum_{i\in I} \sigma_i \cdot (m_i(1+c_it_i)\, d s_i\wedge d\theta_{i}+\gamma_i\wedge d\theta_{i}+m_i(1+c_it_i)\, d s_i\wedge \beta_{i})+\check{\omega}$,
		\item\label{item:omega_flat_local}
		$\omega^{\flat}=\sum_{i\in I\cap \cS} t\cdot (m_i^2\, ds_i\wedge d\theta_{i}+da_i\wedge d\theta_{i}+ d s_i\wedge \beta_{i})+\tilde{\omega}$,
	\end{enumerate}
	for some bounded functions $a_i$, $c_i$ and bounded forms $\beta_{i}$, $\gamma_{i}$, $\check{\omega}$, $\tilde{\omega}$ on $U_{X}\setminus D$. 
\end{lema}
\begin{proof}
	\ref{item:omega_sharp_local} By definition \eqref{eq:omega_sharp} of $\omega^{\sharp}$ we have
	\begin{equation*}
		\omega^{\sharp}=\sum_{i\in I}d\hat{v}_i\wedge \hat{\alpha}_i+\check{\omega}',
	\end{equation*}
	where $\check{\omega}'=\sum_{i\not\in I}d\hat{v}_i\wedge \hat{\alpha}_i+\sum_{i=1}^{N}\hat{v}_i d\alpha_{i}$. Lemma \ref{lem:comparison}\ref{item:dvhat-fiberwise},\ref{item:alpha-comparison} implies that the forms $d\hat{v}_i$, $\hat{\alpha}_{i}$ for $i\not\in I$ and $d\hat{\alpha}_i$ for all $i$ are bounded on $U_X\setminus D$, cf.\ \cite[Lemma 4.5]{FdBP_Zariski}. Thus the form $\check{\omega}'$ is bounded on $U_{X}\setminus D$. Now part \ref{item:omega_sharp_local} follows by substituting the fiberwise formulas for $d\hat{v}_i$ and $\hat{\alpha}_i$ from  Lemma \ref{lem:comparison}\ref{item:dvhat-fiberwise},\ref{item:alpha-comparison}.
	
	\ref{item:omega_flat_local}  By definition \eqref{eq:omega_flat} of $\omega^{\flat}$ we have 
	\begin{equation*}
		\omega^{\flat}=-\sum_{i\in I\cap \cS}m_i\, d\hat{w}_{i}\wedge \hat{\alpha}_i+\tilde{\omega}',
	\end{equation*}
	where $\tilde{\omega}'=-\sum_{i\in \cS\setminus I}m_i\, d\hat{w}_i\wedge \hat{\alpha}_{i}-\sum_{i\in \cS} m_i\, \hat{w}_id\hat{\alpha}_i$, so as before, the form $\tilde{\omega}'\in \Omega^{2}(U_X\setminus D)$ is bounded. By Lemma \ref{lem:comparison}\ref{item:w-comparison} we have $\hat{w}_i=w_{i}+ta_{i}=-m_{i}s_{i}t+ta_{i}$ for some smooth function $a_{i}\in \cC^{\infty}(U_X)$. Combining it with the formula for $\hat{\alpha}_i$ in Lemma \ref{lem:comparison}\ref{item:alpha-comparison}, we get part \ref{item:omega_flat_local}.
\end{proof}

\begin{proof}[Proof of Proposition \ref{prop:omega}\ref{item:Kahler}]
	Lemmas \ref{lem:form-ACampo} and \ref{lem:form-flat} imply that $\omega_{q}^{\epsilon}|_{A\setminus \d A}$ is fiberwise $J$-compatible. Thus it remains to prove that $\omega_{q}^{\epsilon}$ tames $J$ on each fiber $W_{z}$, $z\in \D_{\delta'}^{*}$. This result is analogous to \cite[Proposition 4.2]{FdBP_Zariski}. We now explain how to adapt its proof.

	Covering $\bar{W}_X$ with finitely many adapted charts, we see that it is enough to work locally, in an adapted chart $U_X$. Say that its associated index set is $\{1,\dots,k\}$. Reordering the components of $D$ if needed, we can assume that $\cS\cap \{1,\dots, k\}=\{1,\dots,l\}$. We group terms from Lemma \ref{lem:53}  follows:
	\begin{equation}\label{eq:omega-parts}\begin{split}
		\omega_{1}&\de \sum_{i=1}^{k} \sigma_i \cdot (m_i(1+c_it_i)\, d s_i\wedge d\theta_{i}+\gamma_i\wedge d\theta_{i}+m_i(1+c_it_i)\, d s_i\wedge \beta_{i}),\\
		\omega_{2}&\de \sum_{i=1}^{l} t\cdot (m_i^2\,ds_i\wedge d\theta_{i}+da_i\wedge d\theta_{i}+d s_i\wedge \beta_{i}), \\
		\omega_{3,q}&\de q\cdot \check{\omega}+(1-q)\cdot \tilde{\omega}\quad \mbox{for } q\in [0,1].
	\end{split}\end{equation}
	For a $2$-form $\omega$, we denote by $\omega^{J}$ the symmetric part of the form $\omega(\sdot,J\sdot)$. With this notation, we can split the formula \eqref{eq:omega-s} defining $\omega_{q}^{\epsilon}$ as follows.
	\begin{equation*}
		(\omega_{q}^{\epsilon})^{J}=q\cdot (\tfrac{1}{2}\pi^{*}\omega_{X}^{J}+\epsilon \omega_{1}^{J})+(1-q)\cdot (\tfrac{1}{2}\pi^{*}\omega_{X}^{J}+\epsilon\omega_{2}^{J})+(\tfrac{1}{2}\pi^{*}\omega_{X}^{J}+\epsilon \omega_{3,q}^{J}).
	\end{equation*}
	Therefore, to prove Proposition \ref{prop:omega}\ref{item:Kahler} it is enough to prove the following three assertions. 
	\begin{enumerate}[(i)]
		\item \label{item:claim_sharp}There exists $\epsilon_1>0$ such that for every $\epsilon\in [0,\epsilon_1]$ the form $\frac{1}{2}\pi^{*}\omega_{X}^{J}+\epsilon \omega_{1}^{J}$ is positive definite.
		\item \label{item:claim_flat} There exists $\epsilon_2>0$ such that for every $\epsilon\in [0,\epsilon_2]$ the form $\frac{1}{2}\pi^{*}\omega_{X}^{J}+\epsilon\omega_{2}^{J}$ is positive definite.
		\item \label{item:claim_rest} There exists $\epsilon_3>0$ such that for every $\epsilon\in [0,\epsilon_3]$ and every $q\in [0,1]$ the form $\frac{1}{2}\pi^{*}\omega_{J}+\epsilon \omega_{3,q}^{J}$ is positive definite.
	\end{enumerate}
	To see assertion \ref{item:claim_rest}, recall from Lemma \ref{lem:53} that $\omega_{3,q}\in \Omega^{2}(U_X\setminus D)$ is a family of forms which are bounded in the natural coordinates on $U_X$, smoothly parametrized by $q\in [0,1]$. Since the interval $[0,1]$ is compact, this family is uniformly bounded, so for sufficiently small $\epsilon>0$, positive definiteness of $\pi^{*}\omega_{X}^{J}$ implies that the sum  $\frac{1}{2}\pi^{*}\omega_{X}^{J}+\epsilon \omega_{3,q}^{J}$ is positive definite, too. 
	
	To see assertion \ref{item:claim_sharp}, note that the definition of $\omega_{1}$ is entirely analogous to formula (53) in \cite{FdBP_Zariski}. Repeating the proof of Proposition 4.3 loc.\ cit.\ from this formula onward, we get assertion \ref{item:claim_sharp}.
	
	It remains to prove \ref{item:claim_flat}. The definition of $\omega_{2}$ is still similar to the aforementioned  formula (53). We now explain how to adapt the proof of Proposition 4.3 loc.\ cit.\ to this setting. 
	
	Recall that $d^{c}s_i=-d\theta_{i}$ and $d^{c}\theta_{i}=ds_i$. Applying these rules to the definition \eqref{eq:omega-parts} of $\omega_{2}$, we get
	\begin{equation}\label{eq:omega2J}
		\omega_{2}^{J}=t\cdot \sum_{i=1}^{l} m_{i}^2\, ds_{i}^{2}+m_i^2\, d\theta_{i}^{2}+ds_{i} \cdot \beta_{i}'+d\theta_{i}\cdot \beta_{l+i}'
	\end{equation}
	for some $1$-forms $\beta_{i}'$ bounded with respect to the natural coordinates of $U_X\setminus D$.
	
	Consider a basis $\{ds_1,\dots,  ds_l;d\theta_1,\dots,d\theta_l,dx_{1},\dots,dx_{2(n+1-l)}\}$ of the cotangent space $T^{*}(U\setminus \d A)$, where $x_{j}$ stand for real and imaginary parts of the remaining coordinates of $U_{X}$. Let $\{\nu_1,\dots,\nu_{2l},\allowbreak \xi_{1},\dots,\xi_{2(n+1-l)}\}$ be the corresponding dual basis of $T(U\setminus \d A)$. Let $||\sdot ||$ be the maximum norm in $T(U \setminus \d A)$ with respect to this basis, and let $||\cdot ||_{X}$ be a maximum norm in $TU_X$ with respect to the natural coordinates. We now list some properties of this basis, which, as we will see, are stable under the Gaussian diagonalization.
	
	As in \cite[Lemma 4.10(e),(f)]{FdBP_Zariski}, we claim that the following hold.
	\begin{enumerate}
		\item\label{item:nu-small} For $i\in \{1,\dots, 2l\}$, we have $||\pi_{*}\nu_{i}||_{X}\rightarrow 0$ as $r_1,\dots,r_k\rightarrow 0$.
		\item\label{item:xi-large} Let $\Xi$ be the subspace spanned by $\{\xi_1,\dots,\xi_{2(n+1-l)}\}$. Then there is a constant $K_0>0$ such that for every $\xi\in \Xi$ we have $||\pi_{*}\xi||_{X}\geq K_{0}||\xi||$.
	\setcounter{foo}{\value{enumi}}
	\end{enumerate}
	Indeed, \ref{item:nu-small} holds since $\nu_{i}=\frac{\d}{\d s_i}=r_i\frac{\d}{\d r_i}$, and \ref{item:xi-large} holds because $\pi_{*}|_{\Xi}$ is an isomorphism onto its image.

	The matrix of $\omega_{2}^{J}$ in this basis has a block form
	\begin{equation*}
		t\cdot \left[\begin{matrix}
			A & Q \\ Q & P
		\end{matrix}\right].
	\end{equation*}
	We claim that, after possibly shrinking the chart $U_{X}$, the entries of each block have the following properties, cf.\ \cite[Lemma 4.10(a)-(d)]{FdBP_Zariski}.
	\begin{enumerate}\setcounter{enumi}{\value{foo}}
		\item\label{item:diagonal} The diagonal terms of $A$ are bounded from below by $\frac{1}{4}$.
		\item\label{item:off-diagonal} The off-diagonal terms of $A$ converge to $0$ as $r_1,\dots,r_k\rightarrow 0$.
		\item\label{item:PQ} All entries of the blocks $Q$ and $P$ are bounded functions on $U\setminus \d A$.
	\end{enumerate}
	These properties follow easily from formula \eqref{eq:omega2J}. Indeed, the $(i,j)$-term of $A$ for $i,j\in \{1,\dots, l\}$ is the function $\delta_{i}^{j}\bar{m}^2_{i}+\beta_{i}'(\nu_{j})+\beta_{j}'(\nu_{i})$, where $\delta_{i}^{j}$ is the Kronecker delta, and $\bar{m}_{i}=m_i$ if $i\leq l$ and $\bar{m}_{i}=m_{i-l}$ if $i>l$. Property \ref{item:nu-small} and boundedness of the $1$-forms $\beta_{i}'$ imply that $\beta_{i}'(\nu_{j})\rightarrow 0$ as we approach the origin of the chart, so after possibly shrinking $U_{X}$, we get properties \ref{item:diagonal} and \ref{item:off-diagonal}. To see property \ref{item:PQ}, we note that $P=0$, and the $(i,j)$-term of $Q$ for $i\in \{1,\dots,2l\}$, $j\in \{1,\dots,2(n+1-k)\}$ is $\beta_{i}'(\xi_{j})$. Since the $1$-form $\beta_{i}'$ and the vector field $\xi_{j}$ are bounded with respect to natural coordinates on $U_X$, so is the function $\beta_{i}'(\xi_{j})$, as needed.
	\smallskip
	
	All properties \ref{item:nu-small}--\ref{item:PQ} are preserved by the first $2l$ steps of the Gaussian diagonalization. Indeed, at each step we add to $\nu_{i}$ or $\xi_{j}$ a linear combination of vectors $\nu_{1},\dots,\nu_{2l}$, with bounded coefficients, so properties \ref{item:nu-small}, \ref{item:xi-large} are kept. In the matrix, to each entry of $A$ and $Q$ we add a function converging to $0$ as $r_1,\dots,r_k\rightarrow 0$; and to each entry of $P$ we add a bounded function. Hence after possibly shrinking $U_X$ at each step, we see that properties \ref{item:diagonal}--\ref{item:PQ} are preserved, too. 
	
	This way, we get a basis of $T(U\setminus \d A)$ satisfying \ref{item:nu-small}, \ref{item:xi-large}, such that the matrix of $\omega_{2}^{J}$ in this basis is 
	\begin{equation*}
		t\cdot \left[\begin{matrix}
			A' & 0 \\ 0 & P'
		\end{matrix}\right],
	\end{equation*}
	where $A'$ is a diagonal matrix with entries bounded from below by $\frac{1}{4}$; and all entries of $P'$ are bounded. 
	
	Fix a nonzero vector $v\in T(U\setminus \d A)$, and decompose it as $v=\nu+\xi$ according to the above splitting, i.e., so that $\nu,\xi$ belong to the linear span of the first $2l$ and last $2(n+1-l)$ vectors of the basis in which $\omega_{2}^{J}$ has the above block-diagonal form. We have 
	\begin{equation*}
		\omega_{2}^{J}(\nu,\nu)\geq t\cdot \tfrac{1}{4}||\nu||^2,\quad |\omega_{2}^{J}(\xi,\xi)|\leq t\cdot K||\xi||^2
	\end{equation*}
	for some constant $K>0$, uniform on $U\setminus \d A$. 
	
	Now, we repeat the proof of \cite[Proposition 4.3]{FdBP_Zariski}, starting on p.\ 224 loc.\ cit, with inequalities (58) and (59) replaced by the above ones; and the forms $\pi^{*}\omega_{X}^{J}$, $\omega_{E}$ replaced by $\frac{1}{2}\pi^{*}\omega_{X}^{J}$, $\omega_2$. The proof goes through without any further changes, except for the second-to-last line. There, in the inequality bounding $C$, the number $\epsilon$ is now multiplied by $t$, i.e., the modified inequality reads as $C\geq (\frac{1}{2}K_1-\epsilon tK)||\xi||^2$. Because $t=\frac{-1}{\log|f|}<\frac{-1}{\log\delta}$ is bounded, for sufficiently small $\epsilon>0$ we get $C\geq 0$, which ends the proof of \ref{item:claim_flat}. We note that instead of shrinking $\epsilon>0$ one can also get the required inequality $C\geq 0$ by keeping $\epsilon$ in a fixed bounded interval $(0,\epsilon_0]$, and shrinking $\delta>0$.
\end{proof}

We now move on to the proof of Proposition \ref{prop:omega}\ref{item:symplectic},\ref{item:symplectic_zero}, that is, fiberwise non-degeneracy of the forms $\omega_{q}^{\epsilon}$ for $q>0$ on $W\cap \d A$; and of $\omega_{0}^{\epsilon}$ on $W\cap A_{\cS}^{\textnormal{int}}$. We follow the proof of  \cite[Proposition 4.12]{FdBP_Zariski}, with $\omega^{\epsilon}_{q}$ playing the role of $\omega_{A}^{\epsilon}$. First, we introduce some notation as in Lemma 4.13 loc.\ cit.

Fix a smooth chart \eqref{eq:AC-chart_hat}, call it $U$. Reordering the components of $D$ if needed, we can assume that the associated index set of this chart is $S=\{1,\dots, k\}$, and the index $i$ excluded in the formula \eqref{eq:AC-chart_hat} is $1$. Fix a subset $I\subseteq S$ containing $1$, say $I=\{1,\dots, l\}$. Then as in \cite[Lemma 4.13]{FdBP_Zariski}, we see that at each point of $U_{I}^{\circ}\de U\cap A_{I}^{\circ}$ we have smooth coordinates
\begin{equation*}
	(\hat{v}_2,\dots,\hat{v}_l,s_{l+1},\dots,s_{k},\theta_{1},\dots,\theta_{k},z_{i_1},\dots,z_{i_{n+1-k}}).
\end{equation*}	
These coordinates provide a splitting $TU_{I}^{\circ}=\cV_{I}\oplus \Theta_{I} \oplus \cZ_{I}$, where $\cV_{I}=\sspan\{\frac{\d}{\d\hat{v}_2},\dots,\frac{\d}{\d\hat{v}_{l}}\}$, $\Theta_{I}=\sspan \{\frac{\d}{\d \theta_{1}},\dots,\frac{\d}{\d\theta_{l}}\}$, and $\pi_{*}$ maps $\cZ_{I}$ isomorphically to the tangent bundle to the stratum $\pi(U)\cap X_{I}^{\circ}$.

We have the following analogue of \cite[Lemma 4.14]{FdBP_Zariski}.

\begin{lema}\label{lem:TdA}
	On $U_{I}^{\circ}$, the following holds.
	\begin{enumerate}
		\item\label{item:ort-Z} For every $\beta\in \Omega^{*}(U_X \cap X_{I}^{\circ})$ we have $\pi^{*}\beta|_{\cV_{I}\oplus \Theta_{I}}=0$. 
		\item\label{item:ort-Z-vj} For every $j>l$ we have $d\hat{w}_j=0$ and $d\hat{v}_j|_{\cV_I\oplus \Theta_I}=0$.
		\item\label{item:ort-alpha} For every $i\in \{1,\dots, k\}$, $j\in \{1,\dots, l\}$ we have $\hat{\alpha}_{i}|_{\cV_{I}}=0$ and $\hat{\alpha}_{i}(\tfrac{\d}{\d \theta_{j}})=\delta_{i}^{j}$.
		\item\label{item:dv1} We have $d\hat{v}_{1}=-\sum_{i=2}^{l}\frac{\zeta(u_i)}{\zeta(u_1)}\, d\hat{v}_i$, where $\zeta=(\eta^{-1})'\colon [0,1]\ni s\mapsto s^{-2}e^{1-s^{-1}}\in [0,\infty)$.
		\item\label{item:dvi} For every $i\in \{1,\dots, l\}$ we have  $d\hat{w}_{i}|_{\Theta_{I}\oplus \cZ_{I}}=0$ and $d\hat{v}_{i}|_{\Theta_{I}\oplus \cZ_{I}}=0$. 
		\item\label{item:omega-dtheta} For every $i\in \{1,\dots, l\}$ we have $\omega_{q}^{\epsilon}(\sdot,\tfrac{\d}{\d\theta_{i}})=\epsilon c_{i}\, d\hat{v}_i$, where  $c_{i}=q+(1-q)\cdot m_i \zeta(u_i)$ if $i\in \cS$ and $c_{i}=q$ if $i\not\in \cS$.
	\end{enumerate}
\end{lema}
\begin{proof}
	The computations are exactly the same as in \cite[Lemma 4.14]{FdBP_Zariski}. The additional assertions about the function $\hat{w}_i$ follow from the fact that on each piece $A_{J}^{\circ}$ of $\d A$, we have $\hat{w}_i=0$ if $i\not\in J$ and $\hat{v}_i=-\eta(\hat{w}_i)$ otherwise; so $\hat{w}_i$ is either zero or a re-parametrization of $\hat{v}_i$.
\end{proof}

\begin{proof}[Proof of Proposition \ref{prop:omega}\ref{item:symplectic},\ref{item:symplectic_zero}]
	First, as in \cite[Lemma 4.15]{FdBP_Zariski} we prove that, after possibly shrinking $\epsilon>0$ and the chart $U_X$, the restriction $\omega^{\epsilon}_{q}|_{\cZ_{I}}$ is nondegenerate. To see this, we note that by Lemma \ref{lem:TdA}\ref{item:ort-Z-vj}, \ref{item:dvi} we have $d\hat{w}_i|_{\cZ_I}=0$ for all $i\in \{1,\dots, N\}$, so $\omega^{\flat}|_{\cZ_I}=-\sum_{i=1}^{N} m_i\, \hat{w}_{i}d\hat{\alpha}_{i}$, which is bounded in the natural coordinates of $\pi(U_{X,I}^{\circ})$. This way, we can repeat the proof of Lemma 4.15 loc.\ cit., treating the contribution from $\omega^{\flat}$ as a part of the form $\check{\omega}_1$. 
	
	With non-degeneracy of $\omega^{\epsilon}_{q}|_{\cZ_{I}}$ at hand, we now repeat the proof of Proposition 4.12 loc.\ cit, replacing each $d\bar{v}_i$ for $i\in \{1,\dots, l\}$ by $c_{i}\, d\hat{v}_i$. We now outline each step of the proof.
	
	Fix a point $x\in U_{I}$, and assume that $q>0$ or $x\in \Int_{\d A} A_{I}^{\circ}$ and $I\subseteq \cS$ (so we are in the setting of Proposition \ref{prop:omega}\ref{item:symplectic} or \ref{item:symplectic_zero}, respectively). Fix a nonzero vector $\nu\in T_{x}(\d A)$ which is $\omega^{\epsilon}_{q}$-orthogonal to the radius-zero fiber $F\de f_{A}^{-1}(0,\theta)$. We need to prove that $\nu\not\in T_{x}F$. In our coordinates, we have $T_{x}F=\ker[\sum_{i=1}^{l}m_{i}d\theta_{i}]$ and $\nu=\sum_{i=2}^{l}a_{i}\frac{\d}{\d\hat{v}_i}+\sum_{i=1}^{l}b_{i}\frac{\d}{\d \theta_i}+\xi$ for some $a_i,b_i\in \R$ and a vector $\xi\in \cZ_{I}$, see formulas (65) and (66) loc.\ cit. 
	
	We claim that at the point $x$, the coefficients $c_{1},\dots,c_{l}$ from Lemma \ref{lem:TdA}\ref{item:omega-dtheta} are positive. To this end, we recall that on $A_{I}^{\circ}$ we have $u_i=\hat{u}_i\geq 0$, and the inequality is strict if $i\in I$ and $x\in \Int_{\d A}A_{I}^{\circ}$, see Lemma \ref{lem:product}. Thus the formula for $c_i$ in Lemma \ref{lem:TdA}\ref{item:omega-dtheta} implies that $c_i\geq 0$, and the equality holds if and only if $q=0$ and $u_i=0$, so $x\not\in \Int_{\d A} A_{I}^{\circ}$, which is impossible by assumption.
	
	Next, we claim that $a_i=0$ for all $i\in \{2,\dots, l\}$. Let $\nu_{i}^{\theta}\de \frac{1}{m_{i}}\frac{\d}{\d \theta_{i}}- \frac{1}{m_{1}}\frac{\d}{\d \theta_{1}}\in T_{x}F$. Lemma \ref{lem:TdA}\ref{item:omega-dtheta} gives $
		0=\omega_{q}^{\epsilon}(\nu,\nu_{i}^{\theta})=
		\epsilon \cdot( \tfrac{c_{i}}{m_{i}} a_{i}-\tfrac{c_1}{m_1}d\hat{v}_{1}(\nu))$, 
	so $\frac{c_{i}}{m_{i}}a_{i}=\frac{c_1}{m_{1}}d\hat{v}_{1}(\nu)$. If $a_{i}\neq 0$ then $d\hat{v}_i(\nu)\neq 0$ and by Lemma \ref{lem:TdA}\ref{item:dv1} we have 
	$
	-1=\frac{1}{d\hat{v}_1(\nu)}\cdot \sum_{i=2}^{l}\frac{\zeta(u_i)}{\zeta(u_1)}a_{i}=
	\sum_{i=2}^{l}\frac{\zeta(u_i)m_{i}c_1}{\zeta(u_1)m_1c_i}>0,
	$
	a contradiction.
	
	Now, using non-degeneracy of $\omega^{\epsilon}_{q}|_{\cZ_{I}}$, we prove that $\xi=0$. Fix $\xi'\in \cZ_{I}$. By Lemma \ref{lem:TdA}\ref{item:ort-Z-vj},\ref{item:omega-dtheta} we have $\omega_{q}^{\epsilon}(\xi',\frac{\d}{\d \theta_{i}})=\epsilon c_{i}d\hat{v}_{i}(\xi')=0$, so  $\omega_{q}^{\epsilon}(\xi',\nu)=\omega_{q}^{\epsilon}(\xi',\xi)$ since all the coefficients $a_{i}$ vanish. Thus non-degeneracy of $\omega^{\epsilon}_{q}|_{\cZ_{I}}$ indeed implies $\xi=0$.
	
	Eventually, we compute the coefficients $b_{i}$. Lemma \ref{lem:TdA}\ref{item:omega-dtheta} gives $0=\omega_{q}^{\epsilon}(\frac{\d}{\d \hat{v}_i},\nu)=\sum_{j}b_{j}\cdot \epsilon c_j d\hat{v}_{j}(\frac{\d}{\d \hat{v}_{i}})=\epsilon(b_{i}c_{i}-b_{1}c_1\frac{\zeta(u_i)}{\zeta(u_1)})$ by Lemma \ref{lem:TdA}\ref{item:dv1}. Hence there is a number $b$ such that for all $i$ we have $b_i=b\cdot \frac{\zeta(u_i)}{c_i}$, and $b\neq 0$ since $\nu\neq 0$. Now $\sum_{i}m_{i}d\theta_{i}(\nu)=b\sum_{i}m_{i}\frac{\zeta(u_i)}{c_i}\neq 0$, so $\nu\not\in T_{x}F$, as claimed.
\end{proof}

\begin{remark}[Monodromy]
	We note that the A'Campo space constructed in this article, together with the form $\omega_{1}^{\epsilon}$, can be used to prove the main results of \cite{FdBP_Zariski}, too. Indeed, computing the number $b$ in the above proof shows that the $\omega_{1}^{\epsilon}$-symplectic lift to $\d A$ of the angular vector field $\frac{\d}{\d \theta}$ on $\d\D_{\delta,\log}$ is given by $\frac{1}{\sum_{i}m_{i}\zeta(u_i)}\sum_{i}\zeta(u_i)\frac{\d}{\d \theta_{i}}$. This is precisely formula (48) from \cite{FdBP_Zariski}, which guarantees that the symplectic monodromy at radius zero, i.e., the time one flow of the above vector field, has simple dynamics as in the A'Campo's topological model, see Proposition 7.4 loc.\ cit.
\end{remark}

\begin{remark}[One can make $\omega_{q}^{\epsilon}$ K\"ahler]
	Replacing $\omega_{q}^{\epsilon}$ by $\omega_{q}^{\epsilon}+C\cdot dg\wedge d\theta$ for $C\gg 0$, we can guarantee that the form $\omega_{q}^{\epsilon}$ is symplectic (and K\"ahler away from $\d A$), not just   \emph{fiberwise} symplectic (or K\"ahler). 
\end{remark}

\begin{remark}[One can keep $\omega_{q}^{\epsilon}=\omega_{X}$ over $|z|>\delta_{0}$]\label{rem:delta}
	Fix $\delta_{0}\in (0,\delta')$ and a smooth function $\rho\colon [0,\infty)\to [0,1]$ such that $\rho(0)=1$ and $\rho(r)=0$ for $r\geq \delta_{0}$. Following an analogy with formula (46) in \cite{FdBP_Zariski}, we can define $\lambda^{\sharp}_{\delta_{0}}=\rho(|f|)\lambda^{\sharp}$,  
	$\lambda^{\flat}_{\delta_{0}}=\rho(|f|)\lambda^{\flat}$, 
	 and eventually $\omega_{q}^{\epsilon,\delta_{0}}=\pi^{*}\omega_{X}+\epsilon\cdot (q\cdot d\lambda^{\sharp}_{\delta_{0}}+(1-q)\cdot d\lambda^{\flat}_{\delta_{0}})$. This way, we have fiberwise equality $\omega_{q}^{\epsilon,\delta_0}=\pi^{*}\omega_{X}$ over $\C\setminus \D_{\delta_0}$; and in general $\omega_{q}^{\epsilon,\delta_{0}}=\omega_{q}^{\epsilon'}$ for some $\epsilon'\in [0,\epsilon]$, with $\epsilon'=\epsilon>0$ on $\d A$. Thus by Proposition \ref{prop:omega}, the form  $\omega_{q}^{\epsilon,\delta_0}$ is fiberwise symplectic for all $\epsilon\in (0,\epsilon_0]$.
\end{remark}

\section{Lagrangian fibration at radius zero}\label{sec:expanded}

As in the previous sections, let $f\colon X\to \D_{\delta}$ be a holomorphic function such that the unique singular fiber $f^{-1}(0)=\sum_{i=1}^{N}m_{i}D_{i}$ is snc. Choose a subset $\cS\subseteq \{1,\dots, N\}$, and recall that $\Delta_{\cS}$ denotes the subcomplex of the dual complex $\Delta_{D}$ spanned by vertices with indices in $\cS$. Let $A$ be the A'Campo space of $f$, equipped with the family of forms  $\omega_{q}^{\epsilon}$ introduced in formula \eqref{eq:omega-s-intro}. 

In this section, we introduce the \emph{expanded skeleton} $E_{\cS}$, together with a map 
\begin{equation}\label{eq:Lagrangian}
	\ll_{\cS}\colon A_{\cS}^{\sm}\to E_{\cS}, \quad \mbox{where}\quad  A_{\cS}^{\sm}= \bigsqcup\nolimits_{\#I\geq n,\ I\subseteq \cS} A_{I}^{\circ}, 
\end{equation}
i.e., $A_{\cS}^{\sm}$ is the union of boundary pieces of the A'Campo space which correspond to maximal and submaximal faces of the skeleton $\Delta_{\cS}$. Moreover, we introduce the \emph{generic region} $A_{\cS}^{\gen}\subseteq A_{\cS}^{\sm}$ as 
\begin{equation*}
	A_{\cS}^{\gen}= \bigsqcup\nolimits_{\#I=n+1,\ I\subseteq \cS} \Int_{\d A} A_{I}^{\circ}.
\end{equation*}
and write $A_{I,\theta}^{\circ}$, $A_{\cS,\theta}^{\sm}$ and $A_{\cS,\theta}^{\gen}$ for the intersections of the above regions with a radius-zero fiber $f_{A}^{-1}(0,\theta)$. 

Proposition \ref{prop:Lagranian-at-radius-zero}, which is the main result of this section, asserts that the restriction of $\ll_{\cS}$ to $A_{\cS,\theta}^{\sm}$ is a Lagrangian torus fibration with respect to each form $\omega_{q}^{\epsilon}$, with some mild singularities.

As a set, $E_{\cS}$ is a disjoint union of maximal faces of $\Delta_{\cS}$ and products $\Delta_{I}\times C_{I}$, where $\Delta_{I}$ is a submaximal face, and $C_{I}$ is a certain graph, namely a Reeb space of a fixed Morse function on the Riemann surface $X_{I}^{\circ}$, defined in Example \ref{ex:Reeb_space}. These pieces, glued together in a natural way, give rise to a manifold with boundary, corners and some mild branching coming from branching vertices of $C_{I}$, which correspond to saddle points of $h_{I}$. These singularities of $E_{\cS}$ are therefore contained in the union of simplices $\Delta_{I}\times \Crit(h_{I})$, which, as we will see, be the discriminant locus of $\ll_{\cS}$. To describe this structure in precise terms, we introduce a technical notion of an \emph{ivy-like manifold}. 

In the setting of maximal Calabi--Yau degenerations, by Proposition \ref{prop:model}\ref{item:sk_P1} we can choose $f$ so that
\begin{equation}\label{eq:Cst}
	(X_{I},X_{I}^{\circ})\cong (\mathbb{P}^1,\C^{*})\quad \mbox{for any submaximal face } \Delta_{I} \mbox{ of } \Delta_{\cS},
\end{equation}
hence the Morse functions $h_{I}$ can be chosen without critical points, so $\ll_{\cS}$ is a Lagrangian torus fibration in the usual sense, that is, without singularities. In this case, the above description of $E_{\cS}$ agrees with the one in the introduction. In particular, the interior of $E_{\cS}$ is homeomorphic to $\Delta_{\cS}\setminus \Delta_{\cS}^{\geq 2}$, where $\Delta_{\cS}^{\geq 2}$ is the union of faces of codimension at least $2$. We note that $\Delta_{\cS}\setminus \Delta_{\cS}^{\geq 2}$ is the base of the non-archimedean affinoid torus fibration constructed, for any maximal Calabi--Yau degeneration, by Nicaise--Xu--Yu in \cite{NXY}. Hence our construction can be viewed as an archimedean variant of~\cite{NXY}..

\subsection{Ivy-like manifolds}\label{sec:ivy}

In this section, we introduce the notion of an \emph{ivy-like} manifold, which is a topological space locally modeled on a graph $C$ multiplied by the euclidean space. Such an object will be used as a base $E_{\cS}$ of the Lagrangian fibration in Theorem \ref{theo:general}\ref{item:Lagrangians}; where the branching vertices of $C$ will correspond to saddle points of a Morse function $h_{I}$ on $X_{I}^{\circ}$ for a submaximal face of $\Delta_{\cS}$. In the setting of maximal Calabi--Yau degenerations, one can ensure that condition \eqref{eq:Cst} holds, so $h_{I}$ have no critical points and $E_{\cS}$ is a manifold in the usual sense. Therefore, the reader only interested in the Calabi--Yau case treated in Theorems \ref{theo:CY}, \ref{theo:CY_af} can assume condition \eqref{eq:Cst} and skip this technical section.
\smallskip

An {\em open graph} is a topological space $C$ obtained as a quotient of a finite disjoint union of closed and semi-closed intervals, by some equivalence relation identifying only some endpoints of different intervals.  The images of those endpoints are called \emph{vertices} of $C$, a \emph{degree} of such vertex is the number of intervals containing it. The set of vertices of degree at least $3$ is called a \emph{ramification locus} and denoted $\Ram(C)$; similarly, the set of vertices of degree $1$ is called the \emph{boundary} of $C$, and denoted by $\d C$. Clearly, $C\setminus \Ram(C)$ is a $1$-dimensional manifold with boundary $\d C$. Since we are only interested in $C$ as a topological space, we can and do assume that no vertex has degree $2$.

An \emph{ivy} is an open graph $C$ together with a proper continuous map $p\colon C\to (0,\infty)$ whose restriction to each interval constituting $C$ is a homeomorphism onto its image. Note that the restriction of $p$ to $C\setminus (\Ram(C)\sqcup \d C)$ is a local homeomorphism. We choose a smooth structure on $C\setminus \Ram(C)$ such that the restriction $p|_{C\setminus \Ram(C)}$ becomes a local diffeomorphism onto its image.
	
\begin{example}[Stein factorization of a Morse function, see Figure \ref{fig:Reeb}]\label{ex:Reeb_space}
	The above definition of an ivy is motivated by the following classical construction, explained in more detail in \cite[\textsection 2]{Saeki_Reeb-spaces} and references therein. Let $M$ be a Riemann surface, and let $h\colon M\to \R$ be a proper Morse function. The \emph{Reeb space} of $h$ is the quotient space $C\de M/_{\sim}$, where for two points $x,y\in M$ we put $x\sim y$ if $x$ and $y$ lie in the same connected component of a fiber of $h$. The quotient map $\bar{h}\colon M\to C$ is called a \emph{Stein factorization} of $h$. Writing $h=p\circ \bar{h}$, it is easy to see that $(C,p)$ is an ivy, and that the restriction 
	\begin{equation*}
	\bar{h}\colon M\setminus \bar{h}^{-1}(\Ram(C))\to C\setminus \Ram(C)
	\end{equation*}
	is smooth. 
	For $c\not\in \d C\cup \Ram(C)$, the fiber $M_{c}\de \bar{h}^{-1}(c)$ is a connected submanifold of dimension $1$, i.e., an embedded circle. For $c\in \d C$ the fiber $M_{c}$ is a single point, either a maximum or a minimum of $h$. Eventually, for $c\in \Ram(C)$ the fiber $M_{c}$ is a connected union of circles, meeting at saddle points of $h$ where, since $h$ is Morse, $M_{c}$ is locally of the form $\{xy=0\}$. Choosing \enquote{top} and \enquote{bottom} branch of each such singularity, we conclude that $M_{c}$ is an immersed circle (but the immersion is not canonical).
\begin{figure}[ht]
	\begin{tikzpicture}
	\begin{scope}	
		\draw[densely dashed] (-0.5,0) to[out=10,in=170] (0.5,0);		
		\draw[densely dashed] (-0.5,0) to[out=-10,in=-170] (0.5,0);
		\draw (-0.5,0) to[out=90,in=-90] (-1.5,2.5) to[out=90,in=180] (-1,3.5) to[out=0,in=180] (0,2.5) to[out=0,in=-90] (0.5,4);
		\draw[densely dashed] (0.5,4) to[out=10,in=170] (1.5,4);		
		\draw[densely dashed] (0.5,4) to[out=-10,in=-170] (1.5,4);
		\draw (1.5,4) -- (1.5,2.5) to[out=-90,in=90] (0.5,0);
		\draw (0.05,2.1) to[out=-120,in=120] (0.05,0.9);
		\draw (0,2) to[out=-60,in=60] (0,1);
		\filldraw (-1,3.5) circle (0.06);
		\draw[densely dotted] (0.5,3.5) to[out=10,in=170] (1.5,3.5);
		\draw[densely dotted] (0.5,3.5) to[out=-10,in=-170] (1.5,3.5);
		\draw[densely dotted] (-1.5,2.5) to[out=10,in=170] (0,2.5);
		\draw[densely dotted] (-1.5,2.5) to[out=-10,in=-170] (0,2.5);
		\draw[densely dotted] (0,2.5) to[out=10,in=170] (1.5,2.5);
		\draw[densely dotted] (0,2.5) to[out=-10,in=-170] (1.5,2.5);
		\draw[densely dotted] (-1.4,2) to[out=10,in=170] (0,2);
		\draw[densely dotted] (-1.4,2) to[out=-10,in=-170] (0,2);
		\draw[densely dotted] (0,2) to[out=10,in=170] (1.4,2);
		\draw[densely dotted] (0,2) to[out=-10,in=-170] (1.4,2);
		\draw[densely dotted] (-0.8,1) to[out=10,in=170] (0,1);
		\draw[densely dotted] (-0.8,1) to[out=-10,in=-170] (0,1);
		\draw[densely dotted] (0,1) to[out=10,in=170] (0.8,1);
		\draw[densely dotted] (0,1) to[out=-10,in=-170] (0.8,1);
		\node at (0.8,0.3) {\small{$M$}};
		\draw[->] (2.5,2) -- (4,2);
		\node at (3.25,2.25) {\small{$\bar{h}$}};
	\end{scope}
	\begin{scope}[shift={(6,0)}]
		\draw (0,0) -- (0,1);
		\filldraw (0,1) circle (0.06);
		\draw (0,1) to[out=150,in=-150] (0,2);
		\draw (0,1) to[out=30,in=-30] (0,2);
		\filldraw (0,2) circle (0.06);
		\draw (0,2) -- (0,2.5);
		\filldraw (0,2.5) circle (0.06);
		\draw (0,2.5) -- (-0.5,3.5);
		\filldraw (-0.5,3.5) circle (0.06);
		\node[above] at (-0.5,3.5) {\small{$\d C$}};
		\draw (0,2.5) -- (0.5,4);
		\node at (0.3,0.3) {\small{$C$}};
		\draw[gray, ->] (-1.7,0.3) to[out=0,in=180] (-0.2,2.5);
		\draw[gray, ->] (-1.7,0.3) to[out=0,in=180] (-0.2,2);
		\draw[gray, ->] (-1.7,0.3) to[out=0,in=180] (-0.2,1);
		\node[left] at (-1.6,0.3) {\small{$\Ram(C)$}};
		\draw[->] (1,2) -- (2.5,2);
		\node at (1.75,2.2) {\small{$p$}};
	\end{scope}
	\begin{scope}[shift={(9.5,0)}]
		\draw (0,0) -- (0,4);
		\filldraw (0,1) circle (0.06);
		\filldraw (0,2) circle (0.06);
		\filldraw (0,2.5) circle (0.06);
		\filldraw (0,3.5) circle (0.06);
		\node at (0.3,0.3) {\small{$\R$}};
		\draw[gray, ->] (1.8,0.3) to[out=180,in=0] (0.2,2.5);
		\draw[gray, ->] (1.8,0.3) to[out=180,in=0] (0.2,2);
		\draw[gray, ->] (1.8,0.3) to[out=180,in=0] (0.2,1);
		\draw[gray, ->] (1.8,0.3) to[out=180,in=0] (0.2,3.5);
		\node[right] at (1.85,0.5) {\small{critical}};
		\node[right] at (1.85,0.2) {\small{values of $h$}};
	\end{scope}
	\end{tikzpicture}
	\caption{Stein factorization of a Morse function $h$ through its Reeb space $C$.}
	\label{fig:Reeb}
\end{figure}
\end{example}

We now return to abstract definitions. An \emph{ivy-like chart} (\emph{with boundary and corners}) is a homeomorphism $\varphi\colon V\to U\times C$ from a topological space $V$ to a product $U\times C$ of a smooth manifold $U$ (with boundary and corners) and an ivy $p\colon C\to (0,\infty)$. The \emph{ramification locus} of such a chart is $\Ram(V)\de \varphi^{-1}(U\times \Ram(C))$, and \emph{outer boundary} is $\dout V\de \varphi^{-1}(U\times \d C)$. Note that $V\setminus \Ram(V)$ inherits a structure of a smooth manifold with boundary and corners from $U\times (C\setminus\Ram(C))$. 

We say that a continuous, surjective map $\psi\colon N\to V$ from a smooth manifold (with boundary and corners) to an ivy-like chart $V$ is an \emph{ivy-like submersion} if it locally comes from a Stein factorization, i.e., the following holds. There is an open covering $N=\bigcup_{i} N_{i}$ and for each $i$, there is a smooth manifold $N_{i}'$ (with boundary and corners), a Riemann surface $M_{i}$, and a diffeomorphism $\tau_{i}\colon N_{i}'\times M_i\to N_{i}$ so that the composition $\varphi|_{\psi(N_i)}\circ \psi|_{N_{i}}\circ \tau_{i}$ is of the form $(\varphi',\bar{h}_{i})$, where $\varphi'\colon N_{i}'\to U$ is a submersion onto its image, and $\bar{h}_{i}$ is the Stein factorization of some Morse function $h_{i}\colon M_i\to \R$, see Example \ref{ex:Reeb_space}. Clearly, an ivy-like submersion restricts to a usual submersion over $V\setminus \Ram(V)$.
\smallskip

It is convenient to call a \emph{usual chart} a homeomorphism $\varphi \colon V\to U$ for some smooth manifold $U$ (with boundary and corners). In this case we put $\Ram(V)=\dout V=\emptyset$.

An \emph{ivy-like manifold} (with boundary and corners) is a Hausdorff topological space $E$ together with an open covering $E=\bigcup_{i}V_{i}$, such that each $V_{i}$ is a domain of a usual or ivy-like chart $\varphi_{i}$; for every $i\neq j$ we have $\Ram(V_i)\cap \Ram(V_j)=\emptyset$; and the transition maps $\varphi_{j}|_{V_{i}\cap V_{j}}\circ \varphi_{i}^{-1}|_{\varphi_{i}(V_i\cap V_j)}$ are smooth, as maps between usual manifolds (with boundary and corners). We define its ramification locus $\Ram(E)$ and outer boundary $\dout E$ as the union of, respectively, all ramification loci and outer boundaries of the charts. A collection $\{(V_i,\varphi_{i})\}$ is called an \emph{ivy-like atlas}. Note that the subspace $E\setminus \Ram (E)$ is a manifold (with boundary and corners) in the usual sense; and $\dout E\subseteq \d E$. 

A continuous, surjective map $\psi\colon N\to E$ from a smooth manifold $N$ with boundary and corners to an ivy-like manifold $E$ is an \emph{ivy-like submersion} if for every $i$, the composition $\varphi_{i}\circ \psi|_{\psi^{-1}(V_i)}$ is a (usual, or ivy-like) submersion. In particular, $\psi$ restricts to a usual submersion over $E\setminus \Ram(E)$.

\begin{definition}[Ivy-like Lagrangian torus fibrations]\label{def:Lagranian_fibration}
	Let $(Y,\omega)$ be a symplectic manifold of dimension $2n$, and let $\iota\colon L\to Y$ be an embedding (or an immersion). We say that its image $\iota(L)$ is an (\emph{immersed}) \emph{isotropic submanifold} if $\iota^{*}\omega=0$. It is (\emph{immersed}) \emph{Lagrangian} if additionally $\dim L=n$.
	
	Let $E$ be an ivy-like manifold (with boundary and corners). We say that a map $\ll\colon Y\to E$ is an \emph{ivy-like Lagrangian torus fibration} if it is an ivy-like submersion and for every $p\in E$, and every connected component $T$ of the fiber $\ll^{-1}(p)$ the following hold.
	\begin{enumerate}
		\item If $p\in E\setminus (\Ram E\cup \dout E)$ then $T$ is Lagrangian torus $(\mathbb{S}^1)^{n}$.
		\item If $p\in \Ram E$ then $T$ is an immersed Lagrangian torus $(\mathbb{S}^1)^{n}$.
		\item if $p\in \dout E$ then $T$ is an isotropic torus $(\mathbb{S}^{1})^{n-1}$.
	\end{enumerate}
	We say that $\ll$ is a \emph{Lagrangian torus fibration} if $\Ram E\cup \dout E=\emptyset$. In this case, $E$ is a manifold (with boundary and corners), the map $\ll$ is a usual submersion and all fibers of $\ll$ are Lagrangian tori.
\end{definition}

\subsection{Construction of the expanded skeleton}\label{sec:expanded_construction}

We return to the setting of Section \ref{sec:form}, that is, we let $X$ be a K\"ahler manifold of dimension $n+1$; let  $f\colon X\to \C$ be a holomorphic function whose unique singular fiber $f^{-1}(0)=\sum_{i=1}^{N}m_{i}D_{i}$ is snc; and  let $A$ be the associated A'Campo space with fiberwise symplectic forms $\omega_{q}^{\epsilon}$ introduced in formula \eqref{eq:omega-s-intro}.

Our goal is to define the expanded skeleton $E_{\cS}$ together with a map $\ll_{\cS}\colon A^{\sm}_{\cS}\to E_{\cS}$ as in  \eqref{eq:Lagrangian}, and prove that the latter is an ivy-like Lagrangian torus fibration. We first do it piecewise, that is, we define a space $E_{I}^{\circ}$ and a map $\ll_{I}^{\circ}\colon A_{I}^{\circ}\to E_{I}^{\circ}$ for every maximal or submaximal face $\Delta_{I}$ of $\Delta_{\cS}$.
\smallskip

Let $\Delta_{J}$ be a maximal face of $\Delta_{\cS}$, so the stratum $X_{J}$ is a point. Fix an adapted chart around $X_{J}$ such that its preimage $U_{J}$ is covered by smooth charts \eqref{eq:AC-chart_hat}. On $U_{J}$ we have a free, smooth $(\mathbb{S}^{1})^{n+1}$-action, given in each of those charts by a translation in $\theta_{i}$-directions. Denote the quotient map by 
\begin{equation}\label{eq:thickquotient}
	\qq_{J}\colon U_{J}\to Q_{J},\quad\mbox{and define}\quad \ll_{J}^{\circ}\de \qq_{J}|_{A_{J}^{\circ}}\colon A_{J}^{\circ}\to E_{J}^{\circ},
\end{equation}
where $E_{J}^{\circ}=\qq_{J}(A_{J}^{\circ})$. Clearly, $Q_{J}$ is a smooth manifold, with coordinates $(\hat{v}_{i})_{i\in J}$. By Lemma~\ref{lem:product}, $E_{J}^{\circ}$ is a codimension zero submanifold with boundary and corners of $Q_{J}$, which in the coordinates $(\hat{v}_{i})_{i\in J}$ gets identified with rounded simplex $\tilde{\Delta}_{J}$. In particular, $E_{J}^{\circ}$ is diffeomorphic to the maximal face $\Delta_{J}$.
\smallskip

Let now $\Delta_{I}$ be a submaximal face of $\Delta_{\cS}$, so the corresponding stratum $X_{I}$ is a Riemann surface. We choose a proper Morse function $h_{I}\colon X_{I}^{\circ}\to (0,\infty)$ satisfying the following condition: for every point $p\in X_{I}\setminus X_{I}^{\circ}$ there is a neighborhood $U_{p}$ of $p$ in $X_{I}$, and a holomorphic coordinate $z\in \cO_{X_{I}}(U_p)$ at $p$ such that on $U_{p}\setminus \{p\}\subseteq X_{I}^{\circ}$ we have $h_{I}=|z|$ or $|z|^{-1}$. Clearly, such a Morse function $h_{I}$ exists. 
Furthermore, if $(X_{I},X_{I}^{\circ})\cong (\mathbb{P}^1,\C^{*})$, then we can and do choose $h_{I}=|z_{I}|$, where $z_{I}$ is a holomorphic coordinate on $X_{I}^{\circ}$. In general, let $h_{I}=p_{I}\circ \bar{h}_{I}$ be the Stein factorization of $h_{I}$, 
so $p_{I}\colon C_{I}\to (0,\infty)$ is its Reeb space, see Example \ref{ex:Reeb_space}. We define an ivy-like chart $E_{I}^{\circ}\de \tilde{\Delta}_{I}\times C_{I}$, and an ivy-like submersion 
\begin{equation}\label{eq:submaximalquotient}
	\ll_{I}^{\circ}\de ((\hat{v}_{i})_{i\in I},\bar{h}_{I}\circ \pi)\colon A_{I}^{\circ}\to E_{I}^{\circ}.
\end{equation}
If $(X_{I},X_{I}^{\circ})\cong (\mathbb{P}^1,\C^{*})$ then $C_{I}=\R$, $\bar{h}_{I}=h_I$ and $\ll_{I}^{\circ}$ is a  submersion in the usual sense.

\begin{lema}\label{lem:glue}
	The maps $\ll_{I}^{\circ}$ defined in \eqref{eq:thickquotient} and \eqref{eq:submaximalquotient} glue to an ivy-like submersion $\ll_{\cS}\colon A_{\cS}^{\sm}\to E_{\cS}$. Furthermore, if condition \eqref{eq:Cst} holds then $\ll_{\cS}$ is a submersion in the usual sense.
\end{lema}
\begin{proof}
	Fix a maximal face $\Delta_{J}$ of $\Delta_{\cS}$ and a submaximal face $\Delta_{I}$ contained in $\Delta_{J}$. We explain how to glue $\ll_{J}^{\circ}$ with $\ll_{I}^{\circ}$. The result will follow by repeating this construction for all such inclusions $\Delta_{I}\subsetneq \Delta_{J}$.
	
	Say that $J=\{1,\dots, n+1\}$ and $I=J\setminus \{1\}$. By Lemma \ref{lem:rihat}\ref{item:r_maximal} there is an adapted chart $U_{X}$ around the point $X_{J}$ whose coordinates $(z_1,\dots, z_{n+1})$ satisfy $|z_i|=\hat{r}_i$, so for each $i\in J$ we have $\hat{v}_i=v_i$. Moreover, shrinking the chart $U_{X}$ if needed we can assume that the Morse function $h_{I}$ equals $|\tilde{z}|^{\pm 1}$ for some holomorphic coordinate $\tilde{z}$ of $X_{I}$ at the point  $X_{J}$. Write $\tilde{z}=\lambda_{0}\cdot \lambda(z_1)\cdot z_1$, where $\lambda_0\in \C^{*}$, and $\lambda$ is a holomorphic function on a neighborhood of $0\in \C$ such that $\lambda(0)=1$. Put $z=\lambda_{0}^{-1}\cdot \tilde{z}$, so $z$ is another holomorphic coordinate of $X_{I}$ at $X_{J}$, and we have $z=\lambda(z_1)\cdot z_1$.
	
	Put $U=\pi^{-1}(U_X)$ and $U_{I}=A_{I}\cap U\subseteq \d A$, where $A_{I}=\pi^{-1}(X_{I})$. The holomorphic coordinate $z$ pulls back to a smooth function on $U_{I}$, which we denote by the same letter. Since $|z|=|\lambda_{0}^{-1}|\cdot h_{I}^{\pm 1}$, the fibers of $|z|$ and $h_{I}$ give the same foliations of $U\cap A_{I}^{\circ}$.
	
	Recall from formula \eqref{eq:AC-chart_hat} that $U\cap \d A$ is covered by charts $V_{i}\de \{w_i>\frac{1}{n+2}\}\cap \d A$, with coordinates $((v_{j})_{j\neq i},\vartheta)$. Since on $A_{I}^{\circ}$ we have $w_{1}=0$, the set $\bigcup_{i\neq 1} V_i$ is an open neighborhood of $\bar{A}_{I}^{\circ}\cap U$ in $U_{I}$.
	
	Fix $i\in I$, say $i=n+1$, and put $V\de V_{n+1}$. The smooth coordinates on $V$ are $(v_1,\dots, v_n;\vartheta)$. We define a function $p\colon V\to \R$ piecewise, as $p=v_1$ on $A_{J}^{\circ}$ and $p=\frac{-1}{m_1\log |z|}$ on $A_{I}^{\circ}$. 	We claim that, after shrinking $U_{X}$ around $X_{J}$ if necessary, 
	\begin{equation}\label{eq:claim_gluing}
		\mbox{the map } (p,v_2,\dots,v_{n},\vartheta) \mbox{ is a smooth coordinate system on }V. 
	\end{equation}
	Once \eqref{eq:claim_gluing} is shown, we conclude as follows. The foliation of $V$ given by the projection to the first $n$ coordinates agrees with the foliations of $V\cap A_{J}^{\circ}$ and $V\cap A_{I}^{\circ}$ given by $\ll_{J}^{\circ}$ and $\ll_{I}^{\circ}$, respectively. Hence these maps glue to a submersion $V\to E_{V}$, onto some smooth manifold $E_{V}$. Now, we define a smooth manifold $E_{I,J}$ as $\bigsqcup_{i\neq 1} E_{V_i}/_{\sim}$, where $\sim$ identifies two points whenever they are images of the same point in $\bigcup_{i\neq 1} V_{i}$. This way, we get a submersion $\bigcup_{i\neq 1} V_{i}\to E_{I,J}$, which restricts to $\ll_{I}^{\circ}$ and $\ll_{J}^{\circ}$. Eventually, we glue to $E_{I,J}$ a usual chart $\ll_{J}^{\circ}(\Int_{\d A}A_{J}^{\circ})\subseteq E_{J}^{\circ}$ and a $E_{I}^{\circ}$, which is a usual chart if  $(X_I,X_{I}^{\circ})\cong (\mathbb{P}^1,\C^{*})$, and an ivy-like chart in general.
	\smallskip
	
	It remains to prove claim \eqref{eq:claim_gluing}. Clearly, the function $p$ is smooth on the interior of $A_{J}^{\circ}$, where it equals the coordinate $v_1$; and on $A_{I}^{\circ}$, where it is a pullback of a smooth function from $X_{I}^{\circ}$. It remains to prove that all derivatives of $v_{1}-p$ approach $0$ as we approach $A_{J}^{\circ}$ from $A_{I}^{\circ}$, i.e., as $v_1\searrow 0$.  
	
	Recall that $z=\lambda(z_1)\cdot z_1$ for some non-vanishing, holomorphic function $\lambda$ such that $\lambda(0)=1$. Moreover, on $A_{I}^{\circ}$ we have $v_1=t_1=\frac{-1}{m_1\log|z_1|}$, so  $p=\frac{-1}{m_1\log|z|}=\frac{-1}{m_1\log|\lambda(z_1)|+m_1\log|z_1|}=\frac{v_{1}}{1-v_{1}m_1\log|\lambda(z_1)|}$. Thus $p-v_1=\frac{v_{1}^2m_1\log |\lambda(z_1)|}{1-v_1m_1 \log|\lambda(z_1)|}$. By Leibniz rule, it is enough to show that  $\log|\lambda(z_1)|$ and all its derivatives approach $0$ as $v_{1}\searrow 0$. By assumption we have  $\log|\lambda(0)|=\log 1=0$, so by chain rule it remains to show that all derivatives of $z_1$ vanish as $v_1\searrow 0$. Since $v_{1}=\frac{-1}{m_1\log|z_1|}$ on $A_{I}^{\circ}$, we have $z_1=\exp(-m_{1}^{-1}v_{1}^{-1}+2\pi \imath\cdot \theta_{1})$, where $\theta_{1}$ is a smooth coordinate on $U$. Thus $dz_{1}=\exp(-m_{1}^{-1}v_{1}^{-1})\cdot \exp(2\pi\imath \cdot \theta_1)\cdot (m_{1}^{-1}v_{1}^{-2}\, dv_{1}+2\pi\imath\, d\theta_{1})$. The result follows since all derivatives of the one-variable function $v_1\mapsto \exp(-m_{1}^{-1}v_{1}^{-1})v_{1}^{-2}$ vanish at $0$. 
\end{proof}

\begin{remark}[Choices within the construction of the radius-zero Lagrangian fibration $\ll_{\cS}$]
	\label{rem:canonical_choice}
	The restriction of $\ll_{\cS}$ to the generic region $A_{\cS}^{\gen}$, viewed as a fibration over the disjoint union of all maximal faces of $\Delta_{\cS}$, is canonically defined by formula \eqref{eq:thickquotient}. Its extension $\ll_{\cS}\colon A_{\cS}^{\sm}\to E_{\cS}$ to the submaximal locus depends on the Morse functions $h_{I}\colon X_{I}^{\circ}\to \R$ chosen for each submaximal face $\Delta_{I}$ of $\Delta_{\cS}$.
		
	Nonetheless, if condition \eqref{eq:Cst} holds, then we have chosen the Morse functions in a canonical way, as the absolute value of a holomorphic coordinate on $X_{I}^{\circ}\cong \C^{*}$. This choice is unique up to a multiplication by a nonzero complex number or taking the inverse, so the foliation given by each $h_{I}$, and in consequence the map $\ll_{\cS}$, do not depend on this choice.  Thus we get a  \emph{canonical} smooth submersion $\ll_{\cS}\colon A_{\cS}^{\sm}\to E_{\cS}$, in the usual, not just ivy-like,  sense. Recall that condition \eqref{eq:Cst} holds e.g. in the setting of Theorem~\ref{theo:CY}, when $f$ is snc model of a maximal Calabi--Yau degeneration chosen in Proposition \ref{prop:model}.
\end{remark}

\subsection{Lagrangian fibration at radius zero}\label{sec:Lagrangains-at-radius-zero}

We now prove that the map $\ll_{\cS}\colon A_{\cS}^{\sm}\to E_{\cS}$ constructed in Section~\ref{sec:expanded_construction} restricts to an ivy-like Lagrangian torus fibration on  $A_{\cS,\theta}^{\sm}$ for each $\theta\in \mathbb{S}^1$. We begin by describing the fibers of $\ll_{\cS}$. Recall that $\Ram(E_{\cS})$ and $\dout E_{\cS}$ denote the  singularities and boundary components of $E_{\cS}$ which correspond to the critical points of the chosen Morse functions $h_{I}$, so they are empty if condition \eqref{eq:Cst} holds. 

\begin{lema}\label{lem:ll-fibers}
	Fix a point $b\in E_{\cS}$ and let $L_{b}=\ll_{\cS}^{-1}(b)$. Then the following hold.
	\begin{enumerate}
		\item\label{item:fiber-immersed} The fiber $L_{b}$ is the image of an immersion $\iota_{b}\colon (\mathbb{S}^{1})^{k}\to A$,  where $k=n+1$ if $b\not\in \dout E_{\cS}$ and $k=n$ otherwise.	If $b\not\in \Ram(E_{\cS})$ then the immersion $\iota_{b}$ is an embedding.
		\item\label{item:fiber-isotropic} For every $q\in [0,1]$ and every $\epsilon\geq 0$ we have $\iota_{b}^{*}\omega_{q}^{\epsilon}=0$.
		\item\label{item:fiber-restriction} Fix $\theta\in \mathbb{S}^{1}$ and let $F_{\theta}=f_{A}^{-1}(0,\theta)$ be the radius zero fiber. Then the restriction of $\iota_{b}$ to the preimage of $L_{b}\cap F_{\theta}$ is an immersion $\iota_{b}^{\theta}\colon T\to F_{\theta}$, where $T\subseteq (\mathbb{S}^{1})^{k}$ is a disjoint union of $\gcd\{m_{i}: i\in I\}$ tori $(\mathbb{S}^{1})^{k-1}$. If $b\not\in \Ram(E_{\cS})$ then $\iota_{b}^{\theta}$ is an embedding.
	\end{enumerate} 
\end{lema}
\begin{proof}
	We have $b\in E_{I}^{\circ}$ for some $I$ such that $\Delta_{I}$ is a maximal or submaximal face of $\Delta_{\cS}$. 
	
	Assume that the face $\Delta_{I}$ is maximal. Then by definition \eqref{eq:thickquotient} of $\ll_{I}^{\circ}$, its fiber $L_{b}$ is contained in some chart \eqref{eq:AC-chart_hat} with associated index set $I$, and it is parametrized by the coordinates $(\theta_{i})_{i\in I}$. Thus $L_{b}$ is an embedded torus $(\mathbb{S}^{1})^{n+1}$, which proves  \ref{item:fiber-immersed}. The tangent space $T_{p}L_{b}$ at every point $p\in L_{b}$ is spanned by coordinate vectors $(\frac{\d}{\d \theta_{i}})_{i\in I}$; so by Lemma \ref{lem:TdA}\ref{item:dvi} we have $d\hat{v}_{i}|_{L_{b}}=0$ for every $i\in I$. Thus Lemma \ref{lem:TdA}\ref{item:omega-dtheta} gives $\omega_{q}^{\epsilon}|_{L_{b}}=0$, which proves \ref{item:fiber-isotropic}. Part \ref{item:fiber-restriction} follows from the fact that in coordinates \eqref{eq:AC-chart_hat}, the fiber $F_{\theta}$ is given by equation $\sum_{i\in I} m_{i}\theta_{i}=\theta$.

	Assume that $\Delta_{I}$ is submaximal. By definition \eqref{eq:submaximalquotient} of $\ll_{I}^{\circ}$, its fiber $L_{b}$ is contained in $A_{I}^{\circ}$, and its image through the diffeomorphism \eqref{eq:product} is the preimage in $(X_{I}^{\circ})_{\log}$ of a fiber $C_{b}$ of $\bar{h}_{I}\colon X_{I}^{\circ}\to C_{I}$. By construction, $C_{b}$ is a connected component of a fiber of the Morse function $h_{I}\colon X_{I}^{\circ}\to C_{I}$, so it is either a point (if $b\in \dout E_{I}^{\circ})$, or a closed curve with only transverse self-intersections (and no self-intersections if $b\not\in \Ram E_{I}^{\circ}$). Thus $C_{b}$ is either: an embedded circle if $b\not\in \Ram E_{I}^{\circ}\cup \dout E_{I}^{\circ}$, or an immersed circle if $b\in \Ram E_{I}^{\circ}$, or a point if $b\in \dout E_{I}^{\circ}$. The $(\mathbb{S}^{1})^{n}$-bundle $(X_{I}^{\circ})_{\log}\to X_{I}^{\circ}$ splits as a direct sum of $\mathbb{S}^{1}$-bundles, where each summand can be identified with the circle bundle in $\cO_{X}(D_{i})|_{X_{I}^{\circ}}$ for $i\in I$. The restriction of this summand to $C_{b}$ is an oriented circle bundle over an immersed circle (or a point), so it is trivial. Thus $L_{b}=(\mathbb{S}^{1})^{n}\times C_{b}$, so $L_{b}$ is an embedded $(\mathbb{S}^{1})^{n+1}$ if $b\not\in \Ram E_{I}^{\circ}\cup \dout E_{I}^{\circ}$; an immersed $(\mathbb{S}^{1})^{n+1}$ if $b\in \Ram E_{I}^{\circ}$; or an embedded $(\mathbb{S}^{1})^{n}$ if $b\in \dout E_{I}$. This proves \ref{item:fiber-immersed}. 
	
	Fix a point $p\in L_{b}$, and let $V\subseteq T_{p}A_{I}^{\circ}$ be the tangent space to a branch of $L_{b}$ at $p$. We can assume that $I=\{1,\dots,n\}$, and use the coordinates $(\hat{v}_2,\dots,\hat{v}_n,\theta_{1},\dots,\theta_{n},z)$ at $p\in A_{I}^{\circ}$ introduced right before Lemma \ref{lem:TdA}. In these coordinates, the map $\ll_{I}^{\circ}$ is given by $(\hat{v}_2,\dots,\hat{v}_n,\bar{h}_{I}(z))$. Hence the space $V$ has a basis $\{\frac{\d}{\d \theta_{1}},\dots,\frac{\d}{\d \theta_{n}},\xi\}$, where $\xi$ is a vector in the complex vector space $\cZ_{I}$  spanned by $\frac{\d}{\d z}$, such that $\pi_{*}\xi$ is tangent to the branch of $C_{b}\subseteq X_{I}^{\circ}$. 
	By Lemma  \ref{lem:TdA}\ref{item:omega-dtheta} we have $\omega_{q}^{\epsilon}(\sdot,\frac{\d}{\d \theta_{i}})=\epsilon c_{i} d\hat{v}_{i}$, which by Lemma \ref{lem:TdA}\ref{item:dvi} vanishes both on $\frac{\d}{\d \theta_{j}}$ for $j=1,\dots,n$, and on $\xi$. Thus $\omega_{q}^{\epsilon}|_{V}=0$, which proves~\ref{item:fiber-isotropic}. 
	
	Part~\ref{item:fiber-restriction} follows, as before, from the fact that $F_{\theta}=\{\sum_{i\in I} m_{i}\theta_{i}=\theta\}$.
\end{proof}

The above construction  is now summarized in the following result. Recall that for a radius-zero fiber $F_{\theta}\de f_{A}^{-1}(0,\theta)\subseteq \d A$, the region $A^{\sm}_{\cS, \theta}\subseteq F_{\theta}$ is the union of pieces $A_{I}^{\circ}\cap F_{\theta}$ such that $\#I\geq n$ and $I\subseteq \cS$, i.e.,  $\Delta_{I}$ is a maximal or submaximal face of $\Delta_{\cS}$.

\begin{prop}[Lagrangian fibration at radius zero]\label{prop:Lagranian-at-radius-zero} 
	Let $\ll^{\theta}_{\cS}\colon A^{\sm}_{\cS,\theta}\to E_{\cS}$ be the restriction to $F_{\theta}$ of the map \eqref{eq:Lagrangian} constructed in Lemma \ref{lem:glue}, for some collection of Morse functions $(h_{I})$. Fix $\epsilon\in (0,\epsilon_0]$ as in Proposition \ref{prop:omega}\ref{item:symplectic}. Then the following hold.
	\begin{enumerate}
		\item\label{item:ll-Lagrangian} For every $q\in (0,1]$, the map $\ll^{\theta}_{\cS}$ is an ivy-like Lagrangian torus fibration with respect to the symplectic form $\omega_{q}^{\epsilon}$ introduced in formula \eqref{eq:omega-s-intro}, see Definition \ref{def:Lagranian_fibration}.
		\item\label{item:ll-maximal} If $\Delta_{I}$ is a maximal face of $\Delta_{\cS}$ then the image $\ll^{\theta}_{\cS}(A_{I}^{\circ}\cap F_{\theta})$ is diffeomorphic to $\Delta_{I}$. Moreover, the restriction of $\ll^{\theta}_{\cS}$ to $\Int_{\d A} A_{I}^{\circ}\cap F_{\theta}$ is a Lagrangian torus fibration with respect to $\omega_{0}^{\epsilon}$.
		\item\label{item:ll-submaximal} If $\Delta_{I}$ is a submaximal face of $\Delta_{\cS}$ then the image $\ll_{\cS}^{\theta}(A_{I}^{\circ}\cap F_{\theta})$ is diffeomorphic to $\Delta_{I}\times C_{I}$, where $C_{I}$ is the Reeb space of $h_{I}$, see Example \ref{ex:Reeb_space}. 
		\item\label{item:ll-P1} Assume that condition \eqref{eq:Cst} holds, for instance, $f$ is an snc model of a maximal Calabi--Yau degeneration chosen in Proposition \ref{prop:model}\ref{item:sk_P1}. Then 
		$\ll_{\cS}^{\theta}$ 
		is a Lagrangian torus fibration. 
	\end{enumerate}
\end{prop}
\begin{proof}
	By Lemma \ref{lem:ll-fibers}, the map $\ll^{\theta}_{\cS}$ is an ivy-like Lagrangian torus fibration with respect to $\omega_{q}^{\epsilon}$ whenever the latter is a symplectic form. Together with Proposition \ref{prop:omega}, this proves \ref{item:ll-Lagrangian} and the second statement of \ref{item:ll-maximal}. Lemma \ref{lem:ll-fibers} implies also that each fiber of $\ll_{\cS}$ meets $F_{\theta}$, so  $\ll^{\theta}_{\cS}(A_{I}^{\circ}\cap F_{\theta})=\ll_{\cS}(A_{I}^{\circ})=E_{I}^{\circ}$, which by construction is diffeomorphic to $\Delta_{I}$ if the latter is a maximal face; and to $\Delta_{I}\times C_{I}$ otherwise. This proves \ref{item:ll-maximal} and \ref{item:ll-submaximal}. Part \ref{item:ll-P1} follows from Lemma \ref{lem:ll-fibers}, too, since under condition \eqref{eq:Cst} the sets $\Ram(E_{\cS})$ and $\dout E_{\cS}$ are empty.
\end{proof}

\begin{example}[Hesse pencil, see Figure \ref{fig:Fermat_E}]\label{ex:Fermat_E}
	Consider the Hesse pencil $\{(x_1^3+x_2^3+x_3^3)\cdot z=x_1x_2x_3\}$ from Example \ref{ex:Fermat}. Recall that the central fiber is a triangle $D_1+D_2+D_3$, where each $D_i\cong  \mathbb{P}^1$ is line in $\mathbb{P}^2$. In the A'Campo space, each radius-zero fiber $F$ is a torus, divided into six cylinders $F\cap A_{I}^{\circ}$.

	Choose $\cS=\{1,2,3\}$, so $\Delta_{\cS}=\Delta_{D}$ is a triangle in $\R^3$. 
	The expanded skeleton $E_{\cS}$ is a circle, divided in six parts: three disjoint closed segments corresponding to the intersection points $D_{i}\cap D_{j}$, and three open ones, corresponding to each line $D_{i}$.
\end{example}

\begin{figure}[ht]
	\subcaptionbox{central fiber $D\subseteq X$}[.21\textwidth]{
	\begin{tikzpicture}
		\path[use as bounding box] (-0.8,-1.5) rectangle (1.8,1.5);  
		\coordinate (A) at (0,-1); 
		\coordinate (B) at (0,1); 
		\coordinate (C) at (1.732,0); 
		\draw (A) to[out=110,in=-110] (B) to[out=-70,in=70] (A);
		\node at (-0.5,0) {\small{$D_1$}};
		\draw (C) to[out=130,in=-10] (B) to[out=-50,in=170] (C);
		\node at (1.2,0.9) {\small{$D_2$}};
		\draw (C) to[out=-130,in=10] (A) to[out=50,in=-170] (C);
		\node at (1.2,-0.9) {\small{$D_3$}};
		\draw[densely dotted] ($0.5*(B)+0.5*(C)+(0.09,0.18)$) to[out=-90,in=40] ($0.5*(B)+0.5*(C)-(0.09,0.18)$);
		\draw[densely dotted] ($0.5*(A)+0.5*(C)+(0.09,-0.18)$) to[out=140,in=-90] ($0.5*(A)+0.5*(C)-(0.09,-0.18)$);
		\draw[densely dotted] (-0.2,0) to[out=20,in=160] (0.2,0);
	\end{tikzpicture}
	}
	\subcaptionbox{radius-zero fiber $F\subseteq A$}[.28\textwidth]{
		\begin{tikzpicture}
			\path[use as bounding box] (-1,-0.5) rectangle (3,2.5);
			\path [fill=black!10] (1.626,1.313) to[out=0,in=-120] (2.073,1.537) to[out=-60,in=60] (2.073,0.463) to[out=180,in=-60] (1.626,0.678) to[out=60,in=-60] (1.626,1.313);			
			\path [fill=black!10] (1,-0.2) to[out=120,in=-120] (1,0.3) to[out=180,in=-60] (0.374,0.678) to[out=180,in=60] (-0.073,0.463) to[out=-60,in=180] (1,-0.2);
			\path [fill=black!10] (0.374,1.313) to[out=120,in=0] (-0.073,1.537) to[out=60,in=180] (1,2.2) to[out=-60,in=60] (1,1.7) to[out=180,in=60] (0.374,1.313);
			\draw (1,1) circle (1.2);
			\draw (1,1) circle (0.7);
			\draw (1,2.2) to[out=-60,in=60] (1,1.7);
			\draw (1.626,1.313) to[out=0,in=-120] (2.073,1.537);
			\draw (1.626,0.678) to[out=-60,in=180] (2.073,0.463);
			\draw (1,-0.2) to[out=120,in=-120] (1,0.3);
			\draw (0.374,0.678) to[out=180,in=60] (-0.073,0.463);
			\draw (0.374,1.313) to[out=120,in=0] (-0.073,1.537);
			\node at (-0.5,1) {\small{$A_{1}^{\circ}$}};
			\node at (2.6,1) {\small{$A_{2,3}^{\circ}$}};
			\node at (0,2.2) {\small{$A_{1,2}^{\circ}$}};
			\node at (0,-0.2) {\small{$A_{1,3}^{\circ}$}};
			\node at (1.9,-0.2) {\small{$A_{3}^{\circ}$}};
			\node at (1.9,2.2) {\small{$A_{2}^{\circ}$}};
		\end{tikzpicture}
	}
	\subcaptionbox{expanded skeleton $E_{\cS}$ \label{fig:Fermat_ES}}[.24\textwidth]{
		\begin{tikzpicture}
			\path[use as bounding box] (-0.8,-0.5) rectangle (2.8,2.5);  
			\draw (1,1) circle (1);
			\draw (1,2.15) -- (1,1.85);
			\draw (1.74,1.38) -- (2.028,1.514);
			\draw (1.74,0.62) -- (2.028,0.486);
			\draw (1,0.15) -- (1,-0.15);
			\draw (0.26,0.62) -- (-0.028,0.486);
			\draw (0.26,1.38) -- (-0.028,1.514);
			\node at (-0.3,1) {\small{$E_{1}^{\circ}$}};
			\node at (2.4,1) {\small{$E_{2,3}^{\circ}$}};
			\node at (0.2,2.1) {\small{$E_{1,2}^{\circ}$}};
			\node at (0.2,-0.1) {\small{$E_{1,3}^{\circ}$}};
			\node at (1.7,-0.1) {\small{$E_{3}^{\circ}$}};
			\node at (1.7,2.1) {\small{$E_{2}^{\circ}$}};
			\draw[very thick] (1.894,1.447) to[out=-60,in=60] (1.894,0.553);
			\draw[very thick] (1,0) to[out=180,in=-60] (0.106,0.553);
			\draw[very thick] (0.106,1.447) to[out=60,in=180] (1,2);
		\end{tikzpicture}
	}
	\subcaptionbox{essential skeleton $\Delta_{\cS}$}[.24\textwidth]{
		\begin{tikzpicture}
			\path[use as bounding box] (-2.4,-1.5) rectangle (0.8,1.5); 
			\coordinate (A) at (0,-1); 
			\coordinate (B) at (0,1); 
			\coordinate (C) at (-1.732,0); 
			\draw[very thick] (A) -- (B) -- (C) -- (A);
			\filldraw (A) circle (0.08);
			\filldraw (B) circle (0.08);
			\filldraw (C) circle (0.08);
			\node at ($(A)-(1,-0.2)$) {\small{$\Delta_{1,3}$}};
			\node at ($(B)-(1,0.1)$) {\small{$\Delta_{1,2}$}};
			\node at ($(C)+(2.2,0)$) {\small{$\Delta_{2,3}$}};
			\node at ($(A)+(0.3,-0.1)$) {\small{$\Delta_{3}$}};
			\node at ($(B)+(0.3,0.1)$) {\small{$\Delta_{2}$}};
			\node at ($(C)-(0.3,0)$) {\small{$\Delta_{1}$}};
		\end{tikzpicture}
	}
	\caption{Example \ref{ex:Fermat_E}: Hesse pencil $\{(x_1^3+x_2^3+x_3^3)\cdot z=x_1x_2x_3\}$.}
	\label{fig:Fermat_E}	
\end{figure}

\begin{example}[Families of Calabi--Yau hypersurfaces in $\mathbb{P}^{n+1}$] \label{ex:general-fibration} We continue the discussion of Example~\ref{ex:general-model}. Consider the family $\bar{X}=\{P(x_0,\dots,x_{n+1})\cdot z=x_{0}\cdot\ldots\cdot x_{n+1}\}\subseteq \mathbb{P}^{n+2}\times \C$ for a general $P\in \C[x_{0},\dots,x_{n+1}]_{n+2}$, and let $f\colon X\to \C$ be its snc model constructed by successively blowing up $\bar{X}$ in the strict transforms of the coordinate hyperplanes. As we have seen in Example \ref{ex:general-model}, the restriction $f^{\circ}\colon X^{\circ}\to \C^{*}$ of $\bar{X}$ is a maximal Calabi--Yau degeneration, whose essential skeleton is the standard $n$-simplex, and coincides with the dual complex of $D=f^{-1}(0)$.
	
	The expanded skeleton $E_{\cS}$ is the disjoint union of maximal faces $\Delta_{\cS\setminus \{i\}}$ and products $\Delta_{\cS\setminus \{i,j\}}\times [0,1]$ for $0\leq i<j\leq n+1$, divided by a relation identifying 
	 $\Delta_{\cS\setminus \{i,j\}}\times \{0\}$ and  $\Delta_{\cS\setminus \{i,j\}}\times \{1\}$ with the
	corresponding faces of $\Delta_{\cS\setminus \{i\}}$ and $\Delta_{\cS\setminus \{j\}}$, respectively. Figure \ref{fig:K3-E} shows $E_{\cS}$ for $n=2$. The shaded regions are \emph{not} part of $E_{\cS}$; they represent the moment polyhedra of the toric divisors $\bar{D}_i\subseteq \bar{X}$.
\begin{figure}[ht]
	
		\subcaptionbox{central fiber $D\subseteq X$ \label{fig:K3-D}}[.32\textwidth]{
		\begin{tikzpicture}[scale=0.9]
			\coordinate (A) at (0,0);
			\coordinate (B) at (3,0);
			\coordinate (C) at (2,2.5);
			\coordinate (D) at (4,1);
			\draw[add = 0.1 and 0.1] (A) to (B);
			\draw[add = 0.1 and 0.1] (A) to (C);
			\draw[add = 0.1 and 0.1, densely dashed] (A) to (D);
			\draw[add = 0.1 and 0.1] (B) to (C);
			\draw[add = 0.1 and 0.1] (B) to (D);
			\draw[add = 0.1 and 0.1] (C) to (D);
		\end{tikzpicture}
	}
		\subcaptionbox{essential skeleton $\Delta_{\cS}$ \label{fig:K3-Delta}}[.32\textwidth]{
		\begin{tikzpicture}[scale=0.9]
			\coordinate (A) at (1.5,0.5);
			\coordinate (B) at (0,2);
			\coordinate (C) at (3,2);
			\coordinate (D) at (1.5,3);
			\draw (A) to (B);
			\draw (A) to (C);
			\draw[densely dashed] (A) to (D);
			\draw (B) to (C);
			\draw (B) to (D);
			\draw (C) to (D);
	\end{tikzpicture}	
		}
		\subcaptionbox{expanded skeleton $E_{\cS}$ \label{fig:K3-E}}[.32\textwidth]{
		\begin{tikzpicture}[scale=0.9]
			\coordinate (A) at (1.5,0.5);
			\coordinate (B) at (0,2);
			\coordinate (C) at (3,2);
			\coordinate (D) at (1.5,3);
			\coordinate (A1) at ($(A)-(0.8,-0.2)$);
			\coordinate (B1) at ($(B)-(0.8,-0.2)$);
			\coordinate (C1) at ($(C)-(0.8,-0.2)$);
			\coordinate (D1) at ($(D)-(0.8,-0.2)$);
			\coordinate (A2) at ($(A)+(0.8,0.2)$);
			\coordinate (B2) at ($(B)+(0.8,0.2)$);
			\coordinate (C2) at ($(C)+(0.8,0.2)$);
			\coordinate (D2) at ($(D)+(0.8,0.2)$);
			\coordinate (A3) at ($(A)+(0,0.4)$);
			\coordinate (B3) at ($(B)+(0,0.4)$);
			\coordinate (C3) at ($(C)+(0,0.4)$);
			\coordinate (D3) at ($(D)+(0,0.4)$);
			\filldraw[black!10] (A) -- (A1) -- (A2) -- (A);
			\draw[densely dashed] (A1) -- (A2);
			\draw (A1) -- (A)-- (A2);
			\filldraw[fill=black!20] (B) -- (B1) -- (B3) -- (B);
			\filldraw[fill=black!20] (C) -- (C2) -- (C3) -- (C);
			\filldraw[black!10] (D3) -- (D1) -- (D2) -- (D3);
			\draw[densely dashed] (D1) -- (D2);
			\draw (D2) -- (D3)-- (D1);
			\draw (A) -- (B) -- (C) -- (A);
			\draw[densely dashed] (A1) -- (D1); 
			\draw (D1) -- (B1) -- (A1);
			\draw[densely dashed] (A2) -- (D2); 
			\draw (D2) -- (C2) -- (A2);
			\draw (B3) -- (C3) -- (D3) -- (B3);
	\end{tikzpicture}
	}
	\caption{Example \ref{ex:general-fibration}, $n=2$: expanded skeleton for a maximal K3 degeneration.}
	\label{fig:K3-expanded}
\end{figure}	
		
	Theorem \ref{theo:CY}\ref{item:CY_Lagrangian} provides, for each model of $f^{\circ}$ as above (depending on the order of blowups), a smooth Lagrangian torus  fibration $(X_{z}^{q})^{\sm}_{\cS}\to E_{\cS}$ of a certain region $(X_{z}^{q})^{\sm}_{\cS}$ of $X_{z}$ for $z\in \C^{*}$ close to the origin. This fibration is obtained in Proposition \ref{prop:Lagranian-at-positive-radius} by moving $\ll_{\cS}^{\theta}$ to the positive radius. It can be regarded as an archimedean version of the affinoid torus fibration constructed in \cite{NXY}: indeed, the base of the latter is $\Delta_{\cS}\setminus \Delta_{\cS}^{\geq 2}$, which is homeomorphic to the interior of $E_{\cS}$.
	
	Other affinoid torus fibrations for the family $f^{\circ}$ were constructed, e.g., in \cite{MazzonSchneider-toric,Pille-Schneider}. For $n=2$, archimedean Lagrangian torus fibration of the entire quartic $X_z$ was constructed symplectically in \cite{CG-K3}. For general K3 surfaces, Lagrangian fibrations can also be constructed from elliptic fibrations by a hyperk\"ahler rotation, see \cite[\textsection 1]{GW-3folds}. In all these constructions, the discriminant loci correspond to the singularities of $\bar{X}$, contained in the submaximal strata of the central fiber, see the left part of Figure \ref{fig:K3-simplex}. For instance, the singular fibers in \cite{CG-K3} are 24 pinched tori (focus-focus singularities in the terminology of Gross--Siebert) over 24 points in the submaximal faces of $\Delta_{\cS}$, six in each.
	
	Therefore, our construction, as an archimedean variant of~\cite{NXY}, potentially \emph{differs} from the above ones with respect to the location of the  discriminant. Indeed, even in case $n=2$, if we could extend our construction to a singular Lagrangian torus fibration of the entire $X_{z}$, then the discriminant would be contained in the $4$ shaded triangles in Figure \ref{fig:K3-E}, corresponding to $4$ vertices of $\Delta_{\cS}$.
\end{example}

	\begin{remark}[Affine structure]
		The Lagrangian torus fibration $\ll_{\cS}^{\theta}$ endows $E_{\cS}\setminus (\Ram(E_{\cS})\cup \dout E_{\cS})$ with an affine structure given by the action coordinates, see \cite[Theorem 2.12]{Evans-book}. In the preimage of a maximal face $\Delta_{I}$ of $\Delta_{\cS}$ we have $\pi^{*}\omega_{X}=0$ and $\hat{v}_i=v_i$, $\hat{w}_i=w_i$, $\hat{\alpha}_{i}=d\theta_i$, so formula \eqref{eq:omega-s} gives  $\omega^{\epsilon}_{q}=\epsilon \cdot \sum_{i\in I} (q\cdot dv_{i}+(1-q)dw_i)\wedge d\theta_{i}$. Taking into account the equalities $\sum_{i\in I} w_{i}=1$ and $\sum_{i\in I} m_{i}\theta_{i}=\theta$ and applying  \cite[Lemma 2.15]{Evans-book} we get affine coordinates of the form $(q\cdot v_{i}+(1-q)w_{i})_{i\in I\setminus \{j\}}$. In particular, for $q=0$ we get the standard coordinates $(w_i)_{i\in I\setminus \{j\}}$ of the simplex $\Delta_{I}$. In the preimage of $\Delta_{I}\times  C_{I}\subseteq E_{\cS}$ for a submaximal face $\Delta_{I}$, we get affine coordinate systems consisting of $n-1$ coordinates as above, and one coordinate of $X_{I}^{\circ}$, which depends on the form $\omega_{X}$.
	\end{remark}

\begin{remark}[Submaximal Calabi--Yau degenerations]\label{rem:submaximal}
	Our results are most powerful in case $\dim\Delta_{\cS}=n+1$, e.g., for \emph{maximal} Calabi--Yau degenerations. In fact, if $\dim \Delta_{\cS}<n$ then our results are vacuous. Consider the case $\dim \Delta_{\cS}=n$, so $\Delta_{\cS}$ has submaximal faces, but no maximal one. Then the expanded skeleton $E_{\cS}$ is a disjoint union of ivy-like charts $\Delta_{I}\times C_{I}$, where $\Delta_{I}$ is a submaximal face of $\Delta_{\cS}$, and $C_I$ is a Reeb space of the chosen Morse function $h_{I}\colon X_{I}^{\circ}\to [0,1]$. Since the Riemann surface $X_{I}^{\circ}$ is closed, $h_{I}$ has a maximum, so $\dout E_{\cS}\neq\emptyset$ and therefore $\ll_{\cS}^{\theta}$ has fibers of lower dimension. 
	
	Nonetheless, in the setting of Calabi--Yau degenerations this can be easily fixed. Indeed, assume that $f$ is an snc model chosen in Proposition \ref{prop:model}. Then each $X_{I}^{\circ}$ is an elliptic curve, diffeomorphic to $\mathbb{S}^{1}\times \mathbb{S}^{1}$. Replacing $\bar{h}_{I}$ by one of the projections, we get $E_{\cS}$ to be a disjoint union of products $\mathbb{S}^{1}\times \Delta_{I}$; and the same argument as in Lemma \ref{lem:ll-fibers}\ref{item:fiber-isotropic} shows that $\ll_{\cS}^{\theta}$ is a Lagrangian fibration.
\end{remark}

\section{Lagrangian fibrations at positive radius}

We keep the notation and assumptions from Section \ref{sec:expanded}. In Proposition \ref{prop:Lagranian-at-radius-zero}, we have constructed an ivy-like Lagrangian torus fibration $\ll_{\cS}$ of the part of the radius-zero fiber  corresponding to maximal and submaximal faces of $\Delta_{\cS}$. In this section, we use the symplectic connection to transport those Lagrangian tori to positive radius, and study their properties.

In order to apply  Proposition \ref{prop:omega}, we need to choose an open subset $W_{X}\subseteq X$ such that $f|_{\bar{W}_{X}}$ is proper; and the flows of the vector fields lifted by the symplectic connection do not escape $W_X$. If $f$ was proper to begin with, we simply take $W_X=X$. In general,  we take any open subset $V_{X}\subseteq X$ with compact closure, and obtain $W_{X}$ by slightly modifying $V_X$ near its  boundary. 
This modification was introduced in \cite[Setting 5.12]{FdBP_Zariski}, we summarize it in Setting \ref{setting:flow} below.  The reader only interested in the proper case may skip this part and substitute $W_X=X$ in what follows.

\begin{setting}[{see Figure \ref{fig:flow} or \cite[Figure 5]{FdBP_Zariski}}]\label{setting:flow}
Let $V_{X}$ be a domain with smooth boundary, such that $f|_{\bar{V}_{X}}$ is proper, $f^{-1}(0)$ is transverse to $\d V_{X}$, and the intersection $f^{-1}(0)\cap \d \bar{V}_{X}$ is contained in the union of the open sets $R_{i}^{\circ}$ from Lemma \ref{lem:rihat}\ref{item:r_covering}. Shrinking $\delta>0$ if needed we get that $\bigcup_{i}R_{i}^{\circ}$ contains some neighborhood $C_{X}$ of $\d \bar{V}_{X}$. Put $C\de \pi^{-1}(C_X)$. By Lemma \ref{lem:rihat}\ref{item:r_Ri},\ref{item:r_small}, on $C$ each distance function $\hat{r}_{i}$ is either equal to $\frac{1}{e}$ or $|f|^{1/m_{i}}$, hence is fiberwise constant. Thus each form $\hat{\alpha}_{i}$ is fiberwise zero there, so formulas \eqref{eq:omega_sharp} and \eqref{eq:omega_flat} imply that $\omega^{\sharp}$ and $\omega^{\flat}$ are fiberwise zero. Eventually, definition \eqref{eq:omega-s-intro} of $\omega^{\epsilon}_{q}$ gives a fiberwise equality $\omega_{q}^{\epsilon}|_{C}=\pi^{*}\omega_{X}|_{C}$ for every $q$ and $\epsilon$.

\begin{figure}[ht]
	\begin{tikzpicture}
		\path [fill=black!5] (0.2,2.3) to[out=0,in=180] (1.8,2.7) -- (1.8,-0.8) -- (0.2,-0.8) -- (0.2,2.3);
		\draw [decorate, decoration = {calligraphic brace}, very thick] (1.9,2.7) --  (1.9,-0.8);
		\node[right] at (2,0.95) {\small{$W_{X}$}};
		\path [fill=black!10] (0.2,2.3) to[out=0,in=180] (1.8,2.7) -- (1.8,1.7) to[out=180,in=0] (0.2,1.3) -- (0.2,2.3);
		\draw (0.2,2.3) to[out=0,in=180] (1.8,2.7);
		\draw (0.2,1.3) to[out=0,in=180] (1.8,1.7);
		\draw [decorate, decoration = {calligraphic brace}, very thick] (0.1,1.3) --  (0.1,2.3);
		\node[left] at (0.15,1.8) {\small{$C_{X}'$}};
		\draw[-stealth] (1,2.25)  to[out=20,in=-175]  (1.5,2.4);
		\draw[-stealth] (1,2)  to[out=20,in=-175]  (1.5,2.15);
		\draw[-stealth] (1,1.75)  to[out=20,in=-175]  (1.5,1.9);
		\draw (-0.5,3) -- (1.8,3);
		\draw (0.2,2) -- (3.2,2);
		\node at (3,1.75) {\small{$\d \bar{V}_{X}$}};
		\draw (-0.5,1) -- (1.8,1);
		\draw [decorate, decoration = {calligraphic brace}, very thick] (-0.6,1) --  (-0.6,3);
		\node[left] at (-0.6,2) {\small{$C_{X}$}};
		\draw[thick] (1,3.1) -- (1,1) to[out=-90,in=60] (0.9,0.4);
		\draw[thick] (0.9,0.6) -- (1.1,-0.2);
		\draw[thick] (1.1,0) -- (0.9,-0.8);
		\draw[->] (1,-0.9) -- (1,-1.6);
		\node[left] at (1,-1.3) {\small{$f$}};
		\draw (1,-2) ellipse (0.7 and 0.25);
		\draw[-stealth] (1,-2) -- (1.5,-2);
		\filldraw (1,-2) circle (0.04);
		\filldraw (1.7,-2) circle (0.03);
		\node at  (1.9,-1.8) {\small{$\delta$}};
	\end{tikzpicture}
	\caption{Setting \ref{setting:flow}: modifying the domain in case $f$ is not proper.}
	\label{fig:flow}
\end{figure}

Now, for $r<\delta$ let $\Phi_{r}\colon C_{X}\cap f^{-1}(0)\to X$ be the time $r$ flow of the symplectic lift of the radial vector field $\frac{\d}{\d r}$ on a fixed ray of $\C$, with respect to $\omega_{X}$: note that it is well defined since the restriction of $f^{-1}(0)$ to the collar $C_{X}$ is smooth. Shrinking $\delta>0$ if needed, we get that for some neighborhood $C'_{0}$ of $f^{-1}(0)\cap \d \bar{V}_{X}$ in $f^{-1}(0)\cap C_{X}$, the image $\Phi_{r}(C'_{0})$ is contained in $C_{X}$ for all $r<\delta$. Now, let $C'_{X}$ be the union of the images $\Phi_{r}(C'_{0})$, for all $r<\delta$ and all rays in $\C$. Put $W_{0}=f^{-1}(0)\cap (V_{X}\cup C'_{0})$ and define $W_{X}$ as the union of $C'_{X}$ and all connected components of $X\setminus C'_{X}$ which do not meet the union of images $\Phi_{r}(\d W_{0})$, see \cite[Setting 5.12]{FdBP_Zariski} for details.
\end{setting}

Having fixed $W_{X}$ and $\delta>0$ as in Setting \ref{setting:flow}, we choose $\epsilon_0>0$ and shrink $\delta>0$ further so that they satisfy Proposition \ref{prop:omega}. That is, for every $q\in (0,1]$ and $\epsilon\in (0,\epsilon_{0}]$, the forms $\omega_{q}^{\epsilon}$ are fiberwise symplectic on $f_{A}^{-1}(\D_{\delta,\log})\cap W$, where $W=\pi^{-1}(W_X)$. The choice of $W_{X}$ guarantees that the the symplectic lift of $\frac{\d}{\d r}$ with respect to each $\omega_{q}^{\epsilon}$ defines a flow $W\cap \d A\to W$ for all times $r<\delta$. Indeed, if $W\neq A$ then near $\d W$ we have fiberwise equality $\omega_{q}^{\epsilon}=\pi^{*}\omega_{X}$, so both forms define the same flow, which does not escape $W=\pi^{-1}(W_X)$ by construction of $W_X$.

\begin{notation}\label{not:flow}
	In the following, we replace $X$ and its A'Campo space $A$ by $W_{X}$ and its preimage in $A$. As usual, we identify $X\setminus D$ with its preimage $A\setminus \d A$. We fix one $\epsilon\in (0,\epsilon_0]$, and suppress it in the notation. This way, for every $q\in (0,1]$ we have a  flow 
	\begin{equation}\label{eq:flow}
		\Phi_{\bullet}^{q}\colon \d A\to A
	\end{equation}
	of the $\omega_{q}^{\epsilon}$-symplectic lift of the radial vector field on $\C_{\log}$, with respect to the submersion $(g,\theta)$. 
\end{notation}

\subsection{Moving Lagrangian tori by the symplectic connection}\label{sec:moving_Lagrangians}

Fix $z\in \D_{\delta}^{*}$, and put $\theta=\frac{z}{|z|}$, $\tau=\eta(-\frac{1}{\log|z|})$, so the flow $\Phi_{\tau}^{q}$ maps the radius-zero fiber $F_{\theta}=f_{A}^{-1}(0,\theta)$ to the smooth fiber $X_{z}=f^{-1}(z)$. We push the decomposition $F_{\theta}=\bigsqcup_{I}A_{I,\theta}$ to $X_{z}=\bigsqcup_{I}(X_{z}^{q})_{I}^{\circ}$, that is, we put $(X_{z}^{q})^{\circ}_{I}=\Phi_{\tau}^{q}(A_{I,\theta}^{\circ})$; and we define 
\begin{equation}\label{eq:regions}
	(X_{z}^{q})_{\cS}^{\sm}=\Phi_{\tau}^{q}(A_{\cS,\theta}^{\sm}),\quad
	(X_{z}^{q})_{\cS}^{\gen}=\Phi_{\tau}^{q}(A_{\cS,\theta}^{\gen}).
\end{equation}
This way, $(X_{z}^{q})_{\cS}^{\sm}$ is the union of pieces $(X_{z}^{q})^{\circ}_{I}$ for all maximal and submaximal faces $\Delta_{I}$ of the skeleton $\Delta_{\cS}$; and $(X_{z}^{q})_{\cS}^{\gen}$ is the union of $\Int_{X_z} (X_{z}^{q})_{I}^{\circ}$ for all maximal faces of $\Delta_{\cS}$. Note that $(X_{z}^{q})^{\sm}_{\cS}$ is a submanifold of $X_{z}^{q}$ of codimension $0$ with boundary and corners; and $(X_{z}^{q})^{\gen}_{\cS}$ is its open subset.  

We remark that the sets $\Int_{X_{z}}(X_{z}^{q})^{\circ}_{I}$ play a similar role as the sets $E_{I}$ introduced in \cite[\textsection 3.1]{Li_survey}. 
\smallskip

We have the following immediate consequence of Proposition \ref{prop:Lagranian-at-radius-zero}.

\begin{prop}\label{prop:Lagranian-at-positive-radius} 
	For every $q\in (0,1]$, the composition
\begin{equation}\label{eq:Lagranian-at-positive-radius} 
	\ll_{\cS}^{\theta}\circ (\Phi_{\tau}^{q})^{-1}\colon (X_{z}^{q})_{\cS}^{\sm}\to E_{\cS}
\end{equation}
is an ivy-like Lagrangian torus fibration with respect to the K\"ahler form $\omega_{q}^{\epsilon}$ on $X_{z}$. It is a Lagrangian fibration in the usual sense when restricted to $(X_{z}^{q})_{\cS}^{\gen}$; or if condition \eqref{eq:Cst} holds.
\end{prop}

\begin{remark}[Case $q=0$]\label{rem:q=0}
	By Proposition \ref{prop:omega}\ref{item:Kahler},\ref{item:symplectic_zero}, the form $\omega_{0}^{\epsilon}$ is fiberwise symplectic both at positive radius, and on $A_{\cS}^{\gen}$. Therefore, we have a flow 
	\begin{equation*}
		\Phi_{\tau}^{0}\colon A_{\cS,\theta}^{\gen}\to (X_{z}^{0})_{\cS}^{\gen}\subseteq X_{z}.
	\end{equation*}
	Moreover, by Proposition \ref{prop:Lagranian-at-radius-zero}\ref{item:ll-maximal} we have a Lagrangian torus fibration
	\begin{equation*}
		\ll_{\cS}^{\theta} \circ (\Phi_{\tau}^{0})^{-1} \colon (X_{z}^{0})_{\cS}^{\gen} \to \Delta_{\cS}^{\gen},
	\end{equation*}
	where $\Delta_{\cS}^{\gen}$ is the union of relative interiors of maximal faces of $\Delta_{\cS}$. 
\end{remark}

\begin{remark}[Choices within the definition of Lagrangian fibration at positive radius, cf.\ Remark~\ref{rem:canonical_choice}]\label{rem:canonical_choice-positive-radius}
	The fibration~\eqref{eq:Lagranian-at-positive-radius} depends on: the choice of Morse functions $h_{I}\colon X_{I}^{\circ}\to \R$ for each submaximal face $\Delta_{I}$ of $\Delta_{\cS}$, the distance functions $\hat{r}_{i}$ used to define coordinates $\hat{v}_i$, and the choices made in the definition of the symplectic connection. The choice of $h_{I}$ is canonical under condition \eqref{eq:Cst}, which can be ensured in the setting of maximal Calabi--Yau degenerations, by Proposition \ref{prop:model}\ref{item:sk_P1}.
\end{remark}

\subsection{The associated metrics}\label{sec:metrics}

Motivated by \cite[Conjectures 1 and 2]{KS_Conjecture}, we study the behavior of K\"ahler metrics 
\begin{equation*}
	g_{q}=\omega_{q}^{\epsilon}(\sdot,J\sdot)
\end{equation*}
on $X_{z}$ as $z\rightarrow 0$. For maximal Calabi--Yau degenerations, \cite[Conjecture 1]{KS_Conjecture} expects that the fiberwise Ricci-flat metric should converge, after appropriate re-scaling, to one satisfying a real Monge--Amp\'ere equation. In our setting, an example of such metric is given by $\omega^{\flat}$, so looking at formula \eqref{eq:omega-s-intro}, it is reasonable to expect such limiting behavior as $q\rightarrow 0$. This motivates considering a family of K\"ahler metrics parametrized both by $z\in \D_{\delta}^{*}$ and $q\in (0,1]$; and studying its limit as $q,|z|\rightarrow 0$. 
\smallskip

Before we state our result in Proposition \ref{prop:GH}, we need some preparations. We denote by $T_{\textnormal{vert}} A$ the vertical tangent bundle to the A'Campo space, i.e., the kernel of the map $df_{A}\colon TA\to T\C_{\log}$. 

Recall that we identify $A\setminus \d A$ with its image $X\setminus D$; in particular we have there the standard complex structure $J$ induced from $X$. We study the following sections of $\otimes^{2} T^{*}(A\setminus \d A)$:
\begin{equation*}\begin{split}
	g^{\sharp}=\omega^{\sharp}(\sdot, J\sdot),\ g^{\flat}=\omega^{\flat}(\sdot,J\sdot),\ g_{q}=\omega_{q}^{\epsilon}(\sdot,J\sdot) \\ 
	\mbox{and}\quad
	g^{\sharp}_{\new}=t\cdot g^{\sharp},\  g^{\flat}_{\new} = t\cdot g^{\flat},
	\ g_{q,\new}=t\cdot g_{q}.
\end{split}\end{equation*}
Since $\omega^{\sharp}$ and $\omega^{\flat}$ are fiberwise $J$-compatible, the above forms are fiberwise symmetric, i.e., their images in $\otimes^{2}T^{*}_{\textnormal{vert}} (A\setminus \d A)$ lie in $S^{2}T^{*}_{\textnormal{vert}} (A\setminus \d A)$. We denote these images by the same letters.

\begin{lema}\label{lem:limit-metrics}
	The following hold.
	\begin{enumerate}
		\item \label{item:fiberwise-metrics}
		On $A\setminus \d A$, we have fiberwise equalities
	\begin{equation}\label{eq:fiberwise-metrics}
		g^{\sharp} = \sum_{i=1}^{N} \left(\frac{1}{m_{i}\hat{\sigma}_{i}} (d\hat{v}_{i})^2 + \hat{\sigma}_{i}\cdot m_i\, \hat{\alpha}_{i}^{2}\right)+\check{g}
			\quad\mbox{and}\quad
				g^{\flat} = \sum_{i\in \cS} \left(\frac{1}{t} \cdot (d\hat{w}_{i})^{2}+t\cdot m_{i}^{2}\, \hat{\alpha}_{i}^{2}\right)+\tilde{g},
	\end{equation}
	where $\hat{\sigma}_i= \hat{t}_{i}^2+\hat{t}_{i}\hat{u}_{i}^2$, see formula \eqref{eq:basic_functions_hat}; and $\tilde{g},\check{g}$ are smooth, symmetric $2$-forms on $A$ which are bounded with respect to the natural coordinates of $X$.
	\item\label{item:metrics-extend} The sections $g^{\flat}_{\new}$, $g^{\sharp}_{\new}$ and $g_{q,\new}$ of $S^{2}T_{\textnormal{vert}}^{*}(A\setminus \d A)$ extend to smooth sections of $S^{2}T^{*}_{\textnormal{vert}} A$.
	\item\label{item:limit-metrics} On $\d A$, we have fiberwise equalities
	\begin{equation}\label{eq:limit-metrics}
		g^{\sharp}_{\new} = \sum_{i=1}^{N} \frac{1}{m_{i}}\, \eta'(\hat{w}_i) (d \hat{w}_i)^2
	\quad\mbox{and}\quad
		g^{\flat}_{\new} = \sum_{i\in \cS} (d\hat{w}_{i})^2,
	\end{equation}
	where $\eta'(\hat{w}_{i}) (d\hat{w}_i)^2$ extends to a smooth form on $\d A$, vanishing at the zero locus of $\hat{w}_i$.
	\item\label{item:trash-metric} After possibly shrinking $\epsilon>0$, for every $q\in [0,1]$ we have a fiberwise equality
	\begin{equation*}
		g_{q,\new}=(1-q)\cdot \epsilon \cdot \sum_{i\in \cS} (d\hat{w}_{i})^{2}+\tilde{g}_{q},
	\end{equation*}
	where $\tilde{g}_{q}$ is zero on $\d A$ and a fiberwise Riemannian metric on $A\setminus \d A$.
	\end{enumerate}
\end{lema} 
\begin{proof}
	Put $\hat{\alpha}_{i}^{c}=\hat{\alpha}_{i}\circ J$. Definitions \eqref{eq:basic_functions_hat} give fiberwise equalities 
	\begin{equation}\label{eq:dcw}
		\hat{w}_{i}=-m_{i} t \hat{s}_{i},\quad \mbox{so} \quad d^{c}\hat{w}_{i}=m_{i}t\,\hat{\alpha}_{i}\quad\mbox{and}\quad  \hat{\alpha}_{i}^{c}=-\frac{1}{m_{i}t}\, d\hat{w}_{i}.
	\end{equation}

	First, we prove all claims about $g^{\flat}$. By formula \eqref{eq:omega_flat} we have $\omega^{\flat}=-\sum_{i\in \cS} m_{i} d\hat{w}_{i}\wedge  \hat{\alpha}_{i}+\tilde{\omega}$, where $\tilde{\omega}=-\sum_{i\in \cS} m_{i}\hat{w}_{i}\, d\hat{\alpha}_{i}$. Substituting formulas \eqref{eq:dcw} we get the second equality of \eqref{eq:fiberwise-metrics}, with $\tilde{g}\de \tilde{\omega}(\sdot, J\sdot)$.  By Lemma \ref{lem:comparison}\ref{item:alpha-comparison}, each form $d\hat{\alpha}_{i}$ is a pullback of a smooth $2$-form on $X$; hence $d\hat{\alpha}_{i}(\sdot, J\sdot)$ is a smooth section of $\otimes^{2}T^{*}X$. It follows that $\tilde{g}$ extends to a smooth section of $S^{2}T^{*}A$, bounded in the natural coordinates of $X$, as claimed in \ref{item:fiberwise-metrics}. Multiplying the result by $t$, we get a fiberwise equality
	\begin{equation*}
		g^{\flat}_{\new}=\sum_{i\in \cS} \left( (d\hat{w}_i)^2+t^2\cdot m_{i}^2\, \hat{\alpha}_i^2\right)+t\cdot \tilde{g}.
	\end{equation*}
	By Proposition \ref{prop:AX_hat_smooth}\ref{item:AX-vbar-smooth}, the functions $\hat{w}_{i}$ and forms $\hat{\alpha}_{i}$ are smooth, so the right-hand side extends to a smooth section of $S^{2}T^{*}A$, as claimed in \ref{item:metrics-extend}. Substituting $t=0$ we get the second equality of \eqref{eq:limit-metrics}.
	\smallskip

	Now, we prove 
	the claims about $g^{\sharp}$. Recall from formula \eqref{eq:basic_functions_hat} that $\hat{v}_{i}=\hat{t}_{i}-\hat{u}_{i}$, where $\hat{u}_{i}=\eta(\hat{w}_{i})=(1-\log(\hat{w}_{i}))^{-1}$. We have the  following equalities, cf.\ \cite[Lemma 3.21(g)]{FdBP_Zariski}:
	\begin{equation*}
	d\hat{u}_{i}=\eta'(\hat{w}_{i})\, d\hat{w}_{i}=-(1-\log(\hat{w}_{i}))^{-2}\cdot (-\hat{w}_{i}^{-1})\, d\hat{w}_{i}=\hat{u}_i^{2}\hat{w}_{i}^{-1}\, d\hat{w}_{i}=\hat{u}_{i}^{2} \hat{t}_{i}\cdot t^{-1}\, d\hat{w}_i, 
	\end{equation*}
	 where in the last equality we used an identity $t=\hat{t}_{i}\hat{w}_{i}$. This identity gives also a fiberwise equality $d\hat{t}_{i}=d(t\hat{w}_{i}^{-1})=-t\hat{w}_{i}^{-2}\, d\hat{w}_{i}=-\hat{t}_{i}^{2}\cdot t^{-1}\, d\hat{w}_{i}$. Combining those two, we get a fiberwise equality
	\begin{equation}\label{eq:dv}
		d\hat{v}_{i}=-\hat{t}_{i}^{-2}\cdot t^{-1}\, d\hat{w}_{i}-\hat{t}_{i}\hat{u}_{i}^2\cdot t^{-1}\, d\hat{w}_{i}=
		-t^{-1}\hat{\sigma}_{i}\, d\hat{w}_{i},
	\end{equation}
	cf.\ \cite[Lemma 3.22(e)]{FdBP_Zariski}. Using formula \eqref{eq:dcw}, we get fiberwise equalities
	\begin{equation}\label{eq:dcv}
		d^{c}\hat{v}_{i}=-m_{i}\hat{\sigma}_{i}\,\hat{\alpha}_{i}\quad\mbox{and}\quad \hat{\alpha}_{i}^{c}=\frac{1}{m_{i}\hat{\sigma}_{i}}\, d\hat{v}_{i}.
	\end{equation}
	By definition \eqref{eq:omega_sharp} of $\omega^{\sharp}$, we have $\omega^{\sharp}=\sum_{i=1}^{N} d\hat{v}_{i}\wedge \hat{\alpha}_{i}+\check{\omega}$, where $\check{\omega}=\sum_{i=1}^{N} \hat{v}_{i}\, d\hat{\alpha}_{i}$. Substituting equalities \eqref{eq:dcv} we get the remaining part of \ref{item:fiberwise-metrics}, where like before $\check{g}\de \check{\omega}(\sdot, J\sdot)$ is a smooth section of $S^{2}T_{\textnormal{vert}}^{*}A$, bounded in natural coordinates of $X$. Multiplying by $t$, we get a fiberwise equality
	\begin{equation*}
		g_{\new}^{\sharp} = \sum_{i=1}^{N} \left(\frac{t}{m_i\hat{\sigma}} (d\hat{v}_{i})^2 + t\hat{\sigma}_{i}\cdot m_i\, \hat{\alpha}_{i}^{2}\right)+t\cdot \check{g}
		\overset{\tiny{\mbox{\eqref{eq:dv}}}}{=}\sum_{i=1}^{N}\left(-\frac{1}{m_{i}} d\hat{w}_i\cdot d\hat{v}_i+t\hat{\sigma}_i\cdot m_i\, \hat{\alpha}_i^2\right)+t\cdot\check{g}.
	\end{equation*}
	Recall that by Proposition \ref{prop:AX_hat_smooth}\ref{item:AX-vbar-smooth} the functions $\hat{v}_i$, $\hat{w}_i$ and the forms $\hat{\alpha}_i$ are smooth. By Lemma \ref{lem:tsigma-smooth} below, the functions $t\hat{\sigma}_i$ are smooth, too, so the right-hand side extends to a smooth, fiberwise symmetric $2$-form on $A$, as claimed in \ref{item:metrics-extend}. Substituting $t=0$, we get
	\begin{equation*}
		g^{\sharp}_{\new}|_{\d A}=-\sum_{i=1}^{N} \frac{1}{m_{i}} d\hat{w}_i\cdot d\hat{v}_i.
	\end{equation*}
	It remains to show that $-d\hat{w}_i\cdot d\hat{v}_i$ is a smooth extension of $\eta'(\hat{w}_{i}) (d\hat{w}_i)^2$ to $\d A$, vanishing on $\{\hat{w}_i=0\}$.
	
	The forms $d\hat{w}_{i}$ and $d \hat{v}_i$ are smooth on $\d A$. Clearly, $d\hat{w}_i=0$ on the zero locus of $\hat{w}_i$, so $-d\hat{w}_i\cdot d\hat{v}_i$ is a smooth form on $\d A$, vanishing on $\{\hat{w}_i=0\}$. On the remaining part of $\d A$ we have $\hat{t}_i=t\hat{w}_{i}^{-1}=0$, so $d\hat{v}_i=-d\hat{u}_i=-\eta'(\hat{w}_{i})\, d\hat{w}_i$, which proves the claim.
	\smallskip
	
	It remains to prove part \ref{item:trash-metric}. Let $g_{X}=\omega_{X}(\sdot, J\sdot)$ be the K\"ahler metric on $X$.  We have 
		\begin{equation*}
			g_{q,\new}
			=
			t\cdot \pi^{*} g_{X}+t\cdot \epsilon\cdot (q\cdot g^{\sharp}+(1-q)\cdot g^{\flat})
			=
			(1-q)\cdot\epsilon\cdot \sum_{i\in \cS} (d\hat{w}_{i})^2+
			t \cdot \tilde{g}_q',
		\end{equation*}
		where by the second formula in \eqref{eq:fiberwise-metrics} we have
		\begin{equation*}
			\tilde{g}_q'=
			\left(\tfrac{1}{2}\pi^{*}g_{X}+\epsilon \cdot q \cdot g^{\sharp}\right)+\left(\tfrac{1}{2} \pi^{*}g_{X}+\epsilon \cdot (1-q) \cdot \tilde{g}\right)+t\cdot \epsilon \cdot (1-q)\cdot \sum_{i\in \cS} m_{i}^{2}\hat\alpha_{i}^{2}.
		\end{equation*}
		The first summand above is the fiberwise Riemannian metric associated to the form $\frac{1}{2}\omega_{1}^{2q\epsilon}$ which is fiberwise K\"ahler for $\epsilon<\frac{\epsilon_0}{2}$ by Proposition \ref{prop:omega}\ref{item:Kahler}. By \ref{item:fiberwise-metrics} the symmetric form $\tilde{g}$ is bounded in the natural coordinates of $X\setminus D$, so since $g_{X}$ is a Riemannian metric on $X\setminus D$ and $f_{A}$ is proper, for sufficiently small $\epsilon\geq 0$ the second summand is a fiberwise Riemannian metric, too. Since $t<\frac{-1}{\log \delta}$, shrinking $\delta>0$ we can make $t>0$ as small as we wish, so again by properness of $f_A$ we conclude that the whole sum $\tilde{g}_q'$ is a fiberwise Riemannian metric on $A\setminus \d A$, as needed.	
\end{proof}

In the above proof, we used  Lemma \ref{lem:tsigma-smooth} which asserts that the functions $t\hat{\sigma}_i$ are smooth in the A'Campo space $A$. The proof is similar to the proof of Proposition \ref{prop:AX_hat_smooth}. 
The key idea is that the function $t$ allows to \enquote{flatten} all derivatives of $\hat{\sigma}_i$. 

\begin{lema}\label{lem:tsigma-smooth}
	The functions 
	$t\hat{\sigma}_{i}$ for $i\in \{1,\dots, N\}$ extend to smooth functions on $A$. 
\end{lema}
\begin{proof}	
	We work in the preimage of an adapted chart $U_X$, say with associated index set $\{1,\dots, k\}$; and consider its subset $U_{1}=\{w_{1}>\frac{1}{n+2}\}$; with coordinates \eqref{eq:AC-chart_hat} given by $(g;\hat{v}_{2},\dots,\hat{v}_k;\rest)$. 
	
	Assume first that $i>k$, i.e., $D_i\cap U_X=\emptyset$. Then $\hat{r}_{i}|_{U_{1}}>0$, so $\hat{t}_{i}=-(m_{i}\hat{r}_{i})^{-1}$ is smooth on $U_1$. Thus $\hat{u}_i|_{U_1}=\hat{t}_i|_{U_1}-\hat{v}_i|_{U_1}$ is smooth by Proposition \ref{prop:AX_hat_smooth}\ref{item:AX-vbar-smooth}, and therefore  $\hat{\sigma}_{i}|_{U_1}$ is smooth, too. 
	
	Assume now that $i=1$. We have $\hat{w}_{1}|_{U_1}>0$ by definition of $U_1$, so since $\hat{w}_{1}$ and $\eta|_{(0,1]}$ are smooth, so is their composition $\hat{u}_{1}$. By Proposition \ref{prop:AX_hat_smooth}\ref{item:AX-vbar-smooth}, so are the functions $\hat{t}_{1}=\hat{v}_{1}-\hat{u}_1$ and $\hat{\sigma}_1$. 
	
	It remains to consider the case $i\in \{2,\dots, k\}$, i.e., when $\hat{v}_i$ is a coordinate of $U_1$. Now $\hat{\sigma}_i$ might not be smooth, but we will show that $t\hat{\sigma}_i$ is. Using the identity $t=t_iw_i$, we write 
	\begin{equation}\label{eq:tsigma}
		t\hat{\sigma}_i=w_i\cdot (t_i\hat{t}_i^2+t_i\hat{t}_i \hat{u}_i^2).
	\end{equation}
	Recall that in the proof of Proposition \ref{prop:AX_hat_smooth}\ref{item:AX-vbar-smooth} we have introduced a slight modification of the algebra $\cA_{i}$ from \cite[formula (40)]{FdBP_Zariski}. The functions $t_{i}$ and $u_{i}$ belong to this algebra. Indeed, in the notation of loc.\ cit.\ we have $t_i,u_i\in \cW_i$, $t_i\in \cR_{i}$ by formula (38) and $u_i\in \cW_i$ by Lemma 3.41(a). Now our  modified algebra $\cA_i$ was defined as in loc.\ cit, with the functions $a\in \cT_{i}$ replaced by the single function $a$ from Lemma \ref{lem:comparison}. Thus Lemma \ref{lem:comparison}\ref{item:t-comparison},\ref{item:u-comparison} implies that  $\hat{t}_i-t_i\in \cA_i$ and $\hat{u}_i-u_i\in \cA_i$. We conclude that $\hat{t}_i,\hat{u}_i\in \cA_i$, so $t\hat{\sigma}_{i}\in w_i\cdot \cA_i$ by formula \eqref{eq:tsigma}.
	
	It follows that the function $t\hat{\sigma}_i$ lies in the algebra $\cW_{i}'\cdot \cA_i\cdot \cP$ introduced in p.\ 206 loc.\ cit: here $\cW_{i}'$ and $\cP$ are some algebras containing $w_i$ and $1$, respectively. As observed in the proof of Proposition \ref{prop:AX_hat_smooth}\ref{item:AX-vbar-smooth}, the proof of Lemma 3.49 loc.\ cit. shows that all elements of the algebra $\cW_{i}'\cdot \cA_i\cdot \cP$ are smooth on $U_1$. Thus $t\hat{\sigma}_i$ is smooth on $U_1$; as needed.
\end{proof}

\subsection{Admissible paths}\label{sec:paths}

We have set up a bi-parametric approach to study the degeneration of K\"ahler manifolds $X_{z}^{q}$, where the first parameter $z$ determines the fiber $X_z$ of the degeneration, and the second one, $q$, determines the K\"ahler form $\omega_{q}^{\epsilon}$ within the family \eqref{eq:omega-s-intro}. To prove our main results, we choose some paths in this bi-parametric model. First, we distinguish \emph{admissible paths} as those used in  Theorem~\ref{theo:CY}\ref{item:CY_GH} and Theorem~\ref{theo:CY_af}\ref{item:CY_Lagrangian_volume},\ref{item:assflat}. That is, we pose the following definition.

\begin{definition}[see Figure \ref{fig:any_path}]\label{def:admissible_path}
	A smooth path $\gamma=(\gamma_1,\gamma_2)\colon [0,\rho)\to \C_{\log}\times [0,1]$ is \emph{admissible} if $|\gamma_1(0)|=\gamma_2(0)=0$ and $|\gamma_1(\tau)|>0$, $\gamma_2(\tau)>0$ for $\tau>0$. 
\end{definition}

Let $(r,\theta)$ be coordinates on $\C_{\log}=[0,\infty)\times \mathbb{S}^{1}$. In our setting, it is natural to replace $r$ by $t=e^{-\frac{1}{|f|}}$, see formula \eqref{eq:basic functions_t}: indeed, both the volume $\vol^{\Omega}_{\new}$ and the metrics $g_{q,\new}$, defined to give a finite nonzero limits as $|f|\rightarrow 0$, required a rescaling by a power of $t$, see Sections \ref{sec:intro} and \ref{sec:metrics}. So, we endow $\C_{\log}$ with another smooth coordinate system, not compatible with $(r,\theta)$, given by $(t,\theta)=(e^{-\frac{1}{r}},\theta)$. 

For any admissible path $\gamma\colon [0,\rho)\to \C_{\log}\times [0,1]$, the restriction $\gamma|_{[0,\rho')}$ for some $\rho'>0$ can be parametrized, in the coordinates $(t,\theta)$ of $\C_{\log}$, as $\gamma(h)=(h,\theta(h),q(h))$, where $\theta\colon [0,\rho')\to \mathbb{S}^1$ and $q\colon [0,\rho')\to [0,1]$ are continuous functions whose restrictions to $(0,\rho')$ are smooth. We call such a parametrization \emph{normal}. We can now distinguish paths used in Theorem \ref{theo:CY_af}\ref{item:CY_asymptotic specialty}.

\begin{definition}\label{def:special_path} Let $\gamma(h)=(h,\theta(h),q(h))$ be a normal parametrization of an admissible path $\gamma$. We say that $\gamma$ is \emph{tame} if the function $q$ is $\cC^{1}$.
\end{definition}

\begin{example}[see Figure \ref{fig:special_path}]\label{ex:special}
	Let $\gamma=(\gamma_{1},\gamma_{2})\colon [0,\rho)\to \C_{\log}\times [0,1]$ be an admissible path such that for all $\tau\in [0,\rho)$ we have $|\gamma_{1}(\tau)|=\tau$, where the absolute value refers to the natural coordinates $(r,\theta)$ on $\C_{\log}$, see Theorem \ref{theo:CY_af}\ref{item:CY_asymptotic specialty}. We claim that $\gamma$ is  tame. Indeed, the normal parametrization of $\gamma$ is $(h,\theta(h),q(h))$, where $\theta(h)$ is the argument of $\gamma_{1}(e^{-\frac{1}{h}})$, and $q(h)=\gamma_{2}(e^{-\frac{1}{h}})$ is smooth, as needed.
\end{example}

\begin{notation}\label{not:admissible}
	Let $h \mapsto (h,\theta(h),q(h))$ be a normal parametrization of an admissible path. We denote by $X_{h}^{\gamma}$ the fiber $(t,\theta)^{-1}(h,\theta(h))\subseteq A$, equipped with a $2$-form $\omega^{\gamma}_{h}\de \omega^{\epsilon}_{q(h)}$ defined in formula \eqref{eq:omega-s-intro}. We fix $\epsilon>0$ sufficiently small so that Proposition \ref{prop:omega} and Lemma \ref{lem:limit-metrics}\ref{item:trash-metric} hold. For a fixed $q>0$, we write $\Psi_{h}^{q}=\Phi_{\eta(h)}^{q}$, so $\Psi_{\bullet}^{q}$ is a reparametrization of the flow \eqref{eq:flow}  such that $\Psi_{h}^{q}$ maps $X_{0}^{\gamma}$ to $X_{h}^{\gamma}$. We define regions $(X_{h}^{\gamma})_{\cS}^{\gen}\subseteq (X_{h}^{\gamma})_{\cS}^{\sm}\subseteq X_{h}^{\gamma}$ as in Section \ref{sec:moving_Lagrangians}, i.e., as images of $(X_{0}^{\gamma})_{\cS}^{\gen}$ and $(X_{0}^{\gamma})_{\cS}^{\sm}$ by the flow $\Psi_{h}^{q(h)}$. 	
	Thus on $(X_{h}^{\gamma})^{\sm}_{\cS}$ we have an ivy-like Lagrangian torus fibration \eqref{eq:Lagranian-at-positive-radius}, with respect to the form $\omega_{h}^{\gamma}$. We write $g^{\gamma}_{h,\new}\de h\cdot \omega_{h}^{\gamma}(\sdot,J\sdot)$ for the rescaled metric on $X_{h}^{\gamma}$. 	
	We denote by $A^{\gamma}\subseteq A$ the union of fibers $X_{h}^{\gamma}$; and put $A^{\gamma}_{+}\de A^{\gamma}\setminus X_{0}^{\gamma}$. 
\end{notation}

\subsection{The Gromov--Hausdorff limit}\label{sec:GH}

We keep Notation \ref{not:admissible}. The following result is a direct consequence of Lemma \ref{lem:limit-metrics}\ref{item:trash-metric}.

\begin{lema}\label{lem:g0}
	Let $\gamma$ be an admissible path. Then for every $h\geq 0$ we have 
	\begin{equation*}
		g^{\gamma}_{h,\new}=(1-q(h))\cdot\epsilon\cdot \sum_{i\in \cS} (d\hat{w}_{i})^{2}+ \tilde{g}_h^{\gamma},
	\end{equation*} 
	where $\tilde{g}_{0}^{\gamma}=0$, and for $h>0$, $\tilde{g}_{h}^{\gamma}$ is a Riemannian metric on $X_{h}^{\gamma}$. In particular, $g_{0,\new}^{\gamma}=\epsilon \cdot \sum_{i\in \cS} (d\hat{w}_i)^2$.
\end{lema}

We can now prove the main result of this section. 
For a definition and basic properties of the Gromov--Hausdorff convergence we refer to \cite[\textsection 7.3, 7.4]{BBS_metric}.

\begin{prop}\label{prop:GH}
	Let $\gamma$ be an admissible path. For $h>0$, we equip the fiber $X_{h}^{\gamma}$ with a Riemannian metric $\frac{1}{\epsilon}\cdot g^{\gamma}_{h,\new}$. Then we have Gromov--Hausdorff convergence: 
	\begin{equation*}
		\lim_{h\rightarrow 0} (X_{h}^{\gamma}, X_{h}^{\gamma}\setminus (X_{h}^{\gamma})_{\cS}^{\sm}) =(\Delta_{\cS},\Delta_{\cS}^{\geq 2}),
	\end{equation*}
	where $\Delta_{\cS}\subseteq \Delta_{D}\subseteq \R^{N}$ is the essential skeleton, see Section \ref{sec:dual_complex}, equipped with the standard metric of $\R^{N}$; and $\Delta_{\cS}^{\geq 2}\subseteq \Delta_{\cS}$ is the union of faces of $\Delta_{\cS}$ of codimension at least $2$.
\end{prop}
\begin{proof}
	For $i\in \{1,\dots, N\}$ put $\hat{w}_{i}'=\hat{w}_{i}$ if $i\in \cS$ and $\hat{w}_{i}'=0$ otherwise, and consider a smooth map $\mu=(\hat{w}_{1}',\dots,\hat{w}_{N}')\colon A\to \R^{N}$. It is easy to see that 
	$\mu(A_{I}^{\circ})=\Delta_{I}$ for every face $\Delta_{I}$ of $\Delta_{\cS}$,  cf.\ Lemma \ref{lem:product} or \cite[Lemma 3.10]{FdBP_Zariski}, so 
	$\mu(X_{0}^{\gamma})=\Delta_{\cS}$ and $\mu(X_{0}^{\gamma}\setminus (X_{0}^{\gamma})_{\cS}^{\sm})=\Delta_{\cS}^{\geq 2}$.
	
	For any $h\in [0,\delta)$, the restriction of $\mu$ to a fiber $X_{h}^{\gamma}$ gives maps of compact metric spaces  
	\begin{equation*}
		\mu_{h}\colon X_{h}^{\gamma}\to \Delta_{\cS}\cup \mu(X_{h}^{\gamma}),
		\quad\mbox{and}\quad 
		\mu_{h}^{\geq 2} \colon X_{h}^{\gamma}\setminus (X_{h}^{\gamma})_{\cS}^{\sm}\to \Delta_{\cS}^{\geq 2}\cup \mu(X_{h}^{\gamma}\setminus (X_{h}^{\gamma})_{\cS}^{\sm}).
	\end{equation*} 
	Let $d_{h}$, $d_{\std}$ be the distance functions associated to the metric $\frac{1}{\epsilon}\cdot g_{h,\new}^{\gamma}$ on $X_{h}^{\gamma}$ and to the euclidean metric on $\Delta_{\cS}$, respectively. Lemma \ref{lem:g0} implies that for any two points $x,y\in X_{h}^{\gamma}$ we have $(1-q(h))^{\frac{1}{2}}\cdot d_{\std}(\mu_{h}(x),\mu_{h}(y))\leq d_{h}(x,y)\leq d_{\std}(\mu_{h}(x),\mu_{h}(y))+D_{h}$, where $D_{h}$ is the diameter of $X_{h}^{\gamma}$ with respect to the metric $\frac{1}{\epsilon}\cdot\tilde{g}_{h}^{\gamma}$.  Lemma \ref{lem:g0} also implies that as $h\rightarrow 0$, the metric $\tilde{g}_{h}^{\gamma}$ converges to $0$ as a tensor, hence $D_{h}\rightarrow 0$ by compactness of $X_{h}$. Since $q(h)\rightarrow 0$ as $h\rightarrow 0$, too, we conclude that the distortion of $\mu_h$ and $\mu_{h}^{\geq 2}$, see \cite[Definition 7.1.4]{BBS_metric}, converges to $0$ as $h\rightarrow 0$. 
	
	Continuity of $\mu$ implies that for every $\eta>0$ there is $\delta_0>0$ such that for every $h\in [0,\delta_0]$, each point of $\Delta_{\cS}$ is at distance at most $\eta$ to $\mu(X_{h}^{\gamma})$. This means that $\mu(X_{h}^{\gamma})$ is an $\eta$-net for $\Delta_{\cS}\cup \mu(X_{h}^{\gamma})$, see \cite[Definition 1.6.1]{BBS_metric}. Similarly, $\mu(X_{h}^{\gamma}\setminus (X_{h}^{\gamma})_{\cS}^{\sm})$ is an $\eta$-net for $\Delta_{\cS}^{\geq 2}\cup \mu(X_{h}^{\gamma}\setminus (X_{h}^{\gamma})_{\cS}^{\sm})$. Shrinking $\delta_0>0$ if needed, we infer that $\mu_{h}$ and $\mu_h^{\geq 2}$ are $\eta$-isometries, see Definition 7.3.27 loc.\ cit. The required Gromov--Hausdorff  convergence follows from Corollary 7.3.28(2) loc.\ cit.
\end{proof}

\subsection{Summary}

Before we specialize to the Calabi--Yau case, we summarize the general results obtained so far.

\begin{theo}\label{theo:general}
	Let $(X,\omega_{X})$ be a K\"ahler manifold, let $f\colon X\to \C$ be a holomorphic function whose unique singular fiber $f^{-1}(0)$ is snc, and let $D_1,\dots, D_{N}$ be its irreducible components. Choose a subset $\cS\subseteq \{1,\dots, N\}$ so that $\Delta_{\cS}$ has a maximal or a submaximal face, i.e., there is an $I\subseteq \cS$ such that $\#I\geq \dim X-1$ and $\bigcap_{i\in I}D_{i}\neq \emptyset$. Moreover, choose an open subset $V\subseteq X$ such that $f|_{\bar{V}}$ is proper, let $W$ be its slight modification as in Setting \ref{setting:flow} (one can take $W=X$ if $f$ is proper), and let $W_{z}=f^{-1}(z)\cap W$. Then there is a $\delta>0$ such that for every $\epsilon>0$ small enough, the following hold.
	\begin{enumthm}
		\litem{K\"ahler potential} \label{item:thm_Kahler} For every $z\in \D_{\delta}^{*}$ and every $q\in [0,1]$, the form $\omega_{q}^{\epsilon}|_{W_{z}}$ defined in \eqref{eq:omega-s-intro} is K\"ahler. 
		\litem{Lagrangian fibration}\label{item:Lagrangians} For every $z\in \D^{*}_{\delta}$ and every $q\in (0,1]$, formula \eqref{eq:Lagranian-at-positive-radius} defines an ivy-like Lagrangian torus fibration of a codimension $0$ submanifold $(W_{z}^{q})_{\cS}^{\sm}\subseteq W_{z}$ with boundary and corners, with respect to 
		$\omega_{q}^{\epsilon}$. If condition \eqref{eq:Cst} holds then  \eqref{eq:Lagranian-at-positive-radius} is a Lagrangian torus fibration in the usual sense.
		\litem{Gromov--Hausdorff limit} \label{item:GH} Fix a path $(z,q)\colon [0,\delta) \to \C_{\log}\times [0,1]$ such that $|z(0)|=q(0)=0$ and $z(h),q(h)>0$ for $h>0$. For each $h>0$, let $W_{h}$ be the fiber $W_{z(h)}$ with the K\"ahler metric given by $\omega_{q(h)}^{\epsilon}$, rescaled so that its diameter is independent of $h$. Then we have Gromov--Hausdorff convergence
		\begin{equation*}
			\lim_{h\rightarrow 0} (W_{h},W_{h}\setminus W_{h}^{\sm})=(\Delta_{\cS},\Delta_{\cS}^{\geq 2}),
		\end{equation*}
		where $\Delta_{\cS}$ is the subcomplex of the dual complex of $f^{-1}(0)$ spanned by vertices indexed by $\cS$, equipped with some multiple of the euclidean metric; and $\Delta_{\cS}^{\geq 2}$ is the union of its faces of codimension at least $2$.
	\end{enumthm} 
\end{theo}
\begin{proof}
	Parts \ref{item:thm_Kahler}, \ref{item:Lagrangians} and \ref{item:GH}  are proved in Propositions \ref{prop:omega}\ref{item:Kahler}, \ref{prop:Lagranian-at-positive-radius} and \ref{prop:GH}, respectively.
\end{proof}

\begin{proof}[Proof of Theorem \ref{theo:CY}]
	Let $f$ be a model constructed in Proposition \ref{prop:model}; and let $\Delta_{\cS}$ be the essential skeleton. Then Proposition \ref{prop:model}\ref{item:sk_P1} gives condition \eqref{eq:Cst}. The result follows from Theorem \ref{theo:general}. 
\end{proof}

\section{The Calabi--Yau case}\label{sec:CY}

From now on, we assume that $f^{\circ}$ is a maximal Calabi--Yau degeneration admitting a semi-stable snc model. Recall from Section \ref{sec:CY-intro} that the latter property can always be achieved after a finite base change, which preserves  maximality. We let $f\colon X\to \C$ be an snc model of $f^{\circ}$ as in Proposition \ref{prop:model}. We use the notation introduced in Section \ref{sec:CY-intro}. 

By Proposition \ref{prop:model}\ref{item:sk_nu-i} we have a holomorphic $(n+1)$-form $\Theta$ on $X$ whose divisor is $\sum_{i}(\nu_i-1)D_{i}$ with $\min_{i}\nu_{i}=0$. Thus the subcomplex $\Delta_{\cS}\subseteq \Delta_{D}$ for $\cS=\{i:\nu_{i}=0\}$ is the essential skeleton, see formula \eqref{eq:essential}.   As $f^{\circ}$ is maximal, \cite[Theorem 4.1.10]{NX_skeleton} gives $\dim \Delta_{\cS}=n$, so $\Delta_{\cS}$ has a maximal face.

We choose a section $\Omega$ of the vertical canonical bundle given by $\Omega\de \frac{\Theta}{d\log f}$, see formula \eqref{eq:Omega}. Its restriction to any smooth fiber is a holomorphic $n$-form, which we denote by $\Omega$, too. We denote by $\vol^{\Omega}$ the induced fiberwise volume form, defined in formula \eqref{eq:complex_MA}. We put
\begin{equation}\label{eq:Omega_new_def}
	\Omega_{\new}=t^{n}\cdot \Omega
	\quad \mbox{and}\quad \vol^{\Omega}_{\new} = t^{n}\cdot \vol^{\Omega}.
\end{equation}

\subsection{Almost all volume is fibered by Lagrangian tori}

In Proposition \ref{prop:Lagranian-at-positive-radius} we have constructed a Lagrangian torus fibration on the subset $(X_{z}^q)_{\cS}^{\sm}$ of the fiber $X_{z}$. Proposition \ref{prop:volume} below shows that as $z\rightarrow 0$, the volume of $(X_{z}^q)_{\cS}^{\sm}$ -- and even of its subset $(X_{z}^q)_{\cS}^{\gen}$ -- approaches the whole volume of the fiber.  The proof follows a well-known local computation, see e.g., \cite[\textsection 3]{BJ_measures} or \cite[\textsection 3.1]{Li_survey}, which we perform in the A'Campo space. The key point is that the suitably rescaled fiberwise volume forms \eqref{eq:Omega_new_def} extend to finite, nonzero ones at radius zero. 

\begin{lema}\label{lem:volume-extends}Let $\Delta_{I}$ be a maximal essential face. The following hold.
	\begin{enumerate}
		\item\label{item:vol_extends} The forms $\Omega_{\new} $ and $\vol_{\new}^{\Omega}$ extend to smooth sections of $\bigwedge^{n}_{\C}T^{*}_{\C,\textnormal{vert}}A$ and $\bigwedge^{2n} T^{*}_{\textnormal{vert}}A$, respectively.
		\item\label{item:vol_support} The restrictions $\Omega_{\new}|_{\d A}$ and $\vol_{\new}|_{\d A}$ are zero outside of $A_{\cS}^{\gen}$.
		\item\label{item:vol_formulas} Let $U_{X}$ be an adapted chart with associated index set $I$, let $U=\pi^{-1}(U_X)$, and let $w_j$, $\theta_j$ be the corresponding functions introduced in \eqref{eq:basic functions_r}. Then for some nonvanishing holomorphic function $c\in \cO_{X}^{*}(U_X)$, and every $i\in I$, we have a fiberwise equality on $U$
		\begin{equation}\label{eq:Omega_new}
				\Omega_{\new}= \pm c \bigwedge_{j\in I\setminus\{i\}} (dw_j-t\, \imath\, d\theta_j).
		\end{equation}
		\item\label{item:c_0} The number $c_{0}\de c(X_{I})$ does not depend on the maximal essential face $\Delta_{I}$, up to a sign.
		\item\label{item:vol_formulas_radius-zero} On $U\cap A_{I}^{\circ}$ we have fiberwise equalities
			\begin{equation}\label{eq:volume_on_maximal_face}
				\Omega_{\new}=\pm\, c_0 \cdot \bigwedge_{j\in I\setminus \{i\}} dw_{j}
				\quad\mbox{and}\quad
				\vol^{\Omega}_{\new}=(-1)^{n}\cdot  |c_0|^2 \cdot \bigwedge_{j\in I\setminus \{i\}} (dw_{j}\wedge d\theta_{j}).
			\end{equation}	
	\end{enumerate}
\end{lema}
\begin{proof}
	First, we prove the assertions about $\Omega_{\new}$. Choose an adapted chart $U_{X}$ with associated index set $J$, say $J=\{1,\dots, k\}$, fix an index $i\in J$, say $i=1$; and let $U=\pi^{-1}(U_X)$. Recall that $\Omega=\frac{\Theta}{d\log f}$, where $\Theta$ is a holomorphic $(n+1)$-form with zero of order $\nu_j-1$ along each $D_j$. There is a nonvanishing holomorphic function $c\in \cO_{X}^{*}(U_X)$ such that 
	\begin{equation}\label{eq:Theta_locally}
	\Theta|_{U_X}=c\cdot \prod_{j=1}^{k}z_{j}^{\nu_j-1}\, dz_1\wedge\dots \wedge dz_k\wedge \zeta=c\cdot \prod_{j=1}^{k}z_{j}^{\nu_{j}}\, d\log z_1\wedge d\log z_2\wedge \dots \wedge d\log z_{k} \wedge \zeta,
	\end{equation}
	where $\zeta=dz_{i_{k+1}}\wedge\dots\wedge dz_{i_{n+1}}$ comprises the remaining coordinates of $U_X$. Since $\frac{\d \log f}{\d \log z_{1}}=m_{1}$, we have 
	\begin{equation*}
		\Omega=\frac{\Theta}{d\log f}=
		\frac{c}{m_1}\prod_{j=1}^{k}z_{j}^{\nu_j}\, d\log z_2\wedge \dots\wedge d\log z_k\wedge \zeta.
	\end{equation*}
	Formulas \eqref{eq:basic functions_r} give $\log z_j=s_j+\imath\, \theta_j=-\frac{1}{tm_j} w_j+\imath\, \theta_j$, so we have a fiberwise equality 
	\begin{equation*}
		d\log z_j=-\frac{1}{tm_j}\, dw_j+\imath\, d\theta_j=-\frac{1}{tm_j}\left (dw_j-tm_{j}\cdot \imath\, d\theta_{j}\right).
	\end{equation*} Substituting it to the above expression for $\Omega$, we get a fiberwise equality
	\begin{equation}\label{eq:tOmega_fiberwise}
		\Omega=t^{1-k}\cdot \frac{c\cdot (-1)^{k-1}}{m_1\cdot\ldots\cdot m_{k}}\prod_{j=1}^{k} z_j^{\nu_j}\cdot  (dw_2-tm_2\cdot \imath\, d\theta_2)\wedge \dots\wedge (dw_k-tm_k\cdot \imath\, d\theta_k) \wedge \zeta.
	\end{equation}
	By Proposition \ref{prop:model}\ref{item:sk_nu-i}, we have $\nu_j\geq 0$ for all $j$, with equality if and only if $j\in \cS$. Hence the function $\prod_{j=1}^{k} z_j^{\nu_j}$ is smooth on $U_X$, and equals $0$ on $U_X\cap X_{J}$ unless $J\subseteq \cS$. The function $t$ is smooth on $U$ by Proposition \ref{prop:AX_hat_smooth}\ref{item:AX-gsmooth}, hence so is $t^{n+1-k}$ for $k\leq n+1$. The function $c$, the form $\zeta$, and each $z_{j}$ are smooth on $U_X$ by definition, hence they are smooth on $U$ by Proposition \ref{prop:AX_hat_smooth}\ref{item:AX-pismooth}. The functions $w_{j}$ are smooth, too: indeed by Lemma \ref{lem:comparison}\ref{item:w-comparison} we have $w_{j}=\hat{w}_j-at$ for $a\in \cC^{\infty}(U_X)$, and the functions $\hat{w}_j,a$ and $t$ are smooth on $U$ by Proposition \ref{prop:AX_hat_smooth}. Eventually, the functions $\theta_{j}$ are smooth coordinates  on $U$, see formula \eqref{eq:AC-chart_hat}. Hence formula \eqref{eq:tOmega_fiberwise} shows that $t^{n}\Omega=\Omega_{\new}$ extends to a smooth section of $\bigwedge^{n}_{\C}T^{*}_{\C,\textnormal{vert}}A$, and $\Omega_{\new}|_{U\cap A_{J}^{\circ}}=0$ unless $k=n+1$ and $J\subseteq \cS$, i.e., unless $\Delta_{J}$ is a maximal, essential face. This proves the first parts of \ref{item:vol_extends},\ref{item:vol_support}. Moreover, if  $\Delta_{J}$ is a maximal, essential face, i.e., $J=\{1,\dots, n+1\}\subseteq \cS$ then $(m_{i},\nu_i)=(1,0)$ for all $i\in J$ by Proposition \ref{prop:model}\ref{item:sk_reduced},\ref{item:sk_nu-i}. Substituting this to formula \eqref{eq:tOmega_fiberwise} we get part \ref{item:vol_formulas}. Furthermore, substituting $t=0$ gives the following fiberwise equality on $U\cap \d A$
	\begin{equation}\label{eq:Omega-new-local}
	\Omega_{\new}= c\cdot (-1)^{n} \cdot dw_2\wedge \dots \wedge dw_{n+1}. 
	\end{equation}
	Note that on $U\cap \d A \setminus \pi^{-1}(D_j)$ we have $dw_j=0$, so on $U\cap \d A\setminus A_{J}^{\circ}$ both sides of the equality \eqref{eq:Omega-new-local} are zero. On $A_{J}^{\circ}$, the function $c$ in \eqref{eq:Omega-new-local} is constant, equal to the value of $c\in \cO_{X}^{*}(U_X)$ at the point $X_J$. This proves the first equality in \eqref{eq:volume_on_maximal_face}.
	\smallskip
	
	Now, we prove the assertions about $\vol^{\Omega}_{\new}$. Using formula \eqref{eq:tOmega_fiberwise} we get a fiberwise equality 
	\begin{equation*}\begin{split}
		 & \Omega\wedge \bar{\Omega}= t^{2-2k} \cdot \frac{|c|^2 }{m_1^2\cdot\ldots\cdot m_{k}^2}\prod_{j=1}^{k} |z_j|^{2\nu_j} \cdot  \bigwedge_{j=2}^{k} (dw_j-tm_j\, \imath\,  d\theta_j) \wedge \zeta \wedge \bigwedge_{j=2}^{k} (dw_j+tm_j\, \imath\, d\theta_j) \wedge \bar{\zeta}
		 = \\
		  =\ & t^{1-k} \cdot \frac{|c|^2 \cdot (-1)^{(k-1)(n+1-k)}\cdot (2 \imath)^{k-1}(-1)^{\frac{1}{2}(k-1)(k-2)}}{m_1^2\cdot\ldots\cdot m_{k}^2} \cdot \prod_{j=1}^{k} |z_j|^{2\nu _j}  
		\cdot \bigwedge_{j=2}^{k} (dw_j\wedge m_j\, d\theta_j) \wedge \zeta\wedge \bar{\zeta},
	\end{split}\end{equation*}
	where the sign $(-1)^{(k-1)(n+1-k)}$ comes from moving the $(n+1-k)$-form $\zeta$ to the right; and the constant $(2\imath)^{k-1}(-1)^{\frac{1}{2}(k-1)(k-2)}$ comes from multiplying out the $1$-forms. Like before, we conclude that $t^{n}\Omega\wedge \bar{\Omega}$ extends to a smooth section of $\bigwedge^{2n}T^{*}_{\textnormal{vert}} A$, and vanishes on $U\cap \d A$ unless $\Delta_{J}$ is a maximal face of $\Delta_{\cS}$. In the latter case we have $(m_j,\nu_j)=(1,0)$ for all $j\in J$, so the above expression of $\Omega\wedge \bar{\Omega}$ yields
	\begin{equation*}
		t^{n}\Omega \wedge \bar{\Omega} = |c|^2 \cdot  (2\imath)^{n}(-1)^{\frac{1}{2}n(n-1)} \cdot (dw_2\wedge d\theta_2)\wedge\dots \wedge (dw_{n+1}\wedge d\theta_{n+1}).
	\end{equation*}
	which ends the proof of \ref{item:vol_extends}, \ref{item:vol_support} and \ref{item:vol_formulas_radius-zero}.  The remaining part \ref{item:c_0} is shown in the proof of \cite[Theorem 7.1]{BJ_measures}, let us recall the argument here for completeness. 
	
	Fix two maximal faces $\Delta_{I}$ and $\Delta_{J}$ of $\Delta_{\cS}$. We claim that $c(X_{I})=\pm c(X_J)$. Since the essential skeleton is a pseudomanifold \cite[Theorem 4.2.4(3)]{NX_skeleton}, we can assume that $\Delta_{I}$ and $\Delta_{J}$ meet along a common submaximal face $\Delta_{I'}$, see part (3) of the definition of pseudomanifold given in \textsection 4.1.2 loc.\ cit. 
	
	Say that $I'=\{2,\dots, n+1\}$, $I=I'\sqcup\{1\}$, $J=I'\sqcup\{n+2\}$. By Proposition \ref{prop:model}\ref{item:sk_P1}, the stratum $X_{I'}$ is isomorphic to $\mathbb{P}^1$, and $X_{I}$, $X_{J}$ are two points on $X_{I'}$. By Proposition \ref{prop:model}\ref{item:sk_nu-i} we have $\nu_{i}=0$ for $i\in \{1,\dots,n+2\}$, so in an adapted chart at a point $p\in X_{I'}$, the local formula \eqref{eq:Theta_locally} defining $c$ reads as 
	\begin{equation*}
		\Theta=c\cdot d\log z_2\wedge \dots \wedge d\log z_{n+1} \wedge \zeta,
	\end{equation*}
	where up to a sign we have $\zeta=d \log z_1$ if $\{p\}=X_{I}$, $\zeta=d\log z_{n+2}$ if $\{p\}=X_{J}$, and otherwise $\zeta=d z$ for a holomorphic coordinate $z$ on $X_{I'}$. It follows that the logarithmic form $c\cdot \zeta$ is the residue of $\Theta$ along $X_{I'}$; with poles at $X_{I}$ and $X_{J}$ whose residues are equal, respectively, $\pm c(X_{I})$ and $\pm c(X_{J})$. By the Poincar\'e residue theorem, those residues add up to $0$, so $c(X_{I})=\pm c(X_J)$, as claimed.
\end{proof}

We can now prove the title result of this section. Recall that in Proposition \ref{prop:Lagranian-at-positive-radius}, we have constructed a Lagrangian torus fibration of a region $(X_{z}^{q})_{\cS}^{\sm}\subseteq X_{z}$, which maps the generic region $(X_{z}^{q})_{\cS}^{\gen}$ to the disjoint union of maximal faces of $\Delta_{\cS}$.  The following result shows that this region is large.

\begin{prop}\label{prop:volume}
	Fix $q\in (0,1]$, and let $(X_{z}^{q})_{\cS}^{\gen}$ be as in Section \ref{sec:moving_Lagrangians}. Then we have 
	\begin{equation*}
		\lim_{|z|\rightarrow 0} \frac{\int_{(X_{z}^{q})_{\cS}^{\gen}} \vol^{\Omega}}{\int_{X_{z}} \vol^{\Omega}}=1.
	\end{equation*}
\end{prop}
\begin{proof}
	Put $\theta=\frac{z}{|z|}$, $F_{\theta}=f_{A}^{-1}(0,\theta)$. 
	The sets $(X_{z}^{q})_{\cS}^{\gen}$ and $X_{z}$ are images of $A^{\gen}_{\cS,\theta}$ and $F_{\theta}$ by the flow \eqref{eq:flow}. By Lemma \ref{lem:volume-extends} the form $t^{n}\vol^{\Omega}$ extends to a fiberwise form $\vol_{\new}^{\Omega}$ on $A$, whose restriction to $F_{\theta}$ is zero outside of $A_{\cS,\theta}^{\gen}$. Therefore, we have
	\begin{equation*}
		\lim_{|z|\rightarrow 0} \int_{(X_{z}^{q})_{\cS}^{\gen}} t^{n} \vol^{\Omega}=
		\int_{A_{\cS,\theta}^{\gen}} \vol_{\new}^{\Omega} =
		\int_{F_{\theta}} \vol_{\new}^{\Omega}=\lim_{|z|\rightarrow 0} \int_{X_{z}} t^{n} \vol^{\Omega},
	\end{equation*}
	which ends the proof. 
\end{proof}

\subsection{Asymptotic Ricci flatness}\label{sec:weak_MA}
 
Recall that on each smooth fiber $X_{z}$ of $f^{\circ}$, the cohomology class of $\omega_{X}$ is represented by a unique Ricci-flat form $\omega_{\CY}$, and this condition can be written equivalently as 
\begin{equation*}
	\tfrac{1}{n!}\omega_{\CY}^{n}=c\cdot \vol^{\Omega}_{\new},
\end{equation*}
for some positive function $c$ on the base $\D_{\delta}^{*}$, see formula \eqref{eq:complex_MA}. The volume form on the right-hand side is rescaled so that the total volume on both sides equals $\int_{X_{z}}\frac{1}{n!}\omega_{X}^{n}$. To see that the appropriate rescaling is indeed of the same order as $\vol_{\new}^{\Omega}$, note that it should integrate to a nonzero number on all fibers in the A'Campo space, including the radius-zero ones. This is true for $\vol_{\new}^{\Omega}$ by Lemma \ref{lem:volume-extends}\ref{item:vol_formulas_radius-zero}.

Alternatively, one can use \cite[Definition 1]{KS_Conjecture}: indeed, in the notation of loc.\ cit.\ the rescaled volume form should be of order $((\log|q|)^{m} |q|^{2k})^{-1}\, \vol^{\Omega}$, where $q=f$ (so $-\log|q|=t^{-1}$), $m=n$ since $f^{\circ}$ is maximal, and $k=\min_{i}\nu_{i}$, which equals $0$ by Proposition \ref{prop:model}\ref{item:sk_nu-i}. 
\smallskip

Proposition \ref{prop:MA_weak} below asserts that, roughly speaking, our forms $\omega^{\gamma}$ give some approximation of $\omega_{\CY}$ in the generic region $(X_{z}^{q})^{\gen}_{\cS}$. Of course, due to some choices made in the definition of $\omega^{\gamma}$, we cannot expect an exact equality $\omega^{\gamma}=\omega_{\CY}$ at positive radius (nor off the generic region, see Remark \ref{rem:MA_fails}).

 \begin{prop}\label{prop:MA_weak}
	Let $\gamma$ be an admissible path, see Definition \ref{def:admissible_path}. There is a continuous function $c_{+}\colon A^{\gamma}_{+}\cup (X_{0}^{\gamma})_{\cS}^{\gen}\to \R_{>0}$ such that on each fiber $X_{h}^{\gamma}$ and on $(X_{0}^{\gamma})^{\gen}_{\cS}$ we have an equality
	\begin{equation}\label{eq:MA_weak}
		\tfrac{1}{n!} (\omega_{h}^{\gamma})^{n}=c_{+}\cdot \vol^{\Omega}_{\new}
	\end{equation}
	and the restriction of $c_{+}$ to $(X_{0}^{\gamma})_{\cS}^{\gen}$ is a positive constant. This constant equals $\epsilon^n\cdot (n+1)\cdot |c_0|^{-2}$, where $\epsilon>0$ is a number fixed in Notation \ref{not:flow}, and  $c_0\in \C^{*}$ is the number in Lemma \ref{lem:volume-extends}\ref{item:c_0}.
\end{prop}	
\begin{proof}
	Recall that by definition, the form $\omega_{q}^{\epsilon}$ for any $q\in [0,1]$ is smooth on $A$. By Proposition~\ref{prop:omega} it is fiberwise nondegenerate at the positive radius, and at the generic region at radius zero, i.e., on $(A\setminus \d A)\cup A_{\cS}^{\gen}$. In turn, by Lemma \ref{lem:volume-extends}\ref{item:vol_extends} the form $\vol^{\Omega}_{\new}$ extends to a fiberwise smooth form on $A$, which is a fiberwise volume form at positive radius, and, by Lemma \ref{lem:volume-extends}\ref{item:vol_formulas_radius-zero}, at the generic region at radius zero. Comparing the two fiberwise volume forms, we get a smooth family of functions $c_{+}^{q}\colon (A\setminus \d A)\cup A_{\cS}^{\gen} \to \R_{>0}$ satisfying fiberwise equality $\frac{1}{n!}(\omega_{q}^{\epsilon})^{n}=c_{+}^{q}\cdot \vol^{\Omega}_{\new}$. Restricting it to the closed subset $A^{\gamma}_{+}\cup (X_{0}^{\gamma})_{\cS}^{\gen}$ we get a continuous function $c_{+}$ satisfying the fiberwise equality \eqref{eq:MA_weak}. 
	
	It remains to prove that on $(X_{0}^{\gamma})_{\cS}^{\gen}$ we have
	\begin{equation*}
		\tfrac{1}{n!} (\omega_{0}^{\epsilon})^{n}=(\epsilon^n\cdot (n+1)\cdot |c_0|^{-2})\cdot \vol^{\Omega}_{\new}.
	\end{equation*} 
	Recall that $(X_{0}^{\gamma})_{\cS}^{\gen}$ is the union  of pieces $\Int_{\d A} A_{I}^{\circ}\cap X_{0}^{\gamma}$ for all maximal faces $\Delta_{I}$ of the essential skeleton $\Delta_{\cS}$. Thus it is enough to show that the above equality holds fiberwise on each $A_{I}^{\circ}$.
	
	Fix a maximal, essential face $\Delta_{I}$. By Lemma \ref{lem:rihat}\ref{item:r_maximal}, on $A_{I}^{\circ}$ we have $\hat{w}_i=w_i$, $\hat{\alpha}_i=d\theta_i$ for $i\in I$ and $\hat{w}_i=0$, $\hat{\alpha}_i=0$ for $i\neq I$. Thus $\omega^{\flat}=-\sum_{i\in I}dw_{i}\wedge d\theta_{i}$ by formula \eqref{eq:omega_flat}. Since $\pi(A_{I}^{\circ})=X_{I}^{\circ}$ is a point, definition \eqref{eq:omega-s} of $\omega_{q}^{\epsilon}$ shows that on $A_{I}^{\circ}$ we have $\omega_{0}^{\epsilon}=\epsilon \cdot \omega^{\flat}$, and therefore
	\begin{equation*}
		\tfrac{1}{n!}(\omega_{0}^{\epsilon})^{n}=(-\epsilon)^{n} \cdot \sum_{i\in I} \bigwedge_{j\in I\setminus \{i\}} (dw_{j}\wedge d\theta_{j})
		\overset{\mbox{\tiny{\eqref{eq:volume_on_maximal_face}}}}{=}
		(\epsilon^{n} \cdot (n+1)\cdot |c_0|^{-2})\cdot \vol^{\Omega}_{\new},
	\end{equation*}
	as needed.
\end{proof}

\begin{remark}\label{rem:MA_fails}
	The function $c_{+}$ does not extend to the non-generic region $X_{0}^{\gamma}\setminus (X_{0}^{\gamma})_{\cS}^{\gen}$. Indeed, on $\d A\setminus A_{\cS}^{\gen}$ the form $\vol^{\Omega}_{\new}$ is zero by Lemma \ref{lem:volume-extends}, while $\omega_{0}^{\gamma}$ has some contributions from pullbacks of forms from $X$, namely: the initial K\"ahler form $\omega_{X}$, the forms $\beta$ from Lemma \ref{lem:comparison}\ref{item:alpha-comparison}, and the forms $d\hat{\alpha}_i$, cf.\ Lemma \ref{lem:53}\ref{item:omega_flat_local}. In particular, on the subset $R_{i}^{\circ}\cap X_{0}^{\gamma}$ from Lemma \ref{lem:rihat}\ref{item:r_covering}, which can be identified with an open subset of $D_{i}\setminus \bigcup_{j\neq i} D_j$, the form $\frac{1}{n!}(\omega^{\gamma}_{0})^n$ restricts to the standard volume form $\frac{1}{n!}\omega_{X}^{n}$.
	
	We conclude that the volume of the non-generic region of $X_h$ with respect to $\omega_{q(h)}$ does not converge to zero as $h\rightarrow 0$. As a consequence, the constant $\lim_{h\rightarrow 0}c_{h}$ in Theorem \ref{theo:CY_af}\ref{item:assflat} is smaller than the one in formula \eqref{eq:ricciflatintro}. This explains the factor $\epsilon^{n}$ in Proposition \ref{prop:MA_weak}.
\end{remark}
 
The following easy consequence of Proposition \ref{prop:MA_weak} will be useful in the next section.

\begin{lema}\label{lem:c+}
	Fix $h>0$. Let $L_{h}\subseteq X_{h}^{\gamma}$ be a fiber of a Lagrangian torus fibration \eqref{eq:Lagranian-at-positive-radius}, and let $\vol^{g}_{\new}$ be a Riemannian volume form induced by the metric $g_{h,\new}^{\gamma}$. Then we have $\vol^{g}_{\new}|_{L_h}=\sqrt{c_{+}}\cdot |\Omega_{\new}|_{L_{h}}|$.
\end{lema}
\begin{proof}
	Fix a point $x\in L_{h}$ and let $e_{1},\dots, e_{n}$ be an orthonormal basis of  $T_{x}L_{h}$, with respect to the metric $g_{h,\new}^{\gamma}$. Since $L_{h}$ is Lagrangian with respect to a $J$-compatible form $h\cdot \omega_{h}^{\gamma}=g_{h,\new}(J\sdot,\sdot)$, the collection $e_1,\dots,e_{n}$, $Je_1,\dots, Je_n$ is an orthonormal basis for $T_{x}X_{h}^{\gamma}$ with respect to the metric $g_{h,\new}^{\gamma}$, see formulas (1.2) and (1.2)' in \cite{HL}. Substituting this collection to the formula \eqref{eq:MA_weak} multiplied by $t^{n}$ and applying formula \eqref{eq:complex_MA} for $\vol^{\Omega}$, we get $1=c_{+}\cdot |\Omega_{\new}(e_1,\dots, e_n)|^2$, as needed.
\end{proof}
 
\subsection{Generic Lagrangian tori are asymptotically special}

We recall the notion of \emph{special Lagrangian submanifolds}, introduced in \cite{HL}. Let $Y$ be a Calabi--Yau manifold, with a holomorphic volume form $\Omega_{Y}$ and a Ricci-flat K\"ahler form $\omega_{\CY}$, see equation~\eqref{eq:complex_MA}. A Lagrangian submanifold $L$ of $Y$ is \emph{special of phase $\varpi\in \mathbb{S}^{1}$} if  $\Im (\varpi\cdot \Omega_{Y})|_{L}=0$. Condition \eqref{eq:complex_MA} guarantees that special Lagrangian submanifolds are calibrated with respect to $\Re (\varpi\cdot \Omega_{Y})$, see \cite[\textsection 8.1]{Joyce_yellow}, in particular, they minimize volume in their homology class, see Proposition 7.1 loc.\ cit. The specialty condition makes the geometry  rigid, see \cite[\textsection 6.1]{BigBook}; and indeed the Lagrangian tori in SYZ conjecture are expected to be special, see paragraph \ref{item:intro-special} in the introduction. In this section, we prove that for our tori specialty holds asymptotically in the generic region, as claimed in Theorem~\ref{theo:CY_af}\ref{item:CY_asymptotic specialty}.

To give a precise statement in  Proposition \ref{prop:special} below, we introduce some notation. Fix an admissible path $\gamma$, see Definition \ref{def:admissible_path}, and a point  $b\in E_{\cS}$. Consider a family of Lagrangian torus fibers $L_{h}\subseteq X_{h}^{\gamma}$ over $b$. It follows from Lemma \ref{lem:c+} that for $h>0$, the form  $\Omega$ does not vanish on $L_{h}$, so at each point it is a nonzero complex multiple of a volume form on $L_{h}$. This allows us to pose the following definition.

\begin{definition}\label{def:phase}
	Let $\vol_{h}$ be a volume form on $L_{h}$. Let $c_{h}\colon L_{h}\to \C^{*}$ be the smooth function satisfying 
	\begin{equation*}
		c_{h}\cdot \vol_{h} = \Omega|_{L_h}.
	\end{equation*}
	We define \emph{the phase map} as $\varpi_{h}\de \frac{c_{h}}{|c_{h}|}\colon L_{h}\to \mathbb{S}^{1}$.
\end{definition}

Clearly, this definition does not depend on the choice of the volume form $\vol_{h}$. We let $L^{\gamma}\subseteq A^{\gamma}$ be the union of fibers $L_{h}$ for all $h\geq 0$; and let $L_{+}^{\gamma}=L^{\gamma}\setminus L_{0}$, so the fiberwise phase map yields a smooth map $\varpi\colon L_{+}^{\gamma}\to \mathbb{S}^1$. We can now state the main result of this section, which implies Theorem \ref{theo:CY_af}\ref{item:CY_asymptotic specialty}.

\begin{prop}\label{prop:special}
	Let $\gamma$ be a  tame admissible path, see Definition \ref{def:special_path}. Let $b$ be an interior point of a maximal, essential face $\Delta_{I}$, and let $L^{\gamma}\subseteq A^{\gamma}$ be the family of fibers of the Lagrangian fibration \eqref{eq:Lagranian-at-positive-radius} over $b$. Then the phase map $\varpi\colon L_{+}^{\gamma}\to\mathbb{S}^1$ introduced in Definition \ref{def:phase} extends to a continuous map $L^{\gamma}\to \mathbb{S}^1$, whose value on $L_0$ is a constant 
	\begin{equation}\label{eq:asymptotic_specialty}
		\varpi|_{L_0}=\pm\, \imath^{n}\cdot \frac{c_0}{|c_0|},
	\end{equation}
	where $c_0\in \C^{*}$ is the number in formula \eqref{eq:volume_on_maximal_face}, and the sign depends only on  the face $\Delta_{I}$.
\end{prop}

In the generic region $(X_{h}^{\gamma})_{\cS}^{\gen}$ our form satisfies the \enquote{asymptotic Ricci flatness} condition \eqref{eq:MA_weak}, so the \enquote{asymptotically special} Lagrangian tori $L_h$ should \enquote{asymptotically} minimize volume there. And indeed, we will see that the arguments of \cite[\textsection III.1]{HL} yield the following corollary.
\begin{cor}\label{cor:volume_minimizing}
	Let $\gamma$ be a  tame admissible path, let $b$ be an interior point of a maximal face of $\Delta_{\cS}$, and for $h>0$ let $L_{h}\subseteq (X_{h}^{\gamma})_{\cS}^{\gen}$ be the fiber over $b$ of the Lagrangian torus fibration \eqref{eq:Lagranian-at-positive-radius}. Let $\vol^{g}_{\new}$ be a Riemannian volume form on $L_{h}$ associated to the metric $g_{\new}^{\gamma}$. Then we have a convergence
	\begin{equation}\label{eq:volume_minimizing}
		\lim_{h\rightarrow 0}\ \frac{\int_{L_h}\vol^{g}_{\new}}{\min_{\Gamma\in [L_h]} \int_{\Gamma}\vol^{g}_{\new}}=1,
	\end{equation}
	where $[L_{h}]$ is the class of $L_{h}$ in $H_{n}((X_{h}^{\gamma})_{\cS}^{\gen};\R)$.
\end{cor}
The remaining part of this section is devoted to the proof of Proposition \ref{prop:special} and Corollary \ref{cor:volume_minimizing}.
\smallskip

We work in $(X_{h}^{\gamma})_{I}^{\gen}$ for a fixed maximal essential face $\Delta_{I}$, say $I=\{1,\dots, n+1\}$. By Lemma \ref{lem:rihat}\ref{item:r_maximal}, there is an adapted chart $U_{X}$ around the point $X_{I}^{\circ}$ such that for each $i\in I$, the distance function $\hat{r}_{i}$ equals the radial coordinate $r_{i}$ of $U_{X}$. Thus on $U\de \pi^{-1}(U_{X})$ the functions $\hat{w}_i$, $\hat{v}_i$ and the form $\hat{\alpha}_i$ introduced in formula \eqref{eq:basic_functions_hat} for each $i\in I$ are equal to the functions $w_i$, $v_i$ and the form $d\theta_{i}$ introduced for $U_X$ in formula \eqref{eq:basic_functions_v}. Moreover, for $i\not\in I$ the forms $d\hat{v}_i$, $d\hat{w}_i$ and $\hat{\alpha}_{i}$ are identically zero. As a consequence, on $U$  formulas  \eqref{eq:omega_sharp} and \eqref{eq:omega_flat} read as 
\begin{equation}\label{eq:omega-easy}
\omega^{\sharp}=\sum_{i\in I} dv_{i}\wedge d\theta_{i}\quad\mbox{and}\quad\omega^{\flat}=\sum_{i\in I} dw_{i}\wedge d\theta_{i}.
\end{equation}
Since the radius-zero  torus $L_{0}$ is a compact subset of $U\cap \Int_{\d A}A_{I}^{\circ}$, we can and do work in a neighborhood $V$ of $L_0$ in $U$ where each function $w_{i}$ is bounded away from $0$. As a consequence of this boundedness, we have the following convenient coordinate system.

\begin{lema}\label{lem:new-coordinates}
	The smooth map $(t,\theta;w_1,\dots,w_n,\theta_{1},\dots,\theta_{n})$ endows $V$ with a new smooth structure, fiberwise compatible with the one inherited from $A$. This smooth structure has the following properties.
	\begin{enumerate}
		\item \label{item:V-vw} For every $i\in I$, the functions $v_{i}$, $w_{i}$ and $\theta_i$ are smooth.
		\item \label{item:V-pi} The map $\pi\colon V\to X$ is smooth. 
		\item \label{item:V-omega}  For every $q\in [0,1]$, the form $\omega_{q}^{\epsilon}$ is smooth, closed and fiberwise symplectic.
		\item \label{item:V-dt} The coordinate vector field $\frac{\d}{\d t}$ satisfies $dw_{i}(\frac{\d}{\d t})=0$ and $d\theta_{i}(\frac{\d}{\d t})=0$ for all $i\in I$.
	\end{enumerate}
\end{lema}
\begin{proof}
	Let $V=\bigcup_{i\in I} V_{i}$ be the covering of $V$ by charts \eqref{eq:AC-chart_hat}. Recall that on $V$ we have $\theta=\sum_{i\in I}\theta_{i}$ and $\hat{v}_{i}=v_{i}$ for all $i\in I$, so the coordinate system  \eqref{eq:AC-chart_hat} on $V_i$ is equivalent to $(g,\theta,(v_{j})_{j\in I_i},(\theta_{j})_{j\in I_i})$, where $I_i=I\setminus \{i\}$. Since $t=\exp(1-g^{-1})$ is a reparametrization of $g$, it is clear that the map $(t,\theta,(v_{j})_{j\in I_i},(\theta_{j})_{j\in I_i})$ endows $V$ with a smooth structure, $t$-\emph{fiberwise} compatible with the one from $A$.
	
	We claim that it is equivalent to the one in the statement of the lemma. By definition \eqref{eq:basic_functions_v} of $v_j$, we have $v_{j}=tw_{j}^{-1}-\eta(w_{j})$, so $dv_{j}=w_{j}^{-1}\, dt-(tw_{j}^{-2}+\eta'(w_j))\, dw_{j}$. The coefficient $tw_{j}^{-2}+\eta'(w_j)$ is smooth and positive since $w_{j}>0$, so we can smoothly change each coordinate $v_j$ for $w_{j}$. Eventually, since $\sum_{j\in I}w_j=1$ and $\sum_{j\in I}\theta_{j}=\theta$, we can change $w_{n+1}$, $\theta_{n+1}$ for $w_{i}$, $\theta_i$, as needed. 
	
	Thus the map $(t,\theta;w_1,\dots,w_n,\theta_{1},\dots,\theta_{n})$ gives a smooth structure on the whole $V$, equivalent to $(t,\theta,(v_{j})_{j\in I_i},(\theta_{j})_{j\in I_i})$ for each $i\in I$. This implies property \ref{item:V-vw}. We now prove the remaining ones.
	
	\ref{item:V-pi} It is enough to show that the coordinate functions $r_{i}\cos\theta_i$, $r_{i}\sin\theta_{i}$ are smooth for each $i\in I$.  To see this, recall that $\theta_{i}$ is smooth by \ref{item:V-vw}; and  $r_{i}=\exp(-t_{i}^{-1})=\exp(-w_{i}t^{-1})$ is smooth because $w_{i}$ is positive and smooth by \ref{item:V-vw}.
	
	\ref{item:V-omega} Part \ref{item:V-vw} and formula \eqref{eq:omega-easy} show that the forms $\omega^{\flat}$ and $\omega^{\sharp}$ are smooth and closed. By \ref{item:V-pi}, the form $\pi^{*}\omega_{X}$ is smooth and closed, too, hence so is $\omega_{q}^{\epsilon}$. Since the new smooth structure on $V$ is fiberwise compatible with the one in $A$, fiberwise non-degeneracy of $\omega_{q}^{\epsilon}$ follows from Proposition \ref{prop:omega}.
	
	\ref{item:V-dt} For $i\neq n+1$, this property follows from the definition of the coordinate chart; for $i=n+1$ we use relations $w_{n+1}=1-\sum_{i=1}^{n}w_i$ and $\theta_{n+1}=\theta-\sum_{i=1}^{n}\theta_{i}$.
\end{proof}

\begin{lema}\label{lem:Q}
	Let $\cQ\subseteq \cC^{\infty}(V)$ be the ideal of functions $p$ such that for every $k\in \Z$, $t^{k}p$ extends to a smooth function on $V$, with respect to the smooth structure from Lemma \ref{lem:new-coordinates}. The following hold.
	\begin{enumerate}
		\item\label{item:Q-closed} The ideal $\cQ$ is closed under taking partial derivatives.
		\item\label{item:Q-beta} For every smooth form $\beta$ on $U_{X}$ we have $\pi^{*}\beta(\frac{\d}{\d t})\in \cQ$.
		\item\label{item:Q-lift} Let $\nu_{t}^{0}$ be the symplectic lift, with respect to the form $\omega_{0}^{\epsilon}$ and the submersion $(t,\theta)$, of the unit radial vector field on $\C_{\log}$. Then the vector field $\nu_{t}^{0}-\frac{\d}{\d t}$ has coefficients in $\cQ$.
		\item\label{item:Q-flow} Let $\Psi^{0}\colon \d V\times [0,\rho) \to V$ be the flow of $\nu_{t}^{0}$, where $\d V\de V\cap \d A$. Then for every integer $k\geq 0$ there is a smooth map $\psi_{k}\colon \d V\times [0,\rho) \to V$ such that in coordinates of Lemma \ref{lem:new-coordinates} we have 
		\begin{equation*}
			\Psi_{h}^0(0,\bar{w},h)=(\bar{w},h)+h^k\cdot \psi_{k}(0,\bar{w},h),
		\end{equation*}
		where we write $\bar{w}\de (\theta,w_1,\dots,w_n,\theta_1,\dots,\theta_n)$ for the remaining coordinates of a point in $V$.
	\end{enumerate}
\end{lema}
\begin{proof}
	\ref{item:Q-closed} Let $\d$ be one of the partial derivatives. For every $p\in \cQ$ and $k\in \Z$ we have $t^{k}\cdot \d p=\d(t^{k} p)-kt^{k-1} p\cdot \d t$, which is a smooth function by definition of $\cQ$, so $\d p\in \cQ$, as needed.
	
	\ref{item:Q-beta} Since $\cQ$ is an ideal, it is sufficient to check \ref{item:Q-beta} for the coordinate forms $d(r_{i}\cos\theta_i)$ and $d(r_{i}\sin\theta_i)$. By \ref{item:Q-closed} it is enough to check that $r_{i}\in \cQ$. As in the proof of Lemma \ref{lem:new-coordinates}\ref{item:V-pi}, we write $r_{i}=\exp(-w_{i}t^{-1})$ and use the fact that $w_{i}$ is bounded from below by a positive number.
	
	\ref{item:Q-lift} Since the vector field $\frac{\d}{\d t}-\nu_{t}^{0}$ is vertical, and the form $\omega_{0}^{\epsilon}$ is fiberwise non-degenerate, it is enough to prove that for every vertical vector field $\nu$ we have  $\omega_{0}^{\epsilon}(\frac{\d}{\d t}-\nu_{t}^{0},\nu)\in \cQ$. 
	
	Since $\nu_{t}^{0}$ is orthogonal to the fibers, we have $\omega_{0}^{\epsilon}(\frac{\d}{\d t}-\nu_{t}^{0},\nu)=\omega_{0}^{\epsilon}(\frac{\d}{\d t},\nu)$. Recall from formula \eqref{eq:omega-s} that $\omega_{0}^{\epsilon}=\pi^{*}\omega_{X}+\epsilon \omega^{\flat}$. Writing $\omega_{X}=\sum_{i}\beta_{i}\wedge \beta_{i}'$ for some smooth $1$-forms $\beta_{i},\beta_{i}'$ on $U_{X}$, we see from \ref{item:Q-beta} that $\pi^{*}\omega_{X}(\frac{\d}{\d t},\nu)\in \cQ$. In turn, formula \eqref{eq:omega-easy} and Lemma \ref{lem:new-coordinates}\ref{item:V-dt} imply that $\omega^{\flat}(\frac{\d}{\d t},\nu)=0$, as needed.
	
	\ref{item:Q-flow} By \ref{item:Q-lift} we have $\nu_{t}^{0}-\frac{\d}{\d t}=t^{k}\nu$ for some smooth vertical vector field $\nu$. We treat $t^{k}\nu$ as a time-dependent vector field on $\d V$, and look at its integral curves. A point at time $h$ has coordinates of the form $\int_{0}^{h}s^{k}a(s)\, ds$ for some smooth, time-dependent function $a$, namely a coordinate of $\nu$. The result follows since the first $k-1$ derivatives with respect to $h$ of such an integral vanish at $h=0$. 
\end{proof}

Consider a product $V\times [0,1]$, where $V$ is equipped with the  smooth structure from Lemma \ref{lem:new-coordinates}, and let $q$ be the projection onto $[0,1]$. The family of forms $(\omega^{\epsilon}_{q})$ yields a smooth form on $V\times [0,1]$, which is fiberwise symplectic with respect to the submersion $(t,\theta,q)$. Let $\nu_{t}$ be the symplectic lift of the radial vector field on $\C_{\log}$, and let $\Psi\colon \d V\times [0,1]\times [0,\tilde{\delta})\to V\times [0,1]$ be the flow of $\nu_{t}$. Here $\tilde{\delta}=\exp(-\delta^{-1})>0$ is a small positive number such that the flow $\Psi_{h}$ is defined for all times $h\in [0,\tilde{\delta})$. Let $\Psi^{q}$ be the restriction of $\Psi$ to the slice $\d A\times \{q\}\times [0,\tilde{\delta})$. Note that for $q>0$, this flow agrees with $\Psi^{q}$ introduced in Notation \ref{not:admissible}, once we identify the slice  $\d A\times \{q\}\times [0,\tilde{\delta})$ with $\d A\times [0,\tilde{\delta})$. In turn, the flow $\Psi^{0}$ is described in Lemma \ref{lem:Q}\ref{item:Q-flow}.
\smallskip

Now, we fix a  tame admissible path $\gamma$, see Definition \ref{def:special_path}, and let $[0,\rho)\ni h\mapsto (h,\theta(h),q(h))\in \C_{\log}\times [0,1]$ be its normal parametrization. We put $\Gamma\de \d V\times \{(q(h),h): h\in [0,\rho)\}$. Since $\gamma$ is  tame, $\Gamma$ is a $\cC^{1}$ submanifold of $\d V\times [0,1]\times [0,\rho)$. 

\begin{lema}\label{lem:piecewise_Lie}
	Define a smooth $1$-form $\mathfrak{a}_{j}$ on $\d V\times [0,1]\times [0,\rho)$ piecewise, as
	$\frac{1}{h}((\Psi_{h}^{q})^{*}dw_{j}-dw_{j})$ for $h>0$ and as the Lie derivative $\cL_{\nu_{t}}(dw_j)$ for $h=0$; where $h$ is the projection to the last coordinate. Let $\mathfrak{a}_{j}^{\gamma}\de \mathfrak{a}_{j}|_{\Gamma}$. Then the form $h^{-1}\cdot \mathfrak{a}_{j}^{\gamma}$ extends to a continuous form on $\Gamma$.
\end{lema}
\begin{proof}
	Since $\mathfrak{a}_{j}^{\gamma}$ is a $\cC^{1}$-form on $\Gamma$, it is enough to prove that $\mathfrak{a}_{j}^{\gamma}$ 
	vanishes at $h=0$. Since the path $\gamma$ is admissible, the value $h=0$ is attained precisely on the slice $\d V\times \{(0,0)\}$. Thus $\mathfrak{a}_{j}^{\gamma}|_{{h=0}}=\mathfrak{a}_{j}|_{(h,q)=(0,0)}=0$ by Lemma \ref{lem:Q}\ref{item:Q-flow}, as needed.
\end{proof}

 Recall that $L_{h}\subseteq X_{h}^{\gamma}$ is the fiber of the $\omega_{h}^{\gamma}$-Lagrangian fibration \eqref{eq:Lagranian-at-positive-radius} over the fixed point $b\in E_{\cS}$. Thus in our notation, $L_{h}\times \{q(h)\}=\Psi_{h}^{q(h)}(L_{0}\times \{q(h)\})$. 

\begin{lema}\label{lem:Omega-factor}
	Let $\Omega_{j}\de dw_{j}-t\cdot \imath\, d\theta_{j}$ be a factor in formula \eqref{eq:Omega_new}. There is a continuous, complex valued $1$-form $\beta_{j}$ on $\Gamma$ such that for every small $h>0$ we have the equality of restrictions
	\begin{equation}\label{eq:Omega_factor}
		(\Psi_{h}^{q(h)})^{*}(\Omega_{j}|_{L_{h}})= h \cdot (h\cdot \beta_{j}- \imath\, d\theta_j)|_{L_0}.
	\end{equation} 
\end{lema}
\begin{proof}
	By definition of the Lagrangian torus fibration \eqref{eq:Lagrangian}, see Section \ref{sec:expanded_construction}, its fiber $L_0$ is a level set of $(w_{1},\dots,w_{n+1})$, so $(dw_{j})|_{L_{0}}=0$. Thus on $L_0$ we have the equality of restrictions
	\begin{equation*}
		\begin{split}
			h^{-1}\cdot (\Psi_{h}^{q(h)})^{*}\Omega_{j} & =
			h\cdot \left( h^{-1}\cdot \frac{(\Psi_{h}^{q(h)})^{*}dw_{j}-dw_{j}}{h}
			-
			\imath \cdot \frac{(\Psi_{h}^{q(h)})^{*}d\theta_{j}-d\theta_{j}}{h}\right)
			-
			\imath\, d\theta_{j} = \\ 
	& = h\cdot \left( h^{-1}\mathfrak{a}_j^{\gamma}-\imath\, \mathfrak{b}_j^{\gamma}\right)-\imath\, d\theta_{j}.
		\end{split}
	\end{equation*}
	where $\mathfrak{a}_{j}^{\gamma}$ is the form defined in Lemma \ref{lem:piecewise_Lie}; and $\mathfrak{b}_{j}^{\gamma}$ is defined analogously, with $\theta_{j}$ instead of $w_j$. Note that the forms $\mathfrak{a}_{j}^{\gamma}$ and $\mathfrak{b}_{j}^{\gamma}$ are restrictions of smooth forms to a $\cC^1$-submanifold $\Gamma$, so they are $\cC^1$. Moreover, Lemma \ref{lem:piecewise_Lie} shows that the form  $h^{-1}\mathfrak{a}_j$, which is $\cC^1$ on $\Gamma\cap \{h>0\}$, extends to a continuous form on $\Gamma$. Hence $\beta_{j}\de  h^{-1}\mathfrak{a}_j^{\gamma}-\imath\, \mathfrak{b}_j^{\gamma}$ extends to a continuous form on $\Gamma$, too, as needed.
\end{proof}

\begin{proof}[Proof of Proposition \ref{prop:special}]
	Let $\vol_{0}$ be a volume form on $L_0$ given by  $\pm\bigwedge_{i=1}^{n}d\theta_{i}$. Taking a wedge product in formula \eqref{eq:Omega_factor} and applying Lemma \ref{lem:volume-extends}\ref{item:vol_formulas}, we get the following equality on $L_0$:
	\begin{equation*}
		(\Psi_{h}^{q(h)})^{*}(\Omega_{\new}|_{L_h})= \pm \imath^{n}\cdot (c\circ \Psi_{h}^{q(h)}) \cdot h^{n}\cdot (\vol_{0}+h\cdot \beta),
	\end{equation*}
	for some continuous, complex-valued $n$-form $\beta$ on $L_0$. The sign $\pm$ comes from the definition of $\vol_0$ and from formula \eqref{eq:Omega_new}. In particular, it depends only on the face $\Delta_{I}$.
	
	Write $\beta=b\cdot  \vol_0$, for some continuous, complex-valued function $b$. Let $c_{h}=c\circ \Psi_{h}^{q(h)}$, so $c_{h}$ is a complex-valued function approaching a constant $c_0$ as $h\rightarrow 0$, see Lemma \ref{lem:volume-extends}\ref{item:c_0}. Now for small $h>0$, the phase map $\varpi$ from Definition \ref{def:phase} equals
	\begin{equation*}
		\varpi=\pm\imath^n \cdot \frac{c_h h^{n}\cdot (1+h\cdot b)}{|c_h h^{n} \cdot (1+h\cdot b)|}=\pm\imath^n \cdot  \frac{c_{h}}{|c_h|} \cdot \tilde{\varpi},
		\quad\mbox{where}\quad
		\tilde{\varpi}=\frac{1+h\cdot b}{|1+h\cdot b|}\xrightarrow{\ \small{h\rightarrow 0}\ }1,
	\end{equation*}
	so $\varpi$ extends to a continuous map $L\to \mathbb{S}^{1}$, whose value at $L_{0}$ equals $\pm\imath^n \cdot \frac{c_{0}}{|c_0|}$, as claimed.
\end{proof}

\begin{proof}[Proof of Corollary \ref{cor:volume_minimizing}]
	We follow closely the argument of \cite[\textsection III.1]{HL}, proving that the special Lagrangians of phase $\varpi_0$ are $\Re(\varpi_0 \Omega)$-calibrated, hence volume minimizing. The only difference in our setting is that some bounds hold only up to asymptotically small errors. 
	
	We work in $(X_{h}^{\gamma})^{\circ}_{I}$ for a fixed maximal essential face $\Delta_{I}$. Dividing the chosen holomorphic form $\Theta$ by the constant \eqref{eq:asymptotic_specialty}, we can assume that the value of the phase map $\varpi$ on $L_0$ is $1$. Lemma \ref{lem:c+} gives 
	\begin{equation}\label{eq:ca2}
		\tfrac{1}{\sqrt{c_+}}\cdot \vol^{g}_{\new}|_{L_h}=\Re \Omega_{\new}|_{L_h} + \beta,\quad
		\mbox{where } \beta\rightarrow 0\mbox{ as } h\rightarrow 0.
	\end{equation}
	Now, we argue as in \cite[Lemma 1.9]{HL}. Fix $h>0$, $x\in L_{h}$, and let $e_1,\dots, e_n$ be an orthonormal basis of $T_{x}L_{h}$. Let $H$ be another $n$-dimensional subspace of $T_{x}X_{h}^{\gamma}$, let  $\epsilon_{1},\dots, \epsilon_{n}$ be an orthonormal basis of $H$, and let $M$ be a linear automorphism of $T_{x}X_{h}^{\gamma}$ defined by $Me_i=\epsilon_i$, $MJe_i=J\epsilon_i$. Note that $M$ is $\C$-linear, so we can compute both the complex and real determinant of $M$, and they are related by $|\det_{\C}M|^2=\det_{\R}M$. Using Lemma \ref{lem:c+}, we compute 
	\begin{equation*}
		 (\Re \Omega_{\new}(\epsilon_1,\dots, \epsilon_n))^2\leq |\Omega_{\new}(\epsilon_1,\dots, \epsilon_n)|^2=|\det\nolimits_{\C}M|^2 \cdot |\Omega_{\new}(e_1,\dots, e_n)|^2
=
		  \det\nolimits_{\R} M\cdot c_+^{-1}\leq c_{+}^{-1} 
	\end{equation*}
	by the Hadamard inequality and the fact that the vectors $\epsilon_1,\dots,\epsilon_n$ have length one. Thus on $n$-dimensional subspaces of $T_{x}X_{h}^{\gamma}$ we have an inequality
	\begin{equation}\label{eq:calibration}
		\Re \Omega_{\new} \leq \tfrac{1}{\sqrt{c_+}}\cdot \vol^{g}_{\new}.
	\end{equation}
	On $T_{x}L_{h}$, the inequality \eqref{eq:calibration} becomes the \enquote{asymptotic} equality \eqref{eq:ca2}, cf.\ \cite[Theorem 1.10]{HL}. 
	
	Now fix $\Gamma\in [L_{h}]$. The form $\Omega$, hence $\Re \Omega$ is closed. Its rescaled version $\Re\Omega_{\new}$ is therefore fiberwise closed, so Stokes theorem gives $\int_{L_h} \Re\Omega_{\new}=\int_{\Gamma} \Re \Omega_{\new}$. We conclude that
	\begin{equation*}
		\int_{L_h} \tfrac{1}{\sqrt{c_+}} \cdot \vol^{g}_{\new}
		\overset{\mbox{\tiny{\eqref{eq:ca2}}}}{=}
		\int_{L_h} \Re \Omega_{\new}
		+ \int_{L_h} \beta
		=
		\int_{\Gamma} \Re \Omega_{\new}
		+ \int_{L_h} \beta
		\overset{\mbox{\tiny{\eqref{eq:calibration}}}}{\leq}
		\int_{\Gamma} \tfrac{1}{\sqrt{c_+}} \cdot \vol^{g}_{\new}
		+
		\int_{L_h} \beta.
	\end{equation*}
	The required convergence  \eqref{eq:volume_minimizing} follows since $\lim_{h\rightarrow 0}\beta=0$ and $\lim_{h\rightarrow 0} c_{+}$ is a positive constant.
\end{proof}

\begin{proof}[Proof of Theorem \ref{theo:CY_af}]
	Parts \ref{item:CY_Lagrangian_volume} and \ref{item:assflat} are proved in Propositions \ref{prop:volume} and  \ref{prop:MA_weak}. To see the first statement of \ref{item:CY_asymptotic specialty}, take for $\varpi \in \mathbb{S}^1$ the inverse of the constant \eqref{eq:asymptotic_specialty} and apply Proposition \ref{prop:special}: note that the path with $|z(h)|=h$ is  tame by Example \ref{ex:special}. The second statement of \ref{item:CY_asymptotic specialty} is proved in Corollary \ref{cor:volume_minimizing}. Part \ref{item:CY_Maslov} follows from \ref{item:CY_asymptotic specialty} since the Maslov class is locally constant and the base $E_{\cS}$ of the Lagrangian fibration $\ll_{h}$ is connected.
\end{proof}

\section{Dependence on the choices}
\label{sec:choices}

Within the construction of the fiberwise K\"ahler forms and Lagrangian fibrations in Theorems \ref{theo:CY}, \ref{theo:CY_af} and \ref{theo:general} we made the following choices. 
	\begin{parlist} 
		\item\label{choice:model} The snc model $f\colon X\to \D$ of the degeneration $f^{\circ}\colon X^{\circ}\to \D^{*}$, together with a K\"ahler form $\omega_{X}$. In the setting of Theorems \ref{theo:CY}, \ref{theo:CY_af}, when $f^{\circ}$ is a maximal Calabi--Yau degeneration, the model $f$ is required to satisfy Proposition \ref{prop:model}\ref{item:sk_P1}, but nevertheless is not unique. 
	 	\item\label{choice:rhat} Distance functions $\hat{r}_i$ on $X$, i.e., any functions satisfying the conditions of Lemma \ref{lem:rihat}.
	 	\item\label{choice:eps} The number $\epsilon\in (0,\epsilon_0]$, where $\epsilon_0>0$ is sufficiently small, depending on the choices \ref{choice:model}--\ref{choice:rhat}.
	 	\item\label{choice:Morse} For each $1$-dimensional open stratum $X_{I}^{\circ}$ of $X_0$, a singular foliation given by the level sets of a Morse function $h_{I}$. In the setting of Theorems \ref{theo:CY}, \ref{theo:CY_af}, we do have a canonical choice of a nonsingular foliation: indeed, we have $X_{I}^{\circ}\cong \C^{*}$ and we take $h_{I}$ to be the absolute value of a complex coordinate. 
	 	\item\label{choice:path} The path $\gamma$ in Theorems \ref{theo:CY}\ref{item:CY_GH}, \ref{theo:CY_af} and  \ref{theo:general}\ref{item:GH}.
	\end{parlist}
	We emphasize that the conclusions of Theorems \ref{theo:CY}, \ref{theo:CY_af} and \ref{theo:general} hold for any choice as above, and in general, there is no preferred one (except for choice \ref{choice:Morse} in the Calabi--Yau case, which is unique).
	
	The choice \ref{choice:model} of a model $X\to \D$ is the starting point in the construction of the A'Campo space. Choosing the \enquote{Fubini--Study} form $\omega_{X}$ is also crucial: it is used in the definition \eqref{eq:omega-s-intro} of the forms $\omega_{q}^{\epsilon}$ to guarantee that they are non-degenerate at positive radius, and in those directions at radius zero which collapse in the Gromov-Hausdorff limit. The choice \ref{choice:rhat} of distance functions is a technical device used to endow the A'Campo space $A$ with a smooth structure, see Remark \ref{rem:new_set}. Nonetheless, we recall that the construction of $A$ as the topological fibration over $\D_{\log}$ is classical \cite{A'Campo} and does not involve this choice. 
	Once the choices \ref{choice:model},\ref{choice:rhat} are made, we get a smooth structure on $A$ and family of fiberwise symplectic forms $\omega^{\epsilon}_{q}=\pi^{*}\omega_{X}+\epsilon \omega_{E}$ parametrized by $\epsilon\in (0,\epsilon_0]$. To proceed with the constructions, we need to choose one such $\epsilon$ (which indeed appears in some formulas, cf.\ Remark \ref{rem:MA_fails}). After choices \ref{choice:model}--\ref{choice:eps},  in the Calabi--Yau case, we get canonical $\omega_{q}^{\epsilon}$-Lagrangian torus fibrations; to construct them in the general case we need to make the choice \ref{choice:Morse} of Morse functions, see Remarks \ref{rem:canonical_choice}, \ref{rem:canonical_choice-positive-radius}. The final choice \ref{choice:path} of a path $\gamma(z,q)$ is made only to state the condition \enquote{as $|z|,q\rightarrow 0$} in a convenient way.

\bibliographystyle{amsalpha}
\bibliography{bibl}	
	
\end{document}